\pdfoutput=1
\RequirePackage{ifpdf}
\ifpdf 
\documentclass[pdftex]{sigma}
\else
\documentclass{sigma}
\fi

\numberwithin{equation}{section}

\newtheorem{thm}{Theorem}[section]
\newtheorem{cor}[thm]{Corollary}
\newtheorem{lem}[thm]{Lemma}
\newtheorem{prop}[thm]{Proposition}
\newtheorem{Conjecture}[thm]{Conjecture}
{\theoremstyle{definition}
\newtheorem{rem}[thm]{Remark}
\newtheorem{defn}[thm]{Definition}}

\newcommand{\hyp}[5]{\,\mbox{}_{#1}F_{#2}\!\left(
 \genfrac{}{}{0pt}{}{#3}{#4};#5\right)}

\makeatletter
\def\eqnarray{\stepcounter{equation}\let\@currentlabel=\theequation
\global\@eqnswtrue
\tabskip\@centering\let\\=\@eqncr
$$\halign to \displaywidth\bgroup\hfil\global\@eqcnt\z@
 $\displaystyle\tabskip\z@{##}$&\global\@eqcnt\@ne
 \hfil$\displaystyle{{}##{}}$\hfil
 &\global\@eqcnt\tw@ $\displaystyle{##}$\hfil
 \tabskip\@centering&\llap{##}\tabskip\z@\cr}

\def\endeqnarray{\@@eqncr\egroup
 \global\advance\c@equation\m@ne$$\global\@ignoretrue}

\def\@yeqncr{\@ifnextchar [{\@xeqncr}{\@xeqncr[5pt]}}
\makeatother

\newcommand{\bfX}{{\bf X}}
\newcommand{\bfN}{{\bf N}}
\newcommand{\bfx}{{\bf x}}
\newcommand{\bfxp}{{{\bf x}^\prime}}

\newcommand{\wbfx}{\widehat{\bf x}}
\newcommand{\wbfxp}{{\widehat{\bf x}^\prime}}
\newcommand{\N}{{\mathbf N}}

\newcommand{\R}{{\mathbf R}}
\newcommand{\Si}{{\mathbf S}}
\newcommand{\Hi}{{\mathbf H}}

\newcommand{\Z}{{\mathbf Z}}
\newcommand{\C}{{\mathbf C}}

\newcommand{\sP}{\mathsf{P}}
\newcommand{\sQ}{\mathsf{Q}}

\newcommand{\mch}{{\mathcal H}}
\newcommand{\mss}{{\mathcal S}}
\newcommand{\mas}{{\mathcal A}}

\newcommand{\msss}{{\mathsf S}}
\newcommand{\mssf}{{\mathfrak S}}
\newcommand{\mass}{{\mathsf A}}
\newcommand{\masf}{{\mathfrak A}}

\newcommand{\mcg}{{\mathcal G}}

\newcommand{\sech}{\hspace{2.22222222pt}\mathrm{sech}\hspace{1.666667pt}}

\usepackage{scalerel}
\let\svus_
\catcode`_=\active %
\def_{\ifmmode\expandafter\svus\expandafter\bgroup\expandafter\lowerit\else\svus\fi}
\def\lowerit#1{\ThisStyle{\raisebox{-2\LMpt}{$\SavedStyle#1$}}\egroup}
\makeatletter
 \AtBeginDocument{
 \check@mathfonts
 \fontdimen13\textfont2=3.5pt
 \fontdimen14\textfont2=3.5pt
 \fontdimen16\textfont2=2.5pt
 \fontdimen17\textfont2=2.5pt
 }
\makeatother

\begin{document}

\allowdisplaybreaks

\newcommand{\arXivNumber}{1803.07149}

\renewcommand{\thefootnote}{}

\renewcommand{\PaperNumber}{136}

\FirstPageHeading

\ShortArticleName{Fundamental Solutions and Gegenbauer Expansions of Helmholtz Operators}

\ArticleName{Fundamental Solutions and Gegenbauer Expansions\\ of Helmholtz Operators in Riemannian Spaces\\ of Constant Curvature\footnote{This paper is a~contribution to the Special Issue on Orthogonal Polynomials, Special Functions and Applications (OPSFA14). The full collection is available at \href{https://www.emis.de/journals/SIGMA/OPSFA2017.html}{https://www.emis.de/journals/SIGMA/OPSFA2017.html}}}

\Author{Howard S.~COHL~$^\dag$, Thinh H.~DANG~$^\ddag$ and T.M.~DUNSTER~$^\S$}

\AuthorNameForHeading{H.S.~Cohl, T.H.~Dang and T.M.~Dunster}

\Address{$^\dag$~Applied and Computational Mathematics Division,\\
\hphantom{$^\dag$}~National Institute of Standards and Technology, Mission Viejo, CA 92694, USA}
\EmailD{\href{mailto:howard.cohl@nist.gov}{howard.cohl@nist.gov}}
\URLaddressD{\url{http://www.nist.gov/itl/math/msg/howard-s-cohl.cfm}}

\Address{$^\ddag$~Department of Computer Science, George Washington University,\\
\hphantom{$^\ddag$}~Washington D.C.~20052, USA}
\EmailD{\href{mailto:danghungthinh@gmail.com}{danghungthinh@gmail.com}}

\Address{$^\S$~Department of Mathematics \& Statistics, San Diego State University,\\
\hphantom{$^\S$}~San Diego, CA 92182, USA}
\EmailD{\href{mailto:mdunster@mail.sdsu.edu}{mdunster@mail.sdsu.edu}}

\ArticleDates{Received March 01, 2018, in final form December 14, 2018; Published online December 31, 2018}

\Abstract{We perform global and local analysis of oscillatory and damped spherically symmetric fundamental solutions for Helmholtz operators $\big({-}\Delta\pm\beta^2\big)$ in $d$-dimensional, $R$-radius hyperbolic ${\mathbf H}_R^d$ and hyperspherical ${\mathbf S}_R^d$ geometry, which represent Riemannian manifolds with positive constant and negative constant sectional curvature respectively. In particular, we compute closed-form expressions for fundamental solutions of $\big({-}\Delta \pm \beta^2\big)$ on ${\mathbf H}_R^d$, $\big({-}\Delta+\beta^2\big)$ on~${\mathbf S}_R^d$, and present two candidate fundamental solutions for $\big({-}\Delta-\beta^2\big)$ on ${\mathbf S}_R^d$. Flat-space limits, with their corresponding asymptotic representations, are used to restrict proportionality constants for these fundamental solutions. In order to accomplish this, we summarize and derive new large degree asymptotics for associated Legendre and Ferrers functions of the first and second kind. Furthermore, we prove that our fundamental solutions on the hyperboloid are unique due to their decay at infinity. To derive Gegenbauer polynomial expansions of our fundamental solutions for Helmholtz operators on hyperspheres and hyperboloids, we derive a collection of infinite series addition theorems for Ferrers and associated Legendre functions which are generalizations and extensions of the addition theorem for Gegenbauer polynomials. Using these addition theorems, in geodesic polar coordinates for dimensions greater than or equal to three, we compute Gegenbauer polynomial expansions for these fundamental solutions, and azimuthal Fourier expansions in two-dimensions.}

\Keywords{hyperbolic geometry; hyperspherical geometry; fundamental solution; Helmholtz equation; Gegenbauer series; separation of variables; addition theorems; associated Legendre functions; Ferrers functions}

\Classification{31C12; 32Q45; 33C05; 33C45; 35A08; 35J05; 42A16}

\renewcommand{\thefootnote}{\arabic{footnote}}
\setcounter{footnote}{0}

\section{Introduction}\label{Introduction}

In this paper we derive associated Legendre and Ferrers function expressions for fundamental solutions of Helmholtz operators $\big({-}\Delta\pm\beta^2\big)$, $\beta^2>0$, in Riemannian spaces of constant curvature, namely in the $d$-dimensional $R$-radius hyperboloid $\Hi_R^d$ and hyperspherical $\Si_R^d$ models with negative and positive sectional curvatures respectively, where $R>0$. We also compute eigenfunction expansions for these fundamental solutions of Helmholtz operators in geodesic polar coordinates in these Riemannian manifolds. In particular, we derive Gegenbauer polynomial expansions for spherically symmetric fundamental solutions of Helmholtz operators on Riemannian manifolds of negative-constant and positive-constant sectional curvatures. Useful background material relevant for this paper can be found in \cite{Lee,PogWin,Thurston,Vilen}.

This paper is organized as follows. In Section~\ref{SpecialFunctions}, we introduce and give some useful pro\-per\-ties for the special functions and orthogonal polynomials that we will use in this paper. In particular, we summarize (and derive new) large degree asymptotics for associated Legendre and Ferrers functions of the first and second kind. We also derive a collection of infinite series addition theorems for Ferrers and associated Legendre functions (generalizations and extensions of the addition theorem for Gegenbauer polynomials) which will be used to compute Gegenbauer polynomial expansions for fundamental solutions of Helmholtz operator in spaces of constant curvature. In Section~\ref{Globalanalysisonthehyperboloid}, for the hyperboloid model of $d$-dimensional hyperbolic geometry and for hyperspherical geometry, we describe some of their global properties, such as their respective geodesic distance functions, geodesic polar coordinates, Helmholtz operators, and their corresponding radial harmonics. In Section~\ref{AGreensfunctioninthehyperboloidmodel}, we show how to compute radial harmonics in a geodesic polar coordinate system and derive fundamental solutions for $\big({-}\Delta \pm \beta^2\big)$ on~$\Hi_R^d$, $\big({-}\Delta+\beta^2\big)$ on~$\Si_R^d$, and study two candidate fundamental solutions for $\big({-}\Delta-\beta^2\big)$ on~$\Si_R^d$. In Section~\ref{Gegenbauerexpansioninhyperbolichypersphericalcoordinates}, for $d\ge 3$, we compute Gegenbauer polynomial expansions in geodesic polar coordinates for fundamental solutions of Helmholtz operators on the hyperboloid and hypersphere. We also compute azimuthal Fourier expansions for these fundamental solutions in two-dimensions. In Appendix~\ref{ProofoftheAdditionTheoremGegexpansionsFerrers}, the proof of an addition theorem is presented.

\section{Special functions, asymptotics, and notation}\label{SpecialFunctions}

Throughout this paper we rely on the following definitions. The set of natural numbers is given by $\N:=\{1,2,3,\ldots\}$, the set $\N_0:=\{0,1,2,\ldots\}=\N\cup\{0\}$, and the set $\Z:=\{0,\pm 1,\pm 2,\ldots\}.$ The set $\R$ represents the real numbers and the set $\C$ represents the complex numbers. For $a_1,a_2,a_3,\ldots\in\C$, if $i,j\in\Z$ and $j<i$ then $\sum\limits_{n=i}^{j}a_n:=0$ and $\prod\limits_{n=i}^ja_n:=1$. Note that we often adopt a common notation used for fundamental solution expansions, namely if one takes $a,a^\prime\in\R$, then
\begin{gather*}
a_\lessgtr:={\min \atop \max}\{a,a^\prime\}.
\end{gather*}

\subsection{The gamma function and factorials}\label{Thegammafunction}

The (Euler) gamma function $\Gamma\colon \C\setminus-\N_0\to\C$ (see \cite[Chapter~5]{NIST:DLMF}), which is ubiquitous in special function theory satisfies the recurrence formula $\Gamma(z+1)=z\Gamma(z)$, and is an important combinatoric function which generalizes the factorial function $\Gamma(n+1)=n!$, $n\in\N_0$. The gamma function is naturally defined through Euler's integral \cite[equation~(5.2.1)]{NIST:DLMF}. The following asymptotic approximation involving the ratio of gamma
functions will also be needed.

\begin{lem}\label{Gammaratiolemma}Let $a,b\in\C$. Then we have, as ${0<\tau \rightarrow \infty}$,
\begin{gather}
\frac{\Gamma (a\pm {\rm i}\tau)}{\Gamma (b\pm {\rm i}\tau)}={\rm e}^{\pm {\rm i}\pi (a-b) /2}\tau^{a-b}\left\{ 1+{\mathcal{O}\left(\frac{{1}}{{\tau}}\right)}\right\},\label{Gammaratio}
\end{gather}
where $\tau^{a-b}$ takes its principal value.
\end{lem}

\begin{proof}Let $\delta\in(0,\pi)$. From \cite[equation~(5.11.13)]{NIST:DLMF}
\begin{gather}
\Gamma (z+a) /\Gamma (z+b) =z^{a-b}\big\{ 1+{\mathcal{O}\big( z^{-1}\big)}\big\},\label{Gammaratioz}
\end{gather}
as $z\rightarrow \infty $ with $a$ and $b$ real or complex constants, provided $\vert \arg z\vert \leq \pi-\delta $ ($<\pi $). If one takes $z=\pm {\rm i}\tau $ with $\tau >0$ then the argument restriction implies $\arg
( \pm {\rm i}\tau ) =\pm \pi /2$, and~(\ref{Gammaratio}) follows.
\end{proof}

\begin{lem}\label{negGammaratiolemma} Let $a,b\in\C$, $z-a\not\in\Z$, $\delta\in(0,\pi)$. Then we have, as ${z \rightarrow \infty}$, provided
$\vert \arg z\vert \leq \pi-\delta$,
\begin{gather}
\frac{\Gamma(-z+a)}{\Gamma(-z+b)}=\frac{\sin(\pi(z-b))}{\sin(\pi(z-a))}z^{a-b}
\left\{1+\mathcal{O}\left(\frac{1}{z}\right)\right\},\label{negratiogamma}
\end{gather}
where $z^{a-b}$ takes its principal value.
\end{lem}
\begin{proof}
Using (\ref{Gammaratioz}), \cite[equation~(5.5.3)]{NIST:DLMF}, the result follows.
\end{proof}

The Pochhammer symbol (rising factorial) $(\cdot)_{n}\colon \C\to\C$ is defined by
\begin{gather*}
(z)_0:=1,\qquad (z)_n:=(z)(z+1)\cdots(z+n-1),
\end{gather*}
where $n\in\N$. Note that $(z)_{n}=\Gamma(z+{n})/\Gamma(z)$, $\Gamma(z-n)=(-1)^n\Gamma(z)/(-z+1)_n$, for all $z\in\C\setminus-\N_0$, $n\in\N_0$.

\subsection{The Gauss hypergeometric function}\label{Gausshypergeometricfunction}

The Gauss hypergeometric function ${}_2F_1\colon \C^2\times(\C\setminus-\N_0)\times \mathbf C\setminus[1,\infty)\to\mathbf C$ can be defined in terms of the following infinite series \cite[equation~(15.2.1)]{NIST:DLMF}
\begin{gather*}
\hyp21{a,b}{c}{z}:=\sum_{n=0}^\infty \frac{(a)_n(b)_n}{(c)_n}\frac{z^n}{n!},\qquad |z|<1,\label{Gauss2F1}
\end{gather*}
and elsewhere by analytic continuation. Certain orthogonal polynomials are special cases of the Gauss hypergeometric function, such as Chebyshev polynomials and Gegenbauer polynomials. The Chebyshev polynomial of the first kind $T_{n}\colon \C\to\C$ is defined as \cite[Section~5.7.2]{MOS}
\begin{gather*}
T_n(x):=\hyp21{-n,n}{{\tfrac{1}{2}}}{\frac{1-x}{2}}.\label{chebydef}
\end{gather*}
Note that $T_n(\cos\psi)=\cos(n\psi)$. The Gegenbauer polynomial $C_n^\mu\colon \C\to\C$, $n\in\N_0$, can be defined in terms of the Gauss hypergeometric function as
\begin{gather*}
C_n^\mu(x):=\frac{(2\mu)_n}{n!}\hyp21{-n,2\mu+n}{\mu+{\tfrac{1}{2}}}{\frac{1-x}{2}}.
\end{gather*}
Note that for $\mu\in(-{\tfrac{1}{2}},\infty)\setminus\{0\}$, the Gegenbauer polynomial is orthogonal on $(-1,1)$ with a~positive weight (see \cite[equation~(9.8.20)]{Koekoeketal}) and identically zero when $\mu=0$ for all $n\ge 1$. One also has the following $\mu\to0$ limit holding \cite[equation~(6.4.13)]{AAR}
\begin{gather}
\lim_{\mu\to 0}\frac{{n}+\mu}{\mu}C_n^\mu(x)=\epsilon_{n} T_{n}(x), \label{mutozeroChebyGeg}
\end{gather}
where $\epsilon_n:=2-\delta_{n,0}$. The Legendre polynomial $P_n\colon \C\to\C$ is defined by
\begin{gather}
P_n(x)=P_n^0(x)={\mathsf P}_n(x)={\mathsf P}_n^0(x)=C_n^\frac{1}{2}(x)\label{legendrepoly}
\end{gather}
(see Section \ref{Ferrersfunctions} below).

\subsection{The associated Legendre functions of the first and second kind}\label{assPQsection}

We also frequently use associated Legendre functions of the first and second kind $P_{\nu}^{\mu},Q_{\nu}^{\mu}\colon {\mathbf{C}}\setminus (-\infty,1]\rightarrow {\mathbf{C}}$ respectively. Gauss hypergeometric
representations of these functions are \cite[equation~(14.3.6)]{NIST:DLMF}
\begin{gather*}
P_{\nu}^{\mu}(z):=\frac{1}{\Gamma (1-\mu )}\left( \frac{z+1}{z-1}\right)^{\frac{\mu}{2}}\hyp21{-\nu,\nu+1}{1-\mu}{\frac{1-z}{2}},
\end{gather*}
where $|1-z|<2$ and \cite[equation~(14.3.7)]{NIST:DLMF}
\begin{gather*}
Q_{\nu}^{\mu}(z):=\frac{\sqrt{\pi}{\rm e}^{{\rm i}\pi\mu}\Gamma (\nu+\mu +1)\big(z^{2}-1\big)^{\frac{\mu}{2}}}{2^{\nu+1}\Gamma \big(\nu+\tfrac{3}{2}\big)z^{\nu+\mu+1}}%
\hyp21{\frac{\nu+\mu+1}{2},\frac{\nu+\mu+2}{2}}{\nu+\tfrac{3}{2}}{\frac{1}{z^{2}}}, \label{q2invz2}
 \end{gather*}
where $|z|>1$. In regard to the above definition of the associated Legendre function of the first kind, note that for $\mu\in\N$, $\Gamma(1-\mu)$ is undefined, however the function $\tfrac{1}{\Gamma(c)}\hyp21{a,b}{c}{z}$ is an entire function for all $a,b,c\in\C$, $|z|<1$. See \cite[equation~(15.2.2)]{NIST:DLMF} and the discussion in \cite[Section~14.3(ii)]{NIST:DLMF} for $\mu\in\Z$.

The {\it associated Legendre conical} functions are given by $\mu\in\C$, $\tau\ge 0$, $z\in\C\setminus(-\infty,1]$, $P_{-\frac12\pm {\rm i}\tau}^{\pm\mu}(z)$, $Q_{-\frac12\pm {\rm i}\tau}^{\pm\mu}(z)$. Associated Legendre functions satisfy various transformations such as the Whipple transformation, namely \cite[equations~(14.9.16) and~(14.9.17)]{NIST:DLMF}. They also satisfy various connection relations for (conical) associated Legendre functions of the first kind such as \cite[equation~(14.9.11)]{NIST:DLMF}
\begin{gather}
P_{-{\frac{1}{2}}-{\rm i}\tau}^{\pm\mu}(z) =P_{-{\frac{1}{2}}+{\rm i}\tau}^{\pm\mu}(z)\label{pmdegreeconicalleg}
\end{gather}
for all $\mu$, $\tau$, $z$; and \cite[equations~(14.3.10) and (14.9.15)]{NIST:DLMF}
\begin{gather}
P_{-{\frac{1}{2}}+{\rm i}\tau}^{\mu}(z) =\frac{\Gamma \big(\frac12+\mu+{\rm i}\tau \big)}{\Gamma \big( {{\frac{1}{2}}}-\mu+{\rm i}\tau\big)}P_{-{\frac{1}{2}}+{\rm i}\tau}^{-\mu}(z)+\frac{2}{\pi}{\rm e}^{-{\rm i}\pi\mu}\sin\left( \mu \pi \right) Q_{-{\frac{1}{2}}+{\rm i}\tau}^{\mu}(z);\label{Dun.it}
\end{gather}
connection relations for associated Legendre functions of the second kind, such as \cite[equations~(14.3.10) and (14.9.14)]{NIST:DLMF}
\begin{gather}
Q_{-{\frac{1}{2}}+{\rm i}\tau}^{-\mu}(z)={\rm e}^{-2{\rm i}\pi\mu}\frac{\Gamma \big( \frac{1}{2}-\mu+{\rm i}\tau \big)}{\Gamma \big( \frac12+\mu+{\rm i}\tau \big)}Q_{-{\frac{1}{2}}+{\rm i}\tau}^{\mu}(z).\label{Dun.eq20a}
\end{gather}

For the associated Legendre conical function of the first kind, one has the following important behavior (cf.~\cite[pp.~171, 173]{Olver})
\begin{gather}
P_{-{\frac{1}{2}}+{\rm i}\tau}^{-\mu}(z) \sim \frac{1}{\Gamma ({\mu+1})}\left( {\frac{z-1}{2}}\right)^{\frac{\mu}{2}},\label{Dun.eq10}
\end{gather}
as $z\rightarrow 1^{+}$. For the associated Legendre function of the second kind and its conical form we note the following important behaviors
\begin{gather}
Q_\nu^\mu(z)\sim\frac{\sqrt{\pi} \Gamma(\nu+\mu+1){\rm e}^{{\rm i}\pi\mu}}{\Gamma\big(\nu+\frac32\big)(2z)^{\nu+1}},\label{Qnumuzinftysim}
\end{gather}
as $z\to\infty$, $\nu+\mu+1\ne-1,-2,-3,\ldots$, and
\begin{gather}
Q_{-{\frac{1}{2}}\pm {\rm i}\tau}^{\mu}(z) \sim \frac{\sqrt{\pi}{\rm e}^{{\rm i}\pi\mu}\Gamma \left( {\mu+{\frac{1}{2}}\pm {\rm i}\tau}\right)}{\Gamma \left( {1\pm {\rm i}\tau}\right) (2z)^{{\frac{1}{2}}\pm {\rm i}\tau}},
\label{Dun.eq11}
\end{gather}
as $z\rightarrow \infty$.

\subsubsection{Large degree asymptotics of associated Legendre functions}

The associated Legendre functions of the first and second kind satisfy the following asymptotic representations. Let $\mu \geq 0$ be bounded, $r\in(0,\infty)$. Then uniformly we have, as $0< \nu \rightarrow \infty$
\cite[equations~(14.15.13) and (14.15.14)]{NIST:DLMF},
\begin{gather}
P_{\nu}^{-\mu}(\cosh r) =\frac{1}{\nu^{\mu}} \sqrt{\frac{r}{\sinh r}}I_{\mu} \big(\big( \nu+{\tfrac{1}{2}}\big) r\big) \left\{ 1+\mathcal{O}\left( \frac{1}{\nu}\right) \right\},\label{iPuaa}\\
Q_{\nu}^{\mu}(\cosh r) ={\rm e}^{{\rm i}\pi\mu}\nu^{\mu}\sqrt{\frac{r}{\sinh r}}K_{\mu}\big( \big( \nu+{\tfrac{1}{2}}\big) r\big) \left\{ 1+\mathcal{O}\left( \frac{1}{\nu}\right) \right\}.\label{iQuaa}
\end{gather}
Corresponding results for $P_{\nu}^{\mu}(\cosh r) $, $Q_{\nu}^{-\mu}(\cosh r) $, $Q_{-\nu-1}^{\pm \mu}(\cosh r) $ follow readily from these formulas and the connection formulas in \cite[Section~14.9(iii)]{NIST:DLMF}. For extensions to asymptotic expansions and complex argument see \cite[Chapter~12, Section~13]{Olver}.

We shall give analogous results for associated Legendre conical functions. In order to do so, we shall require so-called envelope functions for our approximants, the Bessel and Hankel functions. The reason for this is they have zeros in our domain of validity, unlike the modified Bessel functions used in (\ref{iPuaa}), (\ref{iQuaa}). For the Bessel function $J_{\mu}(x)$ ($\mu\geq 0$) of real argument~$x$, an envelope function $\operatorname{env}J_{\mu}(x) $ is given in \cite[equations~(2.8.32)--(2.8.34)]{NIST:DLMF}. This function is continuous and nonvanishing, and has the properties that $J_{\mu}(x) /\operatorname{env}J_{\mu}(x)={\mathcal{O}}(1)$
uniformly for $x\in(0,\infty)$; $\lim\limits_{x\rightarrow0+}J_{\mu}(x) /\operatorname{env}J_{\mu}(x) >0$; and $J_{\mu}(x) /\operatorname{env}J_{\mu}(x) $ does not approach zero as $x\rightarrow\infty$. Thus $\operatorname{env}J_{\mu}(x) $ has the same order of magnitude as $J_{\mu}(x) $ uniformly for $x\in(0,\infty)$, but does not have any zeros (except at $x=0$ when $\mu>0$).

We would like to define an envelope function for $J_{\mu}(z) $ where $z$ is complex with $\vert \arg z\vert \leq\frac{1}{2}\pi$ (again assuming $\mu\geq0$), and having similar properties as described above. The desired function is given as follows.

\begin{lem}\label{lemmaA1} Let $\mu\geq0$. The function defined by
\begin{gather}
\operatorname{env}J_{\mu}(z) :=\big\{ \vert J_{\mu} (z) \vert^{2}+ \vert J_{\mu+1}(z) \vert ^{2}\big\}^{\frac12}, \label{envJ}
\end{gather}
has the following properties in the half-plane $ \vert \arg z \vert \leq\frac{1}{2}\pi$:
\begin{enumerate}\itemsep=0pt
\item[$(i)$] it has no zeros, except at $z=0$ when $\mu>0$;
\item[$(ii)$] $J_{\mu}(z) /\operatorname{env}J_{\mu}(z) ={\mathcal{O}}(1) $ uniformly for all non-zero $z$;
\item[$(iii)$] $\lim\limits_{z\rightarrow0} \vert J_{\mu}(z) \vert /\operatorname{env}J_{\mu}(z) >0$; and
\item[$(iv)$] $J_{\mu}(z) /\operatorname{env}J_{\mu}(z) $ does not approach zero as $z\rightarrow\infty$.
\end{enumerate}
\end{lem}

\begin{proof}From \cite[equation~(10.6.2)]{NIST:DLMF}, $J_{\mu}^{\prime}(z) =(\mu/z) J_{\mu}(z)-J_{\mu+1}(z) $, and since $J_{\mu}(z)$, $J_{\mu}^{\prime}(z) $ do not have common zeros (except at $z=0$ when $\mu>0$), the same is true of $J_{\mu}(z) $, $J_{\mu+1}(z) $. Thus from~(\ref{envJ}) we see (i) holds. Finally (ii)--(iv) follow from~(\ref{envJ}) and the following limiting behaviors which hold for $\mu\geq0$, $\vert \arg z \vert \leq\frac{1}{2}\pi$, namely as $z\rightarrow 0$,
\begin{gather}
J_{\mu}(z) \sim \frac{\big( \tfrac{1}{2}z\big)^{\mu}}{\Gamma (\mu+1 )},\label{Jat0}
\end{gather}
and as $z\rightarrow\infty$,
\begin{gather*}
J_{\mu}(z) =\left( \frac{2}{\pi z}\right)^{\frac12}\left[ \cos\left( z-\frac{\pi\mu}{2}-\frac{\pi}{4}\right)+o(1)
\exp \{ \vert {\Im}\, z \vert \} \right],\label{Jinfinity}
\end{gather*}
which completes the proof.
\end{proof}

We also require envelope functions for the Hankel functions $H_{\mu}^{( 1)}(z)$, $H_{\mu}^{( 2)}(z) $. For positive argument these are not required, since these functions have no real zeros. However they do have complex zeros (see \cite[Section~10.21(ix)]{NIST:DLMF}), and for this reason we shall construct them for complex~$z$ such that $ \vert \arg z \vert \leq\frac{1}{2}\pi$. To this end we note the limiting behaviors which hold for $ \vert \arg z \vert \leq \frac{1}{2}\pi$, namely as $z\rightarrow 0$,
\begin{gather}
H_{0}^{(1)}(z) \sim -H_{0}^{(2)}(z) \sim \frac{2{\rm i}}{\pi}\log z, \nonumber\\
H_{\mu}^{(1)}(z) \sim -H_{\mu }^{( 2)}(z) \sim -\frac{{\rm i}\Gamma ( \mu)}{\pi}\big( \tfrac{1}{2}z\big) ^{-\mu},\qquad \mu >0,\label{Hat0}
\end{gather}
and as $z\to\infty$,
\begin{gather}
H_{\mu}^{(1,2)}(z) \sim\left( \frac{2}{\pi z}\right)^{\frac12}\exp\left\{ \pm {\rm i}\left( z-\frac{1}{2}\pi\mu-\frac{1}{4}\pi\right) \right\}.\label{Hinfinity}
\end{gather}

Now for $H_{\mu}^{(1)}(z) $ the natural choice analogous to (\ref{envJ}) would appear to be $\operatorname{env}H_{\mu}^{(1)}(z) =\big\{\bigl| H_{\mu}^{(1)}(z) \bigr|^{2}+\bigl| H_{\mu+1}^{(1)}(z) \bigr|^{2}\big\}^{\frac12}$, but on referring to (\ref{Hat0}) we see that this function would have the incompatible behavior $\operatorname{env}H_{\mu}^{(1)}(z) \sim \big\vert H_{\mu+1}^{(1)}(z) \big\vert \sim \tfrac{1}{\pi} \Gamma(\mu+1)
\big( \tfrac{1}{2} \vert z \vert \big)^{-\mu-1}$ as $z\rightarrow 0$ for $\mu \geq 0$. Thus we would have the undesirable limit $\big\vert H_{\mu}^{(1) }(z) \big\vert /\operatorname{env}H_{\mu}^{(1)} (z) \rightarrow 0 $ as $z\rightarrow 0$. A similar problem occurs for such an envelope function for $H_{\mu}^{( 2)}(z) $. With this in mind, in the following lemma we make the required modification in our definition of the Hankel envelope functions. Note that we omit the proof, which is similar to the proof of Lemma~\ref{lemmaA1}.

\begin{lem}Let $\mu\ge 0$. The functions defined by
\begin{gather}
\operatorname{env}H_{\mu}^{( 1,2)}(z) :=\left\{ \bigl\vert H_{\mu}^{(1,2)}(z) \bigr\vert^{2} +\min\bigl\{1,|z|^{2}\bigr\}\bigl\vert H_{\mu+1}^{\left(1,2\right)}(z) \bigr\vert^{2}\right\}^{\frac12}, \label{envH}
\end{gather}
has the following properties in the half-plane $ \vert \arg z \vert \leq\frac{1}{2}\pi$:
\begin{enumerate}\itemsep=0pt
\item[$(i)$] $\operatorname{env}H_{\mu}^{(1,2)}(z) $ has no zeros;
\item[$(ii)$] $H_{\mu}^{(1,2)}(z) /\operatorname{env}H_{\mu }^{(1,2)}(z) ={\mathcal{O}}(1) $ uniformly for all non-zero $z$;
\item[$(iii)$] $\lim\limits_{z\rightarrow0}\bigl\vert H_{\mu}^{(1,2)}(z) \bigr\vert /\operatorname{env}H_{\mu}^{(1,2)}(z) >0$; and
\item[$(iv)$] $H_{\mu}^{(1,2)}(z) /\operatorname{env}H_{\mu }^{(1,2)}(z) $ does not approach zero as $z\rightarrow\infty$.
\end{enumerate}
\end{lem}

We now state large degree asymptotic approximations for the associated Legendre conical functions.

\begin{thm}\label{thm:1} Let $\mu \geq 0$, $z=\cosh r\in\C$, such that $ \vert \arg z \vert \leq \pi /2$ $($equivalently $r$ lying in the semi-infinite strip ${\Re}\,r\geq 0$, $ \vert {\Im}\,r \vert \leq \pi /2$;
see {\rm \cite[Chapter~12, Section~13.1]{Olver})}. Then uniformly we have, as $0<\tau \rightarrow \infty$,
\begin{gather}
P_{-{\frac{1}{2}}\pm {\rm i}\tau}^{-\mu}(\cosh r) =\frac{1}{\tau^{\mu}} \sqrt{\frac{r}{\sinh r}} \left\{ J_{\mu}( \tau r) +\mathcal{O}\left( \frac{{1}}{{\tau}}\right)\operatorname{env}J_{\mu}( \tau r) \right\},\label{Dun.eq16}\\
 P_{-{\frac{1}{2}}\pm {\rm i}\tau}^{\mu}(\cosh r) =\tau^{\mu} \sqrt{\frac{r}{\sinh r}}
\left[ \cos( \pi \mu) \left\{ J_{\mu}(\tau r) {+\mathcal{O}\left( \frac{{1}}{{\tau}}\right)}\operatorname{env}J_{\mu}( \tau r) \right\}\right.\label{Dun.eq16a}\\
 \left. \hphantom{P_{-{\frac{1}{2}}\pm {\rm i}\tau}^{\mu}(\cosh r) =}{} -\sin(\pi \mu) \left\{ Y_{\mu}(\tau r)+{\mathcal{O}\left( \frac{{1}}{{\tau}}\right)}\operatorname{env}{H_{\mu}^{\left(
1\right)}(\tau r)+\mathcal{O}\left( \frac{{1}}{{\tau}} \right)}\operatorname{env}{H_{\mu}^{(2)}( \tau r )}\right\} \right], \nonumber\\
 Q_{-{\frac{1}{2}}+{\rm i}\tau}^{\pm \mu}(\cosh r)=-\frac{{\rm i}\pi}{2} {\rm e}^{(-1\pm 3) {\rm i}\pi\mu /2} \tau^{\pm \mu}\sqrt{\frac{r}{\sinh r}} \left\{ {H_{\mu}^{( 2)}( \tau r)+\mathcal{O}\left( \frac{{1}}{{\tau}}\right)}\operatorname{env}{H_{\mu}^{(2)}(\tau r)}\right\},\!\!\!\label{Dun.eq21}\\
 Q_{-{\frac{1}{2}}-{\rm i}\tau}^{\pm \mu}(\cosh r) =\frac{{\rm i}\pi}{2} {\rm e}^{( 1\pm 1) {\rm i}\pi\mu /2} \tau^{\pm \mu} \sqrt{\frac{r}{\sinh r}}\left\{ {H_{\mu}^{(1)}( \tau r)+\mathcal{O}\left( \frac{{1}}{{\tau}}\right)}\operatorname{env}{H_{\mu}^{(1)}(\tau r)}\right\},\label{Dun.eq22}
\end{gather}
where in \eqref{Dun.eq16a}--\eqref{Dun.eq22}, $r\neq 0$.
\end{thm}

\begin{proof} The functions $w_{1,2}\colon \C\setminus(-\infty,1]\to\C$ defined by
\begin{gather*}
w_{1}(z) :=\big( {z^{2}-1}\big)^{\frac12}P_{-{\frac{1}{2}}+{\rm i}\tau }^{\mu}(z), \qquad w_{2}(z) :=\big( {z^{2}-1}\big)^{\frac12}Q_{-{\frac{1}{2}}+{\rm i}\tau }^{\mu}(z),
\end{gather*}
satisfy the following ordinary differential equation
\begin{gather*}
\frac{{\rm d}^{2}w}{{\rm d}z^{2}}=\left\{{-\frac{\tau^{2}}{z^{2}-1}+\frac{\mu^{2}-1}{\left({z^{2}-1}\right)^{2}}-\frac{1}{4\big(z^{2}-1\big)}}\right\} w. \label{Dun.eq3}
\end{gather*}%
The first term dominates as $\tau \rightarrow \infty $, and is characterized by having a simple pole at $z=1$. Following \cite[Chapter~12]{Olver} we simplify it by introducing a new independent variable~$\zeta(z)$ defined by (see also \cite[equation~(4.37.19)]{NIST:DLMF})
\begin{gather}
\zeta^{\frac12}:=\int_{1}^{z}{\big(t^{2}-1\big)^{-\frac12}\mathrm{d}t}=\cosh^{-1}z.\label{Dun.eq4}
\end{gather}%
Then with the accompanying change of dependent variable $w=\zeta^{-\frac14}\big(z^{2}-1\big)^\frac14 W$, we arrive at the transformed differential equation
\begin{gather*}
\frac{{\rm d}^{2}W}{{\rm d}\zeta^{2}}=\left( {-\frac{\tau^{2}}{4\zeta}+\frac{\mu^{2}-1}{4\zeta^{2}}+\frac{\psi(\zeta)}{\zeta}}\right) W,\label{Dun.eq6}
\end{gather*}
where
\begin{gather*}
\psi(\zeta) :=\frac{4\mu^{2}-1}{16}\left( {\frac{1}{z^{2}-1}-\frac{1}{\zeta}}\right), \label{Dun.eq7}
\end{gather*}
which is analytic at $z=1$ ($\zeta =0)$.

We apply Olver's theorem \cite[Chapter~12, Theorem~9.1]{Olver}, with $u=\tau $ in the solutions of that theorem (with a slight modification to go from Bessel to modified Bessel functions as described below), and we obtain asymptotic solutions
\begin{gather}
W_{1}(\tau,\zeta) =\zeta^\frac12\left\{ J_{\mu}\big( \tau\zeta^\frac12\big) {+\mathcal{O}\big( \tau^{-1}\big)}\operatorname{env}J_{\mu}{\big(\tau \zeta^\frac12\big)}\right\}, \label{Dun.eq8}\\
W^{(1,2)}(\tau,\zeta) =\zeta^\frac12\left\{ {H_{\mu}^{(1,2)}\big(\tau \zeta^\frac12\big)+\mathcal{O}\big(\tau^{-1}\big)}\operatorname{env}{H_{\mu}^{(1,2)}\big(\tau \zeta^\frac12\big)}\right\}. \label{Dun.eq9}
\end{gather}
We remark that in Olver's theorem the approximants used are the modified Bessel func\-tions $I_{\mu}$,~$K_{\mu}$. In obtaining the Bessel formulas~(\ref{Dun.eq8}),~(\ref{Dun.eq9}) we have used the well-known identities
$I_{\mu}( {\rm i}z) \propto J_{\mu}(z)$, $K_{\mu }( \mp {\rm i}{z} ) \propto H_{\mu}^{(1,2)}(z)$.

On account of the monotonicity requirements that determine the domains of validity of the ${\mathcal{O}}$ terms in (\ref{Dun.eq8}), (\ref{Dun.eq9}) \cite[Chapter~12, Section~9.1]{Olver}, it is straightforward to show that these order terms are uniformly valid in\ an unbounded ${\zeta}$ domain which corresponds to the right half-plane $\vert \arg z\vert \leq \frac{1}{2}\pi $; see \cite[Chapter~12, Section~13.1]{Olver}.

Next we match the associated Legendre conical functions with the asymptotic solutions. To do so we note the fundamental asymptotic behaviors using (\ref{Dun.eq10}), (\ref{Dun.eq11}). From the behavior of Bessel functions at $0$ and $\infty $ (see (\ref{Jat0}), (\ref{Hinfinity})), we have as $\zeta \rightarrow 0$
\begin{gather}
W_{1}(\tau,\zeta) \sim \zeta^{\frac{{\mu+1}} {2}}\big( {{\tfrac{1}{2}}\tau}\big)^{\mu} / \Gamma ( \mu+1), \label{Dun.eq12}
\end{gather}%
and as $\zeta \rightarrow \infty $,
\begin{gather}
W^{(1,2)}(\tau,\zeta) \sim \left( {\frac{2}{\pi \tau}}\right)^{\frac12}\zeta^{\frac14}\exp \big\{ {\pm {\rm i}\tau \zeta^\frac12\mp {\tfrac{1}{2}}{\rm i}\pi\mu \mp {\tfrac{{1}}{4}\pi {\rm i}}}\big\}. \label{Dun.eq13}
\end{gather}
By matching solutions recessive at $z=1$ ($\zeta =0)$ we therefore deduce that
\begin{gather}
P_{-{\frac{1}{2}}+{\rm i}\tau}^{-\mu}(z) =\lambda_{1}(\mu ,\tau) \big( \zeta \big( z^{2}-1\big)\big) ^{-\frac14}W_{1}(\tau,\zeta), \label{Dun.eq14}
\end{gather}
for some constant $\lambda_{1} ( \mu,\tau) $. From (\ref{Dun.eq10}), (\ref{Dun.eq12}) we find that $\zeta \sim 2 ( z-1) $ as $z\rightarrow 1$, and therefore from (\ref{Dun.eq14})
\begin{gather}
\lambda_{1}(\mu,\tau) =\lim\limits_{z\rightarrow 1}\frac{\big(z^{2}-1\big)^{\frac14}P_{-{\frac{1}{2}}+{\rm i}\tau}^{-\mu}(z)}{\zeta^\frac14 J_{\mu}\bigl( \tau \zeta^{\frac12}\bigr)}=\frac{1}{\tau^{\mu}}. \label{Dun.eq15}
\end{gather}%
Now from (\ref{Dun.eq4}), $\zeta =r^{2}$ for $z=\cosh r$, and therefore from (\ref{Dun.eq8}), (\ref{Dun.eq14}), (\ref{Dun.eq15}) we arrive at~(\ref{Dun.eq16}). Noting also~(\ref{pmdegreeconicalleg}), which then completes the proof of~(\ref{Dun.eq16}), (\ref{Dun.eq16a}).

Next from (\ref{Dun.eq4}) it is readily verified that $\zeta^\frac12\sim \log ( 2z) $ as $z\rightarrow \infty $. Hence by matching solutions with the same (unique) behavior at $z=\infty $ ($\zeta =\infty )$, we assert that
there exists a~constant $\lambda^{(2)} (\mu,\tau) $ such that
\begin{gather*}
Q_{-{\frac{1}{2}}+{\rm i}\tau}^{\mu}(z) =\lambda^{(2) }(\mu,\tau) \big( \zeta \big( z^{2}-1\big)\big)^{-\frac14}W^{(2)}(\tau,\zeta). \label{Dun.eq18}
\end{gather*}
This time the constant is found by the following limit
\begin{gather*}
\lambda^{(2)}(\mu,\tau) =\lim\limits_{z\rightarrow \infty}\frac{\big( z^{2}-1\big)^{\frac14}Q_{-{\frac{1}{2}}+{\rm i}\tau}^{\mu}(z)}{\zeta^\frac14 H_{\mu}^{(2)}\bigl( \tau \zeta^\frac12\bigr)}.\label{Dun.eq19}
\end{gather*}
From (\ref{Dun.eq11}), (\ref{Dun.eq13}) we arrive at
\begin{gather}
\lambda^{(2)}(\mu,\tau) = \tfrac12 \pi {\rm e}^{{\rm i}\pi(\frac{\mu}{2}-\frac14)} \tau^{\frac12}\frac{\Gamma \big( \mu+{\frac{1}{2}}+{\rm i}\tau\big)}{\Gamma ( 1+{\rm i}\tau )}. \label{Dun.eq20}
\end{gather}
Then, using (\ref{Gammaratio}) to simplify (\ref{Dun.eq20}) we establish (\ref{Dun.eq21}) for the upper signs. For the lower signs we simply use (\ref{Dun.eq20a}) together with (\ref{Gammaratio}), (\ref{Dun.eq22}) follows similarly. Finally, (\ref{Dun.eq16a}) comes from~(\ref{Gammaratio}), (\ref{Dun.it}), (\ref{Dun.eq16}), (\ref{Dun.eq21}), which completes the proof.
\end{proof}

We remark that approximations for associated Legendre conical functions as $0<\tau \rightarrow \infty $ with $\vert \mu\vert $ large are given in~\cite{Dunster1} ($\mu $ real) and~\cite{Dunster2} ($\mu $ imaginary), but these
are more complicated than those given by (\ref{Dun.eq16})--(\ref{Dun.eq22}), as well as in Theorem~\ref{FerrersConicalThm} below.

\subsubsection{Associated Legendre function addition theorems}\label{alfat}

In \cite[Proposition~5.1]{CohlKalII}, an addition theorem is derived for the associated Legendre function of the second kind $Q_\mu^\mu(\cosh\rho)/\sinh^\mu\rho$, where
\begin{gather}
\cosh\rho=\cosh r\cosh r'-\sinh r\sinh r'\cos\gamma.\label{coshrho}
\end{gather}
We now present generalizations of that addition theorem for
\begin{gather*} P_\nu^\mu(\cosh\rho)/\sinh^\mu\rho \qquad \text{and} \qquad Q_\nu^\mu(\cosh\rho)/\sinh^\mu\rho.\end{gather*}

\begin{thm}Let $\mu\in\C$ such that ${\Re}\,\mu>-{\tfrac{1}{2}}$, $\mu\ne 0$, $\nu\in\C\setminus-\N$, $r,r'\in (0,\infty)$, $r\ne r'$, $\gamma\in\R$ and~\eqref{coshrho}. Then
\begin{gather}
\frac{1}{\sinh^\mu\rho}P_\nu^\mu(\cosh\rho)\nonumber\\
\qquad{} =\frac{2^\mu\Gamma(\mu)}{(\sinh r\sinh r^\prime)^\mu} \sum_{n=0}^\infty (-1)^n(n+\mu)P_\nu^{-(\mu+n)}(\cosh r_<)P_\nu^{\mu+n}(\cosh r_>)C_n^\mu(\cos\gamma),\label{conjP}\\
\frac{1}{\sinh^\mu\rho}Q_\nu^\mu(\cosh\rho)\nonumber\\
\qquad{}=\frac{2^\mu\Gamma(\mu)}{(\sinh r\sinh r^\prime)^\mu} \sum_{n=0}^\infty (-1)^n(n+\mu)P_\nu^{-(\mu+n)}(\cosh r_<)Q_\nu^{\mu+n}(\cosh r_>)C_n^\mu(\cos\gamma).\label{conj}
\end{gather}
For $\mu=0$, see Corollary~{\rm \ref{muzeroasslegexpan}}.
\end{thm}

\begin{proof}The proof for the expansion of the associated Legendre function of the second kind is identical to that given in \cite[Proposition~5.1]{CohlKalII}, except one does not specialize to $\nu=\mu$. For the associated Legendre function of the first kind, instead start with \cite[equation~(8.3)]{DurandFishSim}.
\end{proof}

\begin{cor}\label{muzeroasslegexpan}Let $\nu\in\C\setminus-\N$, $r,r'\in (0,\infty)$, $r\ne r'$, $\gamma\in\R$ and~\eqref{coshrho}. Then
\begin{gather}
P_\nu(\cosh\rho)=\sum_{n=0}^\infty \epsilon_n(-1)^nP_\nu^{-n}(\cosh r_<)P_\nu^n(\cosh r_>)T_n(\cos\gamma),\nonumber\\
Q_\nu(\cosh\rho)=\sum_{n=0}^\infty \epsilon_n(-1)^nP_\nu^{-n}(\cosh r_<)Q_\nu^n(\cosh r_>)T_n(\cos\gamma).\label{deq2addnthm}
\end{gather}
\end{cor}
\begin{proof}Taking the limit as $\mu\to0$ in (\ref{conj}) and using (\ref{mutozeroChebyGeg}) produces the result.
\end{proof}

\begin{cor}\label{Assocaddncoshexp} Let $\nu\in\C\setminus-\N$, $r,r'\in (0,\infty)$, $r\ne r'$, $\gamma\in\R$ and~\eqref{coshrho}. Then
\begin{gather*}
\frac{\cosh\big(\big(\nu+\frac12\big)\rho\big)}{\sinh\rho}=\frac{\pi}{2\sqrt{\sinh r\sinh r'}} \\
\hphantom{\frac{\cosh\big(\big(\nu+\frac12\big)\rho\big)}{\sinh\rho}=}{}\times \sum_{n=0}^{\infty}(-1)^n (2n+1) P_\nu^{-n-\frac12}(\cosh r_<) P_\nu^{n+\frac12}(\cosh r_>) P_n(\cos\gamma),\\
\frac{\exp\big({-}\big(\nu+\frac12\big)\rho\big)}{\sinh\rho}= \frac{-{\rm i}}{\sqrt{\sinh r\sinh r'}} \\
\hphantom{\frac{\exp\big({-}\big(\nu+\frac12\big)\rho\big)}{\sinh\rho}=}{}\times \sum_{n=0}^{\infty}(-1)^n (2n+1) P_\nu^{-n-\frac12}(\cosh r_<) Q_\nu^{n+\frac12}(\cosh r_>) P_n(\cos\gamma).
\end{gather*}
\end{cor}
\begin{proof}Let $\mu=\tfrac12$ in (\ref{conjP}), (\ref{conj}), and using \cite[equations~(14.5.15) and (14.5.17)]{NIST:DLMF}, (\ref{legendrepoly}), completes the proof.
\end{proof}

\subsection{The Ferrers functions of the first and second kind}\label{Ferrersfunctions}

The Ferrers functions of the first and second kind respectively $\mathsf{P}_{\nu}^{\mu},\mathsf{Q}_{\nu}^{\mu}\colon (-1,1)\rightarrow{\mathbf{C}}$ provide a~definition for the associated Legendre functions on $(-1,1)$. The Ferrers functions have the following Gauss hypergeometric representations \cite[equations~(14.3.1) and (14.3.12)]{NIST:DLMF}
\begin{gather}
\mathsf{P}_{\nu}^{\mu}(x):=\frac{1}{\Gamma(1-\mu)}\left( \frac{1+x}{1-x}\right)^{\frac{\mu}{2}}\hyp21{-\nu,\nu+1}{1-\mu}{\frac{1-x}{2}},\label{FerrersPdefn}\\
{\mathsf{Q}}_{\nu}^{\mu}(x):= \sqrt{\pi}\,2^{\mu}\frac{\Gamma\big( \frac{\nu+\mu+2}{2}\big)}{\Gamma\big( \frac {\nu-\mu+1}{2}\big)}\frac{x\cos\big[ \frac{\pi}{2}(\nu+\mu)\big]}{\big(1-x^{2}\big)^{\frac{\mu}{2}}}\hyp21{\frac{1-\nu-\mu}{2},\frac{\nu-\mu+2}{2}}{\tfrac{3}{2}}{x^{2}}\nonumber\\
\hphantom{{\mathsf{Q}}_{\nu}^{\mu}(x):=}{} -\sqrt{\pi} 2^{\mu-1}\frac{\Gamma\big( \frac{\nu+\mu+1}{2}\big)}{\Gamma\big( \frac{\nu-\mu+2}{2}\big)}\frac{\sin\big[\frac{\pi}{2}(\nu+\mu)\big]}{\big(1-x^{2}\big)^{\frac{\mu}{2}}}\hyp21{\frac {-\nu-\mu}{2},\frac{\nu-\mu+1}{2}}{{\tfrac{1}{2}}}{x^{2}},\label{defFerrersQhypergeo}
\end{gather}
where for ${\mathsf{Q}}_{\nu}^{\mu}$ we must impose the constraint $\nu+\mu\in{\mathbf{C}}\setminus-{\mathbf{N}}$, except for the anomalous cases $\nu=-\frac32,-\frac52,\ldots$. In regard to the above definition of the Ferrers function of the first kind (see also the definition of $P_\nu^\mu(z)$ given in Section~\ref{assPQsection}), note that $\tfrac{1}{\Gamma(c)}\hyp21{a,b}{c}{z}$ is an entire function for all $a,b,c\in\C$, $|z|<1$. See \cite[equation~(15.2.2)]{NIST:DLMF} and the discussion in \cite[Section~14.3(ii)]{NIST:DLMF} for $\mu\in\Z$.

Note that the Ferrers function of the first kind is related to the Gegenbauer function of the first kind (generalization of the Gegenbauer polynomial for non-integer degrees) with general~$\lambda$,~$\mu$, using \cite[equation~(14.3.21)]{NIST:DLMF}, namely
\begin{gather}
C_\lambda^\mu(\cos\gamma) =\frac{\sqrt{\pi}\,\Gamma(2\mu+\lambda)}{2^{\mu-\frac{1}{2}}\Gamma(\lambda+1)\,\Gamma(\mu)}\left( {\sin \gamma}\right) ^{\frac12-\mu}\mathsf{P}_{\mu +\lambda-\frac{1}{2}}^{\frac{1}{2}-\mu}(\cos\gamma) .\label{Dun2eq12}
\end{gather}

The {\it Ferrers conical} functions are given by $\mu\in\C$, $\tau\ge 0$, $x\in(-1,1)$, ${\mathsf P}_{-\frac12\pm {\rm i}\tau}^{\pm\mu}(x)$, ${\mathsf Q}_{-\frac12\pm {\rm i}\tau}^{\pm\mu}(x)$, which satisfy the following connection formulas, namely \cite[equation~(14.9.11)]{NIST:DLMF}
\begin{gather}
\mathsf{P}_{-{\frac{1}{2}}-{\rm i}\tau}^{\pm\mu}(x) =\mathsf{P}_{-{\frac{1}{2}}+{\rm i}\tau}^{\pm\mu}(x)\label{Dun.st}
\end{gather}
and from \cite[equation~(14.9.10)]{NIST:DLMF}, we have
\begin{gather}
\mathsf{P}_{-\frac{1}{2}+{\rm i}\tau}^{-\mu}( -x) =\sin(\pi {( {\rm i}\tau-\mu)})\mathsf{P}_{-\frac{1}{2}+{\rm i}\tau}^{-\mu}(x)+\frac{2}{\pi} \cos ( \pi ( {\rm i}\tau-\mu ) ) \mathsf{Q}_{-\frac{1}{2}+{\rm i}\tau}^{-\mu}(x). \label{Dun.eq34}
\end{gather}
For the Ferrers functions of the first and second kind, one has the connection formulas, namely \cite[Section~14.9(ii)]{NIST:DLMF}, and after replacing $\mu$ by $-\mu$ one obtains
\begin{gather}
{\sf P}_\nu^{-\mu}(x)=\frac{\Gamma(\nu-\mu+1)}{\Gamma(\nu+\mu+1)}\left[\cos(\pi\mu){\sf P}_\nu^\mu(x)-\frac{2}{\pi}\sin(\pi\mu){\sf Q}_\nu^\mu(x)\right],\label{Pnummuconnection}\\
{\sf Q}_\nu^{-\mu}(x)=\frac{\Gamma(\nu-\mu+1)}{\Gamma(\nu+\mu+1)}\left[\cos(\pi\mu){\sf Q}_\nu^\mu(x)+\frac{\pi}{2}\sin(\pi\mu){\sf P}_\nu^\mu(x)\right].\label{Qnummuconnection}
\end{gather}
We will also need \cite[equations~(14.9.8) and (14.9.10)]{NIST:DLMF},
\begin{gather}
{\sf P}_\nu^{-\mu}(-x)=\cos(\pi(\nu-\mu)){\sf P}_\nu^{-\mu}(x)-\frac{2}{\pi}\sin(\pi(\nu-\mu)){\sf Q}_\nu^{-\mu}(x),\label{Pnummumxconnection}\\
{\sf Q}_\nu^{-\mu}(-x)=-\cos(\pi(\nu-\mu)){\sf Q}_\nu^{-\mu}(x)-\frac{\pi}{2}\sin(\pi(\nu-\mu)){\sf P}_\nu^{-\mu}(x).\label{Qnummumxconnection}
\end{gather}

\subsubsection[Some properties of functions associated with ${\sf P}_\nu^{-\mu}(-x)$]{Some properties of functions associated with $\boldsymbol{{\sf P}_\nu^{-\mu}(-x)}$}\label{SomepropertiesFerrers}

\begin{prop}Let $x\in(-1,1)$, $\nu\in\C$, $\Re\mu>0$. Then
\begin{gather}
{\sf P}_\nu^{-\mu}(-x)=\frac{1}{\Gamma(\mu+1)}\left(\frac{1+x}{1-x}\right)^{\frac{\mu}{2}}\hyp21{-\nu,\nu+1}{1+\mu}{\frac{1+x}{2}}.\label{hypergeomFerrersPnummumx}
\end{gather}
\end{prop}
\begin{proof}Use (\ref{FerrersPdefn}) and map $x\mapsto-x$, $\mu\mapsto-\mu$.
\end{proof}

\begin{cor}Let $x\in(-1,1)$, $\nu\in\C$, $\Re\mu>0$. Then
\begin{gather}
\lim_{x\to-1^{+}}{\sf P}_\nu^{-\mu}(-x)=0,\label{firstPnummumxmonelimit}\\
\lim_{x\to-1^{+}}\frac{1}{\big(1-x^2\big)^{\frac{\mu}{2}}}{\sf P}_\nu^{-\mu}(-x)=\frac{1}{2^\mu\Gamma(\mu+1)}.\label{secondPnummumxmonelimit}
\end{gather}
\end{cor}
\begin{proof} As $x\to-1^{+}$ (\ref{hypergeomFerrersPnummumx}) clearly vanishes, establishing (\ref{firstPnummumxmonelimit}). Dividing (\ref{hypergeomFerrersPnummumx}) by $\big(1-x^2\big)^{\frac{\mu}{2}}$ and estimating as $x\to-1^{+}$ establishes~(\ref{secondPnummumxmonelimit}).
\end{proof}

\begin{cor}Let $x\in(-1,1)$, $\nu\in\C$. Then
\begin{gather}
\left.\frac{{\rm d}}{{\rm d}x}{\sf P}_\nu^{-\mu}(-x)\right|_{x\to-1^{+}}=0,\qquad \Re\mu>2,\nonumber\\
 \left.\frac{{\rm d}}{{\rm d}x}\left[\frac{1}{\big(1-x^2\big)^{\frac{\mu}{2}}}{\sf P}_\nu^{-\mu}(-x)\right]\right|_{x\to-1^{+}}=-\frac{\nu(\nu+1)}{2^{\mu+1}\Gamma(\mu)}.\label{derivopppole}
\end{gather}
\end{cor}

We next define an odd Ferrers function ${\mathsf f}_\nu^\mu\colon (-1,1)\to\C$, by the difference
\begin{gather}
{\mathsf f}_\nu^\mu(x):={\sf P}_\nu^{\mu}(-x)-{\sf P}_\nu^{\mu}(x).\label{FPodddefn}
\end{gather}
\begin{rem}\label{oddFerrersrem} This difference is directly proportional to the difference of Ferrers functions of the second kind through \cite[p.~170]{MOS}
\begin{gather}
{\sf Q}_\nu^{\mu}(-x)-{\sf Q}_\nu^{\mu}(x) =\frac{\pi}{2}\cot\big(\tfrac{\pi}{2}(\nu+\mu)\big)\big({\sf P}_\nu^{\mu}(-x)-{\sf P}_\nu^{\mu}(x)\big)=\frac{\pi}{2}\cot\big(\tfrac{\pi}{2}(\nu+\mu)\big) {\mathsf f}_\nu^\mu(x).\label{qdifference}
\end{gather}
We note that (\ref{FPodddefn}) vanishes if ${\sf P}_\nu^{\mu}(x)$ is even. This occurs if $\nu+\mu$ is a non-negative even integer, or if $\nu-\mu$ is a negative odd integer (see \cite[Chapter~5, p.~187]{Olver}). Similarly, the Ferrers function of the second kind difference (\ref{qdifference}) vanishes if $\nu+\mu$ is a~positive odd integer, or if $\nu-\mu$ is a~negative odd integer. For our purposes~(\ref{FPodddefn}) will suffice, since in our applications this function will not vanish identically.
\end{rem}

\begin{lem}\label{lemma2.13}Let $\nu\in\C$, $\Re\big(\alpha\pm\frac{\mu}{2}\big)>0$. Then
\begin{gather}
\int_{-1}^{1} \big(1-x^2\big)^{\alpha-1} {\sf P}_{\nu}^{-\mu}(\pm x) \mathrm{d}x = \frac{\pi\Gamma\big(\alpha+\frac{\mu}{2}\big)\Gamma\big(\alpha-\frac{\mu}{2}\big)}{2^\mu\Gamma\big(\alpha+\frac{\nu+1}{2}\big)\Gamma\big(\alpha-\frac{\nu}{2}\big)
\Gamma\big(\frac{\nu+\mu+2}{2}\big)\Gamma\big(\frac{\mu-\nu+1}{2}\big)}.\label{eq2.50}
\end{gather}
\end{lem}
\begin{proof}Let $I$ denote the integral in (\ref{eq2.50}), with positive sign in the argument of the Ferrers function. Insert the hypergeometric representation
of the Ferrers function of the first kind~(\ref{FerrersPdefn}), and make the substitution $\xi=(1-x)/2$. This con\-verts~$I$ to a~Mellin transform, namely after re-mapping $\xi\mapsto x$,
\begin{gather*}
I=\frac{2^{2\alpha-1}}{\Gamma(1+\mu)} \int_0^1 x^{\alpha+\frac{\mu}{2}-1} (1-x)^{\alpha-\frac{\mu}{2}-1}\hyp21{-\nu,\nu+1}{1+\mu}{x}\mathrm{d}x.
\end{gather*}
Using \cite[equation~(7.512.5)]{Grad}, we can evaluate the integral in terms of a~${}_3F_2(1)$, namely
\begin{gather*}
I=\frac{2^{2\alpha-1}\Gamma\big(\alpha+\frac{\mu}{2}\big)\Gamma\big(\alpha-\frac{\mu}{2}\big)} {\Gamma(1+\mu)\Gamma(2\alpha)}\hyp32{-\nu,\nu+1,\alpha+\frac{\mu}{2}}{1+\mu,2\alpha}{1}.
\end{gather*}
The ${}_3F_2(1)$ in question can be evaluated using Whipple's sum \cite[equation~(16.4.7)]{NIST:DLMF} for $\Re\big(\alpha\pm \frac{\mu}{2}\big)>0$. The invariance of $I$ with argument of the Ferrers function of the first
kind as $-x$ follows by direct substitution. Hence the proof of the lemma is demonstrated.
\end{proof}

\begin{cor}Let $\Re{\mu}>-1$. Then
\begin{gather}
\int_{-1}^{1} \big(1-x^2\big)^{\frac{\mu}{2}} {\sf P}_{\nu}^{-\mu}(\pm x) \mathrm{d}x =\frac{2^{\mu+1}\Gamma(\mu+1)}{\Gamma(\mu-\nu+1)\Gamma(\nu+\mu+2)}.\label{Mellininquestion}
\end{gather}
\end{cor}
\begin{proof}Using Lemma~\ref{lemma2.13} and properties of gamma function completes the proof.
\end{proof}

\subsubsection{Large degree asymptotics of Ferrers functions}

\begin{thm}\label{Ferrersasympt} Let $\mu \geq 0$ bounded, $\delta\in(0,\pi)$, $\theta \in \left( 0,\pi-\delta \right]$, and the envelope functions for~$J_\mu$,~$Y_\mu$ being the real-valued ones given by {\rm \cite[equations~(2.8.32)--(2.8.34)]{NIST:DLMF}}. Then uniformly we have, as $\nu \rightarrow \infty$,
\begin{gather}
 \mathsf{P}_{\nu}^{-\mu}(\cos\theta )=\frac{1}{\nu^{\mu}} \sqrt{ \frac{\theta}{\sin\theta}}\left\{ J_{\mu}\big(\big(\nu+\tfrac{1}{2}\big)\theta\big)+{\mathcal{O}\left( \frac{{1}}{{\nu}}\right)}\operatorname{env}J_{\mu}\big(\big(\nu+\tfrac{1}{2}\big)\theta\big) \right\},\label{Fuaa}\\
\mathsf{P}_{\nu}^{\mu}(\cos\theta )=\nu^{\mu}\sqrt{ \frac{\theta}{\sin\theta}}\left[ \cos(\pi \mu)\left\{ J_{\mu} \big(\big(\nu+\tfrac{1}{2}\big)\theta\big)+{\mathcal{O}\left( \frac{{1}}{{\nu}}\right)}\operatorname{env}J_{\mu}\big(\big(\nu+\tfrac{1}{2}\big)\theta\big) \right\} \right. \nonumber \\
\left.\hphantom{\mathsf{P}_{\nu}^{\mu}(\cos\theta )=}{} -\sin(\pi \mu) \left\{ Y_{\mu}\big(\big(\nu+\tfrac{1}{2}\big)\theta\big)+{\mathcal{O}\left( \frac{{1}}{{\nu}}\right)}\operatorname{env}Y_{\mu} \big(\big(\nu+\tfrac{1}{2}\big)\theta\big) \right\} \right],\label{FerrersPp}\\
\mathsf{Q}_{\nu}^{-\mu}(\cos\theta )=-\frac{\pi}{2\nu^{\mu}}\sqrt{ \frac{\theta}{\sin\theta}}\left\{ Y_{\mu}\big(\big(\nu+\tfrac{1}{2}\big)\theta\big)+{\mathcal{O}\left( \frac{{1}}{{\nu}}\right)}\operatorname{env}Y_{\mu}\big(\big(\nu+\tfrac{1}{2}\big)\theta\big)\right\}, \label{Fuab}\\
\mathsf{Q}_{\nu}^{\mu}(\cos\theta )=-\frac{\pi \nu^{\mu}}{2}\sqrt{ \frac{\theta}{\sin\theta}}\left[ \sin\left( \pi\mu \right) \left\{ J_{\mu} \big(\big(\nu+\tfrac{1}{2}\big)\theta\big)+{\mathcal{O}\left( \frac{{1}}{{\nu}}\right)}\operatorname{env}J_{\mu}\big(\big(\nu+\tfrac{1}{2}\big)\theta\big)\right\} \right. \nonumber\\
\left.\hphantom{\mathsf{Q}_{\nu}^{\mu}(\cos\theta )=}{} +\cos(\pi \mu) \left\{ Y_{\mu}\big(\big(\nu+\tfrac{1}{2}\big)\theta\big)+{\mathcal{O}\left( \frac{{1}}{{\nu}}\right)}\operatorname{env}Y_{\mu} \big(\big(\nu+\tfrac{1}{2}\big)\theta\big)
\right\} \right],\label{FerrersQp}\\
{\sf P}_{\nu}^{-\mu}(- \cos\theta)=\frac{1}{\nu^{\mu}}\sqrt{ \frac{\theta}{\sin\theta}}\left[\cos(\pi(\nu-\mu))\left\{ J_\mu\big(\big(\nu+\tfrac{1}{2}\big)\theta\big)+{\mathcal O}\left(\frac{1}{\nu}\right)
\operatorname{env}J_{\mu} \big(\big(\nu+\tfrac{1}{2}\big)\theta\big)\right\}\right. \nonumber\\
\left. \hphantom{{\sf P}_{\nu}^{-\mu}(- \cos\theta)=}{} +\sin(\pi(\nu-\mu))\left\{Y_\mu \big(\big(\nu+\tfrac{1}{2}\big)\theta\big)+{\mathcal O}\left(\frac{1}{\nu}\right)\operatorname{env}Y_{\mu}\big(\big(\nu+\tfrac{1}{2}\big)\theta\big)\right\} \right],\label{FerrersPmmx}\\
{\sf Q}_{\nu}^{-\mu}(-\cos\theta)=-\frac{\pi}{2\nu^{\mu}}\sqrt{ \frac{\theta}{\sin\theta}}\left[\sin(\pi(\nu-\mu))\left\{ J_\mu \big(\big(\nu+\tfrac{1}{2}\big)\theta\big)+{\mathcal O}\left(\frac{1}{\nu}\right)
\operatorname{env}J_{\mu}\big(\big(\nu+\tfrac{1}{2}\big)\theta\big)\right\}\right.\nonumber\\
\left.\hphantom{{\sf Q}_{\nu}^{-\mu}(-\cos\theta)=}{} -\cos(\pi(\nu-\mu))\left\{ Y_\mu \big(\big(\nu+\tfrac{1}{2}\big)\theta\big)+{\mathcal O}\left(\frac{1}{\nu}\right)\operatorname{env}Y_{\mu} \big(\big(\nu+\tfrac{1}{2}\big)\theta\big)\right\}\right].\label{FerrersQmmx}
\end{gather}
\end{thm}

\begin{proof} Using \cite[equations~(14.15.11) and (14.15.12)]{NIST:DLMF}, the uniform asymptotic approximations follow from simplifying the ratio of gamma functions in (\ref{Pnummuconnection}), (\ref{Qnummuconnection}), via (\ref{Gammaratioz}), and setting $x=\cos\theta $. The asymptotic approximations (\ref{FerrersPmmx}), (\ref{FerrersQmmx}) follow using (\ref{Qnummuconnection}), (\ref{Pnummumxconnection}), (\ref{Fuaa}), (\ref{Fuab}).
\end{proof}

\begin{cor}\label{cor:2.14}Let $\mu \geq 0$ be bounded. Then for the odd Ferrers function \eqref{FPodddefn}, we have as $0<\nu \rightarrow\infty $,
\begin{gather}
 {\mathsf f}_{\nu}^{-\mu}(\cos \theta )=\frac{1}{\nu ^{\mu}}\sqrt{\frac{\theta}{\sin \theta}}\left[ ( \cos (\pi (\nu -\mu ))-1)\left\{ J_{\mu}\big(\big(\nu +\tfrac{1}{2}\big)\theta\big) +{\mathcal{O}}
\left( \frac{1}{\nu}\right) \operatorname{env}J_{\mu}\big(\big(\nu +\tfrac{1}{2}\big)\theta\big) \right\} \right. \nonumber\\
\left.\hphantom{{\mathsf f}_{\nu}^{-\mu}(\cos \theta )=}{} +\sin (\pi (\nu -\mu ))\left\{ Y_{\mu}\big(\big(\nu +\tfrac{1}{2}\big)\theta\big) +{\mathcal{O}}\left( \frac{1}{\nu}\right) \operatorname{env}Y_{\mu}
\big(\big(\nu +\tfrac{1}{2}\big)\theta\big) \right\} \right],\label{FPodd1}
\end{gather}
uniformly for $\theta \in \big(0,\tfrac{1}{2}\pi\big] $, and
\begin{gather}
 {\mathsf f}_{\nu}^{-\mu}(\cos \theta )=\frac{1}{\nu ^{\mu}}\sqrt{\frac{\pi -\theta}{\sin \theta}}\bigg[ (1- \cos (\pi (\nu -\mu ))) \nonumber\\
\hphantom{{\mathsf f}_{\nu}^{-\mu}(\cos \theta )=}{} \times \left\{ J_{\mu} \big( \big(\nu + \tfrac{1}{2}\big) ( \pi -\theta) \big) +{\mathcal{O}}\left( \frac{1}{\nu}\right) \operatorname{env}J_{\mu} \big(\big(\nu +\tfrac{1}{2}\big)(\pi -\theta)\big)\right\}\label{FPodd2} \\
 \left. \hphantom{{\mathsf f}_{\nu}^{-\mu}(\cos \theta )=}{} -\sin (\pi (\nu -\mu ))\left\{ Y_{\mu} \big(\big(\nu +\tfrac{1}{2}\big)(\pi -\theta)\big) +{\mathcal{O}}\left( \frac{1}{\nu}\right) \operatorname{env}Y_{\mu} \big( \big(\nu +\tfrac{1}{2}\big) ( \pi -\theta) \big) \right\} \right] ,\nonumber
\end{gather}
uniformly for $\theta \in \big[ \frac{1}{2}\pi ,\pi \big) $.
\end{cor}
\begin{proof} For (\ref{FPodd1}), combine~(\ref{Fuaa}), (\ref{FerrersPmmx}). For~(\ref{FPodd2}) replace $\theta $ by $\pi -\theta $ in~(\ref{FPodd1}) and observe that ${\mathsf f}_{\nu}^{-\mu}(\cos ( \pi -\theta) )={\mathsf f}_{\nu}^{-\mu}(- \cos \theta )=-{\mathsf f}_{\nu}^{-\mu}(\cos\theta )$.
\end{proof}

Similar approximations for the Ferrers conical functions read as follows.

\begin{thm}\label{FerrersConicalThm} Let $\mu \geq 0$ bounded, $\delta\in(0,\pi)$, $\theta \in ( 0,\pi-\delta]$. Then uniformly, we have, as $0<\tau \rightarrow \infty$,
\begin{gather}
\mathsf{P}_{-{\frac{1}{2}}\pm {\rm i}\tau}^{-\mu}(\cos\theta )=\frac{1}{\tau^{\mu}}\sqrt{\frac{\theta}{\sin\theta}}I_{\mu}(\tau \theta )\left\{ 1+{\mathcal{O}\left( \frac{{1}}{{\tau}}\right)}\right\},\label{Fcuab}\\
\mathsf{P}_{-{\frac{1}{2}}\pm {\rm i}\tau}^{\mu}(\cos\theta )=\tau^{\mu}\sqrt{\frac{\theta}{\sin\theta}}\left[ I_{\mu}(\tau\theta )\left\{ 1+{\mathcal{O}\left( \frac{{1}}{{\tau}}\right)}\right\}\right. \nonumber\\
\left.\hphantom{\mathsf{P}_{-{\frac{1}{2}}\pm {\rm i}\tau}^{\mu}(\cos\theta )=}{}
+\frac{2}{\pi}\sin( \pi \mu) K_{\mu}(\tau \theta)\left\{ 1+{\mathcal{O}\left( \frac{{1}}{{\tau}}\right)}\right\} \right],\label{Fcuab2}\\
\mathsf{Q}_{-{\frac{1}{2}}\pm {\rm i}\tau}^{-\mu}(\cos\theta )=\frac{1}{\tau^{\mu}}\sqrt{\frac{\theta}{\sin\theta}}\left[{\rm e}^{\mp {\rm i}\pi\mu} K_{\mu}(\tau \theta )\left\{ 1+{\mathcal{O}\left( \frac{{1}}{{\tau}}\right)}\right\} \right. \nonumber\\
\left. \hphantom{\mathsf{Q}_{-{\frac{1}{2}}\pm {\rm i}\tau}^{-\mu}(\cos\theta )=}{} \mp {{\frac{{\rm i}\pi}{2}}}I_{\mu}(\tau \theta )\left\{ 1+{\mathcal{O}\left( \frac{{1}}{{\tau}}\right)}\right\} \right], \label{Fcuaa}\\
\mathsf{Q}_{-{\frac{1}{2}}\pm {\rm i}\tau}^{\mu}(\cos\theta )=\tau^{\mu}\sqrt{\frac{\theta}{\sin\theta}}\left[ \cos(\pi \mu)K_{\mu}(\tau \theta )\mp {{\frac{{\rm i}\pi}{2}}}I_{\mu
}(\tau \theta )\right] \left\{ 1+{\mathcal{O}\left( \frac{{1}}{{\tau}}\right)}\right\}, \label{Fcuac} \\
\mathsf{P}_{-{\frac{1}{2}}\pm {\rm i}\tau}^{-\mu}(-\cos\theta )=\frac{{\rm e}^{\pi\tau}}{\pi\tau^{\mu}}\sqrt{\frac{\theta}{\sin\theta}}K_{\mu}(\tau \theta )\left\{ 1+{\mathcal{O}\left( \frac{{1}}{{\tau}}\right)}\right\}, \label{FcPmmumx}\\
\mathsf{Q}_{-{\frac{1}{2}}\pm {\rm i}\tau}^{-\mu}(-\cos\theta )=\mp\frac{{\rm i}{\rm e}^{\pi\tau}}{2\tau^{\mu}} \sqrt{\frac{\theta}{\sin\theta}}K_{\mu}(\tau \theta )\left\{ 1+{\mathcal{O}\left( \frac{{1}}{{\tau}}
\right)}\right\}. \label{FcQmmumx}
\end{gather}
\end{thm}
\begin{proof}Under the hypotheses of the theorem, the approximation (\ref{Fcuab}), along with
\begin{gather}
\hat{\mathsf{Q}}_{-\frac{1}{2}+{\rm i}\tau}^{-\mu}( {\cos\theta}) =\frac{1}{\tau^{\mu}}\sqrt{ {\frac{\theta}{\sin\theta}}}K_{\mu}( {\tau \theta}) \left( 1+{\mathcal{O}\left( \frac{{1}}{{\tau}}\right)}\right), \label{Dun.eq30}
\end{gather}
where $\hat{\mathsf{Q}}_{-\frac{1}{2}+{\rm i}\tau}^{-\mu}(x) $ is defined by \cite[equation~(14.20.2)]{NIST:DLMF}, follow directly from \cite[equa\-tions (14.20.17) and~(14.20.18)]{NIST:DLMF}. Now from \cite[equation~(14.20.3)]{NIST:DLMF} we have the connection formula
\begin{gather}
\hat{\mathsf{Q}}_{-\frac{1}{2}+{\rm i}\tau}^{-\mu}(x)=A( {\mu,\tau}) \mathsf{P}_{-\frac{1}{2}+{\rm i}\tau}^{-\mu}(x)+B(\mu,\tau) \mathsf{P}_{-\frac{1}{2}+{\rm i}\tau}^{-\mu}( {-x}), \label{Dun.eq31}
\end{gather}%
where $A,B\colon [0,\infty)^2\to\R$, are defined by
\begin{gather*}
A(\mu,\tau)
:=\frac{\pi {\rm e}^{-\pi \tau}\sin ( {\mu\pi}) \sinh ( {\tau \pi})}{2\big( {\cosh^{2}( {\tau \pi})-\sin^{2}( {\mu\pi})}\big)}, \qquad B(\mu,\tau):=\frac{\pi \big( {{\rm e}^{-\pi \tau}\cos^{2}( {\mu\pi})+\sinh ( {\tau \pi})}\big)}{2\big( {\cosh^{2}( {\tau \pi})-\sin^{2}( {\mu\pi} )}\big)}.
\end{gather*}
Hence from (\ref{Dun.eq34}), (\ref{Dun.eq31}) we obtain
\begin{gather}
 \mathsf{Q}_{-\frac{1}{2}+{\rm i}\tau}^{-\mu}(x) =G( {\mu,\tau}) \mathsf{P}_{-\frac{1}{2}+{\rm i}\tau}^{-\mu}(x)+H(\mu,\tau) \hat{\mathsf{Q}}_{-\frac{1}{2}+{\rm i}\tau}^{-\mu}(x), \label{Dun.eq37}
\end{gather}
where
\begin{gather*}
G:=-\frac{A+BC}{BD},\qquad H:=\frac{1}{BD},\qquad C:=\sin(\pi({\rm i}\tau-\mu)),\qquad D:=\frac{2}{\pi}\cos (\pi({\rm i}\tau-\mu) ).\label{Dun.eq23}
\end{gather*}
Then from solving (\ref{Dun.eq37}) for $\hat{\mathsf{Q}}_{-\frac{1}{2}+{\rm i}\tau}^{-\mu}(x) $, and simplifying and approximating $G(\mu,\tau)$, $H(\mu,\tau)$ for large ${\tau}$, one finds after some calculation that
\begin{gather}
\mathsf{Q}_{-\frac{1}{2}+{\rm i}\tau}^{-\mu}(x) ={\rm e}^{-{\rm i}\pi\mu} \hat{\mathsf{Q}}_{-\frac{1}{2}+{\rm i}\tau}^{-\mu}(x) \big( 1+\mathcal{O}\big(\tau^{-1}\big)\big)-{{\frac{{\rm i}\pi}{2}}\mathsf{%
P}_{-\frac{1}{2}+{\rm i}\tau}^{-\mu}(x)}\big( 1+ \mathcal{O}\big( \tau^{-1}\big)\big). \label{Dun.eq38}
\end{gather}%
We then insert (\ref{Fcuab}), (\ref{Dun.eq30}) into (\ref{Dun.eq38}) and with $x=\cos\theta$, arrive at (\ref{Fcuaa}). Finally, using the connection formulas~(\ref{Pnummuconnection}), (\ref{Qnummuconnection})
(with $\nu =-\tfrac{1}{2}+{\rm i}\tau $), along with (\ref{Gammaratio}), (\ref{Fcuab}), (\ref{Fcuaa}), we establish (\ref{Fcuab2}), (\ref{Fcuac}). For (\ref{Fcuab}), (\ref{Fcuab2}), note (\ref{Dun.st}). Also note that in
each of (\ref{Fcuab2}), (\ref{Fcuaa}), the two terms of $\big(1+\mathcal{O}\big(\tau^{-1}\big)\big)$ do not factor as they do in (\ref{Fcuab}), (\ref{Fcuab2}) because there may be cancellations in the two terms which would then leave the result untrue. The asymptotic approximations (\ref{FcPmmumx}), (\ref{FcQmmumx}) follow using (\ref{Pnummumxconnection}), (\ref{Qnummumxconnection}), (\ref{Fcuab}), (\ref{Fcuaa}).
\end{proof}

\begin{cor} \label{cor:2.16}Let $\mu \geq 0$ be bounded. Then for the odd Ferrers conical function $($\eqref{FPodddefn} with $\nu =-\tfrac{1}{2}+{\rm i}\tau )$, as $0<\tau \rightarrow \infty $,
\begin{gather}
{\mathsf f}_{-\tfrac{1}{2}\pm {\rm i}\tau}^{-\mu}(\cos \theta )=\frac{1}{\tau^{\mu}}\sqrt{\frac{\theta}{\sin \theta}}\left[ \frac{{\rm e}^{\pi \tau}}{\pi}K_{\mu}( \tau \theta)-I_{\mu}(\tau\theta) \right] \left\{ 1+{\mathcal{O}\left( \frac{{1}}{{\tau}}\right)}\right\} , \label{CPodd1}
\end{gather}
uniformly for $\theta \in \big( 0,\frac{1}{2}\pi \big] $, and
\begin{gather}
{\mathsf f}_{-\tfrac{1}{2}\pm {\rm i}\tau}^{-\mu}(\cos \theta )=-\frac{1}{\tau^{\mu}}\sqrt{\frac{\pi -\theta}{\sin \theta}}\left[\frac{{\rm e}^{\pi \tau}}{\pi}K_{\mu
} ( \tau ( \pi -\theta ) )-I_{\mu} ( \tau ( \pi -\theta ) ) \right] \left\{ 1+{\mathcal{O}\left( \frac{{1}}{{\tau}}\right)}\right\} , \label{CPodd2}
\end{gather}
uniformly for $\theta \in \big[ \frac{1}{2}\pi ,\pi \big) $.
\end{cor}
\begin{proof}
For (\ref{CPodd1}) combine (\ref{Fcuab}), (\ref{FcPmmumx}). The proof of (\ref{CPodd2}) is similar to that of (\ref{FPodd2}).
\end{proof}

\subsubsection{Ferrers function addition theorems}\label{ffat}

In \cite[Theorem 4.1]{CohlPalmer}, an addition theorem is derived for a fundamental solution of the Laplace--Beltrami operator on a positive constant sectional curvature space, the $d$-dimensional hypersphere. That addition theorem was for the Ferrers function the second kind ${\sf Q}_\mu^{-\mu}(\cos\Theta)/\sin^\mu\Theta$, where
\begin{gather}
\cos\Theta=\cos\theta\cos\theta'+\sin\theta\sin\theta'\cos\gamma.\label{cosTheta}
\end{gather}
In this subsection we present generalizations of that addition theorem for ${\sf P}_\nu^{\pm\mu}(\pm\cos\Theta)/\sin^\mu\Theta$, ${\sf Q}_\nu^{\pm\mu}(\pm\cos\Theta)/\sin^\mu\Theta$.

\begin{thm}\label{GegexpansionsFerrers} Let $\mu \in \C$ such that ${\Re}\,\mu >-{\tfrac{1}{2}}$, $\mu\ne 0$, $\nu \in \C \setminus -\N$, $\gamma \in \R$, $\theta ,\theta ^{\prime}\in (0,\pi )$ such that
\begin{gather}
\tan\big({{\tfrac{1}{2}}\theta_{<}}\big)\tan\big({{\tfrac{1}{2}}\theta_{>}}\big)<1. \label{Dun2eq21}
\end{gather}
Then provided the Ferrers functions of the second kind are defined,
\begin{gather}
\frac{1}{\sin^\mu\Theta}{\sf P}_\nu^\mu(\cos\Theta)=\frac{2^\mu\Gamma(\mu)}{(\sin\theta\sin\theta^\prime)^\mu}\nonumber\\
\hphantom{\frac{1}{\sin^\mu\Theta}{\sf P}_\nu^\mu(\cos\Theta)=}{}\times \sum_{n=0}^\infty (-1)^n(n+\mu){\sf P}_\nu^{-(\mu+n)}(\cos\theta_<){\sf P}_\nu^{\mu+n}(\cos\theta_>)C_n^\mu(\cos\gamma),\label{conjPmPp}\\
\frac{1}{\sin^\mu\Theta}{\sf Q}_\nu^\mu(\cos\Theta)=\frac{2^\mu\Gamma(\mu)}{(\sin\theta\sin\theta^\prime)^\mu}\nonumber\\
\hphantom{\frac{1}{\sin^\mu\Theta}{\sf Q}_\nu^\mu(\cos\Theta)=}{}\times \sum_{n=0}^\infty (-1)^n(n+\mu){\sf P}_\nu^{-(\mu+n)}(\cos\theta_<){\sf Q}_\nu^{\mu+n}(\cos\theta_>)C_n^\mu(\cos\gamma),\label{conjPmQp}\\
\frac{1}{\sin^\mu\Theta}{\sf P}_\nu^{-\mu}(\cos\Theta)=\frac{2^\mu\Gamma(\mu)}{(\sin\theta\sin\theta^\prime)^\mu}\sum_{n=0}^\infty\Big[(-1)^n(n+\mu)(\nu+\mu+1)_n(\mu-\nu)_n\nonumber\\
\hphantom{\frac{1}{\sin^\mu\Theta}{\sf P}_\nu^{-\mu}(\cos\Theta)=}{} \times{\sf P}_\nu^{-(\mu+n)}(\cos\theta){\sf P}_\nu^{-(\mu+n)}(\cos\theta')C_n^\mu(\cos\gamma)\Big],\label{conjPmPm} \\
\frac{1}{\sin^\mu\Theta}{\sf Q}_\nu^{-\mu}(\cos\Theta)=\frac{2^\mu\Gamma(\mu)}{(\sin\theta\sin\theta^\prime)^\mu}\sum_{n=0}^\infty\Big[(-1)^n(n+\mu)(\nu+\mu+1)_n(\mu-\nu)_n\nonumber\\
\hphantom{\frac{1}{\sin^\mu\Theta}{\sf Q}_\nu^{-\mu}(\cos\Theta)=}{} \times{\sf P}_\nu^{-(\mu+n)}(\cos\theta_<){\sf Q}_\nu^{-(\mu+n)}(\cos\theta_>)C_n^\mu(\cos\gamma)\Big],\label{conjPmQm} \\
\frac{1}{\sin^\mu\Theta}{\sf P}_\nu^{-\mu}(- \cos\Theta)=\frac{2^\mu\Gamma(\mu)}{(\sin\theta\sin\theta^\prime)^\mu}\sum_{n=0}^\infty\Big[(n+\mu)(\nu+\mu+1)_n(\mu-\nu)_n\nonumber\\
\hphantom{\frac{1}{\sin^\mu\Theta}{\sf P}_\nu^{-\mu}(- \cos\Theta)=}{} \times{\sf P}_\nu^{-(\mu+n)}(\cos\theta_<){\sf P}_\nu^{-(\mu+n)}(- \cos\theta_>)C_n^\mu(\cos\gamma)\Big],\label{conjPmPmmx} \\
\frac{1}{\sin^\mu\Theta}{\sf Q}_\nu^{-\mu}(-\cos\Theta)=\frac{2^\mu\Gamma(\mu)}{(\sin\theta\sin\theta^\prime)^\mu}\sum_{n=0}^\infty\Big[(n+\mu)(\nu+\mu+1)_n(\mu-\nu)_n\nonumber\\
\hphantom{\frac{1}{\sin^\mu\Theta}{\sf Q}_\nu^{-\mu}(-\cos\Theta)=}{} \times{\sf P}_\nu^{-(\mu+n)}(\cos\theta_<){\sf Q}_\nu^{-(\mu+n)}(- \cos\theta_>)C_n^\mu(\cos\gamma)\Big],\label{conjQmPmmx}
\end{gather}
where in all but \eqref{conjPmPm} $\theta \neq \theta ^{\prime}$. For $\mu=0$, see Corollary~{\rm \ref{cormuzero}}.
\end{thm}

\begin{proof}See Appendix \ref{ProofoftheAdditionTheoremGegexpansionsFerrers}.
\end{proof}

\begin{rem}\label{remGegexpansionsFerrers}If one is interested in the above addition theorems which correspond to the Ferrers functions with $\nu=\mu$, then one can directly use (\ref{conjPmPp}), (\ref{conjPmQp}).
If one examines the behavior of (\ref{conjPmPm}), when $\nu=\mu$, it is seen that only the $n=0$ term in the sum survives. However, one can see through \cite[equation~(14.4.18)]{NIST:DLMF}, that the left-hand side exactly matches the $n=0$ term on the right-hand side. Note that the left-hand side of (\ref{conjPmQm}) is defined for $\nu=\mu$, but (except for $n=0$) the terms in the series on the right-hand side are undefined in this case. However, the limit of the series exists as $\nu\to\mu$, and is given by Corollary~\ref{Corollaryabovemunmu} below. Similar limiting results can be obtained if $\nu-\mu\in\N$ but we do not pursue this.
\end{rem}

The addition theorems for the Ferrers function of the second kind with $\nu=-\mu$ can be obtained in two different ways. These Ferrers functions appear in the study of a fundamental solution of the Laplace--Beltrami operator on the $d$-dimensional $R$-radius hypersphere. See~(\ref{thms1deq}) below and \cite{Chapling16,Cohlhypersphere,CohlPalmer}.

\begin{cor}\label{Corollaryabovemunmu} Let $\mu\in\C$ such that ${\Re}\,\mu>-{\tfrac{1}{2}}$, $\mu\ne0$, $\theta,\theta'\in(0,\pi)$, $\gamma\in\R$, \eqref{cosTheta} and \eqref{Dun2eq21} hold with $\theta\ne\theta'$. Then
\begin{gather}
\frac{1}{\sin^\mu\Theta}{\sf Q}_\mu^{-\mu}(\cos\Theta)=\frac{\pi\tan(\pi\mu)}{2^{\mu+1}\Gamma(\mu+1)}+\frac{\sqrt{\pi}\,\sec(\pi\mu)}{\mu 2^{\mu}\Gamma\big(\mu+\tfrac12\big)(\sin\theta\sin\theta')^\mu}\nonumber\\
\hphantom{\frac{1}{\sin^\mu\Theta}{\sf Q}_\mu^{-\mu}(\cos\Theta)=}{}\times \sum_{n=0}^\infty (-1)^n(n+\mu){\sf P}_\mu^{-n-\mu}(\cos\theta_<){\sf Q}_\mu^{n+\mu}(\cos\theta_>)C_n^\mu(\cos\gamma)\label{munothalfinteger}\\
\hphantom{\frac{1}{\sin^\mu\Theta}{\sf Q}_\mu^{-\mu}(\cos\Theta)}{}
=\frac{-\pi\cot(\pi\mu)}{2^{\mu+1}\Gamma(\mu+1)}+\frac{\pi^{\frac32}\csc(\pi\mu)}{\mu 2^{\mu+1}\Gamma\big(\mu+\tfrac12\big)(\sin\theta\sin\theta')^\mu}\nonumber\\
\hphantom{\frac{1}{\sin^\mu\Theta}{\sf Q}_\mu^{-\mu}(\cos\Theta)=}{}\times \sum_{n=0}^\infty (-1)^n(n+\mu){\sf P}_\mu^{-n-\mu}(\cos\theta_<){\sf P}_\mu^{n+\mu}(\cos\theta_>)C_n^\mu(\cos\gamma),\label{munotinteger}
\end{gather}
where the first equality is valid except when $\mu$ is a half odd integer, and the second equality is valid except when $\mu$ is an integer.
\end{cor}

\begin{proof} First start with (\ref{conjPmQp}), let $\nu=\mu$, and apply the connection relation~(\ref{Qnummuconnection}) to the left-hand side of the equation only. The result (\ref{munothalfinteger}) immediately follows using \cite[equation~(14.5.18)]{NIST:DLMF}. The derived formula is clearly undefined for $\mu$ being a half odd integer, but is otherwise valid. Next start with (\ref{conjPmPp}), let $\nu=\mu$, and apply the connection relation~(\ref{Pnummuconnection}), to the left-hand side of the formula only and \cite[equation~(14.5.18)]{NIST:DLMF}. The derived formula (\ref{munotinteger}) is clearly undefined for $\mu$ being an integer, but is otherwise valid.
\end{proof}

In order to obtain the addition theorems for the Ferrers function of the second kind when the order is equal to the negative degree, and is either an integer or a half odd integer, one can then use the above corollary. The following formulae, specializations of Corollary~\ref{Corollaryabovemunmu}, are exactly those results which are derived in \cite[Theorems~3.1 and 4.1]{CohlPalmer} for a fundamental solution of the Laplace--Beltrami operator on the $d$-dimensional $R$-radius hypersphere.

\begin{cor} Let $k\in\N$, $m\in\N_0$, $\theta,\theta'\in(0,\pi)$, $\gamma\in\R$, \eqref{cosTheta} and \eqref{Dun2eq21} hold with $\theta\ne\theta'$. Then
\begin{gather}
\log\cot\big(\tfrac12\Theta\big) =\log\cot\big(\tfrac12\theta_>\big)\nonumber\\
\hphantom{\log\cot\big(\tfrac12\Theta\big) =}{}+\sum_{n=1}^\infty \frac{2^n}{n}\left[\frac{\tan\big(\tfrac12\theta_<\big)}{\sin(\theta_>)}\right]^n
\bigl[\cos^{2n}\big(\tfrac12\theta_>\big)-(-1)^n\sin^{2n}\big(\tfrac12\theta_>\big)\bigr]T_n(\cos\gamma),\label{logcothalfTheta}\\
\frac{1}{\sin^k\Theta}{\sf Q}_k^{-k}(\cos\Theta)=\frac{\sqrt{\pi}(-1)^k}{k2^k\Gamma\big(k+\tfrac12\big)(\sin\theta\sin\theta')^k}\nonumber\\
\hphantom{\frac{1}{\sin^k\Theta}{\sf Q}_k^{-k}(\cos\Theta)=}{}\times\sum_{n=0}^\infty (-1)^n(n+k){\sf P}_k^{-n-k}(\cos\theta_<){\sf Q}_k^{n+k}(\cos\theta_>)C_n^k(\cos\gamma),\label{Qkmk}\\
\frac{1}{\sin^{m+\frac12}\Theta}{\sf Q}_{m+\frac12}^{-m-\frac12}(\cos\Theta)=\frac{(-1)^m\pi^{\frac32}}{(2m+1)2^{m+\frac32}m!(\sin\theta\sin\theta')^{m+\frac12}}\nonumber\\
\qquad {}\times\sum_{n=0}^\infty (-1)^n(2n+2m+1) {\sf P}_{m+\frac12}^{-n-m-\frac12}(\cos\theta_<){\sf P}_{m+\frac12}^{n+m+\frac12}(\cos\theta_>)C_n^{ m+\frac12}(\cos\gamma).\label{Qmhmmh}
\end{gather}
\end{cor}

\begin{proof}First substitute $\mu=k\in\N$ in (\ref{munothalfinteger}) which produces (\ref{Qkmk}). Next, substituting $\mu=m+\tfrac12$, $m\in\N_0$, in~(\ref{munotinteger}), to produce~(\ref{Qmhmmh}). Then take the limit as $\mu\to0$ in (\ref{munothalfinteger}) using~(\ref{mutozeroChebyGeg}) which produces~(\ref{logcothalfTheta}).
\end{proof}

\begin{cor}\label{cormuzero}Let $\nu\in\C\setminus-\N$, $\theta,\theta'\in(0,\pi)$, $\gamma\in\R$, \eqref{cosTheta} and \eqref{Dun2eq21} hold with $\theta\ne\theta'$. Then
\begin{gather}
{\sf P}_\nu(\cos\Theta)=\sum_{n=0}^\infty (-1)^n\epsilon_n{\sf P}_\nu^{-n}(\cos\theta_<){\sf P}_\nu^n(\cos\theta_>)T_n(\cos\gamma),\label{ferdep2addnthm}\\
{\sf Q}_\nu(\cos\Theta)=\sum_{n=0}^\infty (-1)^n\epsilon_n{\sf P}_\nu^{-n}(\cos\theta_<){\sf Q}_\nu^n(\cos\theta_>)T_n(\cos\gamma),\label{ferdeq2addnthm}\\
{\sf P}_\nu(-\cos\Theta)=\sum_{n=0}^\infty \epsilon_n(-\nu)_n(\nu+1)_n{\sf P}_\nu^{-n}(\cos\theta_<){\sf P}_\nu^{-n}(-\cos\theta_>)T_n(\cos\gamma),\label{ferdepm2addnthm}\\
{\sf Q}_\nu(-\cos\Theta)=\sum_{n=0}^\infty \epsilon_n(-\nu)_n(\nu+1)_n{\sf P}_\nu^{-n}(\cos\theta_<){\sf Q}_\nu^{-n}(-\cos\theta_>)T_n(\cos\gamma),\label{ferdepm2addnthmq}
\end{gather}
where $\epsilon_n=2-\delta_{n,0}$.
\end{cor}
\begin{proof}Taking the limit as $\mu\to0$ in (\ref{conjPmPp}), (\ref{conjPmQp}), (\ref{conjPmPmmx}), (\ref{conjQmPmmx}), and using (\ref{mutozeroChebyGeg}) produces the result.
\end{proof}

Note that the same result can be obtained by using (\ref{conjPmPm}), (\ref{conjPmQm}), using connection formulas in \cite[equations~(14.9.3) and~(14.9.4)]{NIST:DLMF}.
Similarly we can use the addition Theorem~\ref{GegexpansionsFerrers} to expand the functions on the left-hand sides in terms of Legendre polynomials $\big(\mu=\tfrac12\big)$. In this case, the Ferrers functions on the left-hand side reduce to trigonometric functions \cite[equation~(54.5.11)--(14.5.14)]{NIST:DLMF}.

\begin{cor}\label{trigaddnthms}Let $\theta,\theta'\in(0,\pi)$, $\gamma\in\R$, \eqref{cosTheta} and \eqref{Dun2eq21} hold with $\theta\ne\theta'$. Then
\begin{gather*}
\frac{\cos\big(\big(\nu+\tfrac12\big)\Theta\big)}{\sin\Theta}=\frac{\pi}{2\sqrt{\sin\theta\sin\theta'}}\sum_{n=0}^\infty (-1)^n(2n+1){\sf P}_{\nu}^{-n-\frac12}(\cos\theta_<){\sf P}_{\nu}^{n+\frac12}(\cos\theta_>)
P_n(\cos\gamma),\\
\frac{\sin\big(\big(\nu+\tfrac12\big)\Theta\big)}{\sin\Theta}=-\frac{1}{\sqrt{\sin\theta\sin\theta'}}\sum_{n=0}^\infty (-1)^n(2n+1){\sf P}_{\nu}^{-n-\frac12}(\cos\theta_<)
{\sf Q}_{\nu}^{n+\frac12}(\cos\theta_>)P_n(\cos\gamma).
\end{gather*}
\end{cor}
\begin{proof} Specializing (\ref{conjPmPp}), (\ref{conjPmQp}), with $\mu=\frac12$ and \cite[equations~(14.5.11) and (14.5.13)]{NIST:DLMF}, produces the result.
\end{proof}

\section[Global analysis on Riemannian manifolds of constant curvature]{Global analysis on Riemannian manifolds\\ of constant curvature}\label{Globalanalysisonthehyperboloid}

In this paper, when we refer to a {\it fundamental solution}, it is meant to be a fundamental solution of a partial differential operator on a Riemannian manifold~$M$. A fundamental solution for a~linear partial differential operator ${\mathcal E}_\bfx$ on a~$d$-dimensional Riemannian manifold~$M$, is in general a~distribution ${\mathcal G}^d\colon (M\times M)\setminus\{(\bfx,\bfx)\colon \bfx\in M\}\to\C$ which satisfies the following linear partial differential equation
\begin{gather*}
{\mathcal E}_\bfx{\mathcal G}^d(\bfx,\bfxp)=\delta_g(\bfx,\bfxp),
\end{gather*}
where $\bfx,\bfxp\in M$, $g$ is the Riemannian structure on $M$, and $\delta_g$ is the Dirac delta distribution on~$M$. Note that for the operators that are treated in this paper, a fundamental solution will always be function, but for other operators such as for wave operators (see, e.g., \cite[equation~(4.1.5)]{Kythe1996}), fundamental solutions are in general distributions. Furthermore, we will only perform global and local analysis for Riemannian manifolds $M$ with constant curvature (hyperspherical $\Si_R^d$, hyperbolic $\Hi_R^d$, and Euclidean geometry $\R^d$), and we will only study linear elliptic partial differential operators of the form $\big({-}\Delta\pm\beta^2\big)$, where $-\Delta$ is the positive Laplace--Beltrami operator on $M$ and $\beta^2\ge 0$, $R>0$. In this paper, these operators are referred to as Helmholtz operators, and in the case where $\beta=0$, they are Laplacian operators.

In this section we develop the necessary material in order to study fundamental solutions for these operators on these manifolds.

\subsection[Hyperspherical geometry and the hyperboloid model of hyperbolic geometry]{Hyperspherical geometry and the hyperboloid model\\ of hyperbolic geometry}\label{Thehyperboloidmodelofhyperbolicgeometry}

Hyperbolic space in $d$-dimensions is a fundamental example of a space exhibiting hyperbolic geometry. It was developed independently by Lobachevsky and Bolyai around 1830 (see~\cite{Trudeau}), and most likely by Gauss and Schweikart (although they never published this result), even earlier (see \cite[Chapter~6]{Livio2009}). It is a geometry analogous to Euclidean geometry, but such that Euclid's parallel postulate is no longer assumed to hold.

There are several models of $d$-dimensional hyperbolic space including the Klein, Poincar\'{e}, hyperboloid, upper-half space and hemisphere models (see~\cite{Thurston}). The hyperboloid model for $d$-dimensional
hyperbolic geometry is closely related to the Klein and Poincar\'{e} models: each can be obtained projectively from the others. The upper-half space and hemisphere models can be obtained from one another by inversions with the Poincar\'{e} model (see \cite[Section~2.2]{Thurston}). The model we will be focusing on in this paper is the hyperboloid model.

The hyperboloid model, also known as the Minkowski or Lorentz model, is a model of $d$-dimensional hyperbolic geometry in which points are represented by the upper sheet (submanifold) $S^+$ of a two-sheeted hyperboloid
embedded in the Minkowski space $\R^{d,1}$. Minkowski space is a~$(d+1)$-dimensional pseudo-Riemannian manifold which is a real finite-dimensional vector space, with coordinates given by $\bfx=(x_0,x_1,\ldots,x_d)$. It is equipped with a nondegenerate, symmetric bilinear form, the Minkowski bilinear form $[\cdot,\cdot]\colon \R^{d,1}\times\R^{d,1}\to\R$ defined such that
\begin{gather*}
[\bfx,{\mathbf y}]:=x_0y_0-x_1y_1-\dots-x_dy_d.
\end{gather*}
The above bilinear form is symmetric, but not positive-definite, so it is not an inner product. It is defined analogously with the Euclidean inner product $(\cdot,\cdot)\colon \R^{d+1}\times\R^{d+1}\to\R$ defined such that
\begin{gather*}
(\bfx,{\mathbf y}):=x_0y_0+x_1y_1+\dots+x_dy_d.
\end{gather*}

The variety $[\bfx,\bfx]=x_0^2-x_1^2-\dots-x_d^2=R^2$, for $\bfx\in\R^{d,1}$, using the language of~\cite{Beltrami} (see also \cite[p.~504]{Vilen}), defines a pseudo-sphere of radius~$R$. Points on the pseudo-sphere with zero radius coincide with the cone. Points on the pseudo-sphere with radius greater than zero lie within this cone, and points on the pseudo-sphere with purely imaginary radius lie outside the cone. The upper sheets of the positive radii pseudo-spheres are maximally symmetric, simply connected, negative-constant sectional curvature (given by $-1/R^2$, see for instance \cite[p.~148]{Lee}), $d$-dimensional Riemannian submanifolds, embedded and with induced metric from the ambient Minkowski space $\R^{d,1}$. For $R\in(0,\infty)$, we refer to the upper sheet of this variety $[\bfx,\bfx]=R^2$, with $\bfx\in\R^{d,1},$ as the $R$-radius hyperboloid $\Hi_R^d$.
Similarly, we refer to the variety $(\bfx,\bfx)=R^2$ for $R>0$ and $\bfx\in\R^{d+1}$, as the $R$-radius hypersphere $\Si_R^d$ which is a maximally symmetric, simply connected, positive-constant sectional curvature (given by $1/R^2$) $d$-dimensional Riemannian submanifold, embedded and with induced metric from the ambient Euclidean space. The Euclidean space $\R^d$ equipped with the Pythagorean norm, is a space with vanishing curvature. We denote the unit radius hyperboloid by $\Hi^d:=\Hi_1^d$ and the unit radius hypersphere by $\Si^d:=\Si_1^d$.

In our discussion of a fundamental solution for the Helmholtz operator in the hyperboloid model of hyperbolic geometry, we focus on the positive radius pseudo-sphere which can be parameterized through {\it subgroup-type coordinates,} i.e., those which correspond to a maximal subgroup chain ${\rm O}(d,1)\supset \cdots$ (see for instance~\cite{PogWin}). There exist separable coordinate systems which parameterize points on positive radius pseudo-spheres which can not be constructed using maximal subgroup chains, e.g., such as those which are analogous to parabolic coordinates, etc. We will no longer discuss these.

Geodesic polar coordinates are coordinates which correspond to the maximal subgroup chain given by ${\rm O}(d,1)\supset {\rm O}(d)\supset \cdots$. What we will refer to as {\it standard geodesic polar coordinates} correspond to the subgroup chain given by ${\rm O}(d,1)\supset {\rm O}(d)\supset {\rm O}(d-1)\supset \cdots\supset {\rm O}(2)$. Standard geodesic polar coordinates (see~\cite{groschepogsis, Olevskii}), similar to standard hyperspherical coordinates in Euclidean space, can be given on $\Hi_R^d$ by
\begin{gather}
x_0=R\cosh r,\nonumber\\
x_1=R\sinh r\cos\theta_{d-1},\nonumber\\
x_2=R\sinh r\sin\theta_{d-1}\cos\theta_{d-2},\nonumber\\
\cdots\cdots\cdots\cdots\cdots\cdots\cdots\cdots\cdots\cdots\cdots\nonumber\\
x_{d-2}=R\sinh r\sin\theta_{d-1}\cdots\cos\theta_2,\nonumber\\
x_{d-1}=R\sinh r\sin\theta_{d-1}\cdots\sin\theta_2\cos\phi,\nonumber\\
x_{d}=R\sinh r\sin\theta_{d-1}\cdots\sin\theta_2\sin\phi,\label{standardhyp}
\end{gather}
where $r\in[0,\infty)$, $\phi\in[0,2\pi)$, and $\theta_i\in[0,\pi]$ for $2\le i\le d-1$, and on $\Si_R^d$ by
\begin{gather}
x_0=R\cos\theta,\nonumber\\
x_1=R\sin\theta\cos\theta_{d-1},\nonumber\\
x_2=R\sin\theta\sin\theta_{d-1}\cos\theta_{d-2},\nonumber\\
\cdots\cdots\cdots\cdots\cdots\cdots\cdots\cdots\cdots\cdots\cdots\nonumber\\
x_{d-2}=R\sin\theta\sin\theta_{d-1}\cdots\cos\theta_2,\nonumber\\
x_{d-1}=R\sin\theta\sin\theta_{d-1}\cdots\sin\theta_2\cos\phi,\nonumber\\
x_{d}=R\sin\theta\sin\theta_{d-1}\cdots\sin\theta_2\sin\phi,\label{standardsph}
\end{gather}
 where $\theta:=\theta_d\in[0,\pi]$, $\theta_i\in[0,\pi]$ for $2\leq i\leq d-1$, and $\phi\in[0,2\pi)$.

The isometry group of the space $\Hi_R^d$ is the pseudo-orthogonal group ${\rm SO}(d,1)$, the Lorentz group in $(d+1)$-dimensions. Hyperbolic space $\Hi_R^d$, can be identified with the quotient space ${\rm SO}(d,1)/{\rm SO}(d)$. The isometry group acts transitively on~$\Hi_R^d$. That is, any point on the hyperboloid can be carried, with the help of a~Euclidean rotation of ${\rm SO}(d-1)$, to the point $(\cosh\alpha,\sinh\alpha,0,\ldots,0)$,
and a hyperbolic rotation
\begin{gather*}
x_0^\prime=-x_1\sinh\alpha+x_0\cosh\alpha,\qquad x_1^\prime=-x_1\cosh\alpha-x_0\sinh\alpha,
\end{gather*}
maps that point to the origin $(1,0,\ldots,0)$ of the space.

In order to do analysis on a fundamental solution of the Helmholtz equation on the hyperboloid and hypersphere, we need to describe how one computes distances in these spaces. One may naturally compare distances on the positive radius pseudo-sphere through analogy with the $R$-radius hypersphere. Distances on the hypersphere are simply given by arc lengths, angles between two arbitrary vectors, from the origin, in the ambient Euclidean space. We consider the $d$-dimensional hypersphere embedded in $\R^{d+1}$. Points on the hypersphere can be parameterized using hyperspherical coordinate systems. Any parameterization of the hypersphere $\Si_R^d$, must have $(\bfx,\bfx)=x_0^2+\dots+x_d^2=R^2$, with $R>0$. The geodesic distance between two points on the hypersphere $d_s\colon \Si_R^d\times\Si_R^d\to[0,\pi R]$ is given by
\begin{gather}
d_s(\bfx,\bfxp):=R\rho_s =R\cos^{-1}\left(\frac{(\bfx,\bfxp)}{(\bfx,\bfx)(\bfxp,\bfxp)} \right) =R\cos^{-1}\left(\frac{1}{R^2}(\bfx,\bfxp)\right).\label{dhypersphere}
\end{gather}
This is evident from the fact that the geodesics on $\Si_R^d$ are great circles (i.e., intersections of $\Si_R^d$ with planes through the origin) with constant speed parameterizations (see \cite[p.~82]{Lee}).

Accordingly, we now look at the geodesic distance function on the $d$-dimensional positive radius pseudo-sphere $\Hi_R^d$. Distances between two points on the positive radius pseudo-sphere are given by the hyperangle between two arbitrary vectors, from the origin, in the ambient Minkowski space. Any parameterization of the hyperboloid $\Hi_R^d$, must have $[\bfx,\bfx]=R^2$. The geodesic distance $d_h\colon \Hi_R^d\times\Hi_R^d\to[0,\infty)$ between any two points on the hyperboloid is given by
\begin{gather}
d_h(\bfx,\bfxp):=R\cosh^{-1}\left(\frac{[\bfx,\bfxp]}{[\bfx,\bfx][\bfxp,\bfxp]} \right) =R\cosh^{-1}\left(\frac{1}{R^2}[\bfx,\bfxp]\right),\label{dhyperboloid}
\end{gather}
where the inverse hyperbolic cosine with argument $x\in(1,\infty)$ is given by \cite[equation~(4.37.19)]{NIST:DLMF}, (\ref{Dun.eq4}). Geodesics on $\Hi_R^d$ are great hyperbolas (i.e., intersections of $\Hi_R^d$ with planes through the origin) with constant speed parameterizations (see \cite[p.~84]{Lee}). We also define two global functions $\rho_h\colon \Hi^d\times\Hi^d\to[0,\infty)$, $\rho_s\colon \Si^d\times\Si^d\to[0,\pi]$ which represent the projections of the global geodesic distance functions~(\ref{dhyperboloid}) on~$\Hi_R^d$ and~$\Si_R^d$ onto the corresponding unit radius hyperboloid~$\Hi^d$ and hypersphere~$\Si^d$ respectively, namely
\begin{gather}
\rho_h(\wbfx,\wbfxp):=d_h(\bfx,\bfxp)/R,\label{rhodefn}\\
\rho_s(\wbfx,\wbfxp):=d_s(\bfx,\bfxp)/R,\label{sphrhodefn}
\end{gather}
where $\wbfx=\bfx/R$ and $\wbfxp=\bfxp/R$.

\subsection{The Helmholtz equation in Riemannian spaces of constant curvature}\label{Laplaceequationandharmonicsonthehyperboloid}

Parametrizations of a submanifold embedded in either a Euclidean or Minkowski space are given in terms of coordinate systems whose coordinates are curvilinear. These are coordinates based on some transformation that converts the standard Cartesian coordinates in the ambient space to a coordinate system with the same number of coordinates as the dimension of the submanifold in which the coordinate lines are curved.

On a $d$-dimensional Riemannian manifold $M$ (a manifold together with a Riemannian metric~$g$), the Laplace--Beltrami operator (Laplacian) $\Delta\colon C^p(M)\to C^{p-2}(M)$, $p \ge 2$, in curvilinear coordinates
${\mathbf{\xi}}=\big(\xi^1,\ldots,\xi^d\big)$ is given by
\begin{gather}
\Delta=\sum_{i,j=1}^d\frac{1}{\sqrt{|g|}} \frac{\partial}{\partial \xi^i}\left(\sqrt{|g|}g^{ij}\frac{\partial}{\partial \xi^j} \right),\label{laplacebeltrami}
\end{gather}
where $|g|=|\det(g_{ij})|,$ the infinitesimal distance is given by
\begin{gather}
{\rm d}s^2=\sum_{i,j=1}^{d}g_{ij}{\rm d}\xi^i{\rm d}\xi^j,\label{metric}
\end{gather}
and
\begin{gather*}
\sum_{i=1}^{d}g_{ki}g^{ij}=\delta_k^j.
\end{gather*}
For a Riemannian submanifold, the relation between the metric tensor in the ambient space and~$g_{ij}$ of~(\ref{laplacebeltrami}), (\ref{metric}) is
\begin{gather*}
g_{ij}({\mathbf{\xi}})=\sum_{k,l=0}^dG_{kl}\frac{\partial x^k}{\partial \xi^i} \frac{\partial x^l}{\partial \xi^j}.
\end{gather*}
On $\Hi_R^d$ the ambient space is Minkowski, and therefore $G=\operatorname{diag}(1,-1,\ldots,-1)$. On $\Si_R^d$, $G =\operatorname{diag}(1,1,\ldots,1)$.

The set of all geodesic polar coordinate systems on the hyperboloid correspond to the many ways one can put coordinates on a hyperbolic hypersphere, i.e., the Riemannian submanifold $U\subset\Hi_R^d$ defined for a bounded $\bfxp\in\Hi_R^d$ such that $d_h(\bfx,\bfxp)=b={\rm const}$, where $b\in(0,\infty)$. These are coordinate systems which correspond to maximal subgroup chains starting with ${\rm O}(d,1) \supset {\rm O}(d) \supset \cdots$, with standard geodesic polar coordinates given by~(\ref{standardhyp}) being only one of them. (For a thorough description of these see \cite[Section~X.5]{Vilen}.) They all share the property that they are described by
$d$-variables: $r\in[0,\infty)$ plus $(d-1)$-angles each being given by the values $[0,2\pi)$, $[0,\pi]$, $[-\pi/2,\pi/2]$ or $[0,\pi/2]$ (see~\cite{IPSWa, IPSWb}).

In any of the geodesic polar coordinate systems, the global geodesic distance between any two points on the hyperboloid and hypersphere are given by (cf.~(\ref{dhyperboloid}), (\ref{sphrhodefn}))
\begin{gather}
 d_h(\bfx,\bfxp) =R\cosh^{-1}\bigl( \cosh r\cosh r^\prime-\sinh r\sinh r^\prime\cos\gamma\bigr),\label{diststandard}\\
d_s(\bfx,\bfxp)=R\cos^{-1}\bigl( \cos\theta\cos\theta^\prime+\sin\theta\sinh \theta^\prime\cos\gamma\bigr),\label{diststandards}
\end{gather}
where $\gamma$ is the unique separation angle defined on the unit radius submanifold $\Si^{d-1}$ which is embedded in both $d$-dimensional Riemannian manifolds $\Hi_R^d$ and $\Si_R^d$ for a fixed geodesic radius. For instance, one may write the separation angle in standard geodesic polar coordinates (\ref{standardhyp}) or (\ref{standardsph}) as follows
\begin{gather}
\cos\gamma=\cos(\phi-\phi^\prime)\prod_{i=1}^{d-2}\sin\theta_i{\sin\theta_i}^\prime+\sum_{i=2}^{d-2}\cos\theta_i{\cos\theta_i}^\prime\prod_{j=i+1}^{d-2}\sin\theta_j{\sin\theta_j}^\prime.\label{prodform}
\end{gather}
Corresponding separation angle formulae for any geodesic polar coordinate system can be computed using (\ref{dhypersphere}), (\ref{dhyperboloid}), and the associated formulae for the appropriate inner-products.

Note that by making use of the isometry group ${\rm SO}(d,1)$ to map~$\bfxp$ to the origin, then the geodesic distance $\rho$ as measured from the origin to a point $\bfx\in\Hi_R^d$ with curvilinear coordinate~$r$ is $\rho=Rr$. Hence if $R=1$ which corresponds to $\Hi^d$ (the unit radius hyperboloid), then on this Riemannian manifold, the geodesic distance is $\rho=r$, and there is no distinction between the global geodesic distance and the $r$-parameter in a geodesic polar coordinate system.

The infinitesimal distance in a geodesic polar coordinate system on the submanifold $\Hi_R^d$ is given by
\begin{gather}
{\rm d}s^2=R^2\big({\rm d}r^2+\sinh^2r{\rm d}\gamma^2\big),\label{stanhypmetric}
\end{gather}
and on $\Si_R^d$ by
\begin{gather}
{\rm d}s^2=R^2\big({\rm d}\theta^2+\sin^2\theta{\rm d}\gamma^2\big),\label{stansphmetric}
\end{gather}
where an appropriate expression for $\gamma$ in a curvilinear coordinate system is given by~(\ref{prodform}). If one combines (\ref{standardhyp}) or (\ref{standardsph}), (\ref{laplacebeltrami}), (\ref{prodform})
and (\ref{stanhypmetric}) or (\ref{stansphmetric}), then in a particular geodesic polar coordinate system, the Helmholtz equation on $\Hi_R^d$ is given by
\begin{gather}
\big({-}\Delta \pm \beta^2\big)f = \frac{1}{R^2}\left[-\frac{\partial^2 f}{\partial r^2}-(d-1)\coth r\frac{\partial f}{\partial r}-\frac{1}{\sinh^2r} \Delta_{\Si^{d-1}} f\pm \beta^2 R^2 f \right] =0,\label{genhyplap}
\end{gather}
and on $\Si_R^d$ by
\begin{gather}
\big({-}\Delta\pm\beta^2\big)f = \frac{1}{R^2}\left[-\frac{\partial^2f}{\partial \theta^2}-(d-1)\cot\theta\frac{\partial f}{\partial\theta}-\frac{1}{\sin^2\theta}\Delta_{\Si^{d-1}}f \pm \beta^2R^2f\right]=0,\label{gensphlap}
\end{gather}
where $\Delta_{\Si^{d-1}}$ is the corresponding Laplace--Beltrami operator on the unit radius hypersphe\-re~$\Si^{d-1}$.

\subsection[Homogeneous solutions of the Helmholtz equation in geodesic polar coordinates]{Homogeneous solutions of the Helmholtz equation\\ in geodesic polar coordinates}\label{SepVarStaHyp}

Geodesic polar coordinate systems partition $\Hi_R^d$ into a family of $(d-1)$-dimensional hyperbolic hyperspheres, each with a geodesic radius $Rr$ with $r\in(0,\infty)$ on which all possible hyperspherical coordinate systems for $\Si^{d-1}$ may be used (see for instance~\cite{Vilen}). One then must also consider the limiting case for $r=0$ to fill out all of $\Hi_R^d$. In subgroup-type coordinate systems, one can compute the normalized hyperspherical harmonics in that space by solving the Laplace equation using separation of variables. This results in a general procedure which is given explicitly in~\cite{IPSWa, IPSWb}. These angular harmonics are given as general expressions involving trigonometric functions, Gegenbauer polynomials and Jacobi polynomials.

The harmonics in geodesic polar coordinate systems are given in terms of a radial solution multiplied by the angular harmonics. The angular harmonics are eigenfunctions of the Laplace--Beltrami operator on $\Si^{d-1}$ which satisfy the following eigenvalue problem
\begin{gather*}
\Delta_{\Si^{d-1}}Y_l^K (\wbfx) =-l(l+d-2)Y_l^K(\wbfx),
\end{gather*}
where $\wbfx\in\Si^{d-1}$, $Y_l^K (\wbfx)$ are normalized hyperspherical harmonics, $l\in\N_0$ is the angular momentum quantum number, and~$K$ stands for the set of $(d-2)$-quantum numbers identifying degenerate harmonics for each $l$. The degeneracy $(2l+d-2)(d-3+l)!/(l!(d-2)!)$ (see \cite[equation~(9.2.11)]{Vilen}), tells you how many linearly independent solutions exist for a particular $l$ value and dimension~$d$. The hyperspherical harmonics can optionally be normalized so that
\begin{gather*}
\int_{\Si^{d-1}}Y_l^K(\wbfx)\overline{Y_{l^\prime}^{K^\prime}(\wbfx)} \mathrm{d}\omega=\delta_{l}^{l^\prime}\delta_{K}^{K^\prime},
\end{gather*}
where $d\omega$ is the Riemannian (volume) measure (see for instance \cite[Section 3.4]{Grigor}) on $\Si^{d-1}$ which is invariant under the isometry group ${\rm SO}(d)$ (cf.~(\ref{eucsphmeasureinv})), and for $x+{\rm i}y=z\in\C$, $\overline{z}=x-{\rm i}y$, represents complex conjugation. The generalized Kronecker delta $\delta_K^{K^\prime}\in\{0,1\}$ is defined such that it equals~1 if all of the $(d-2)$-quantum numbers identifying degenerate harmonics for each~$l$ coincide, and equals zero otherwise.

Since the angular solutions (hyperspherical harmonics) are well-known (see for instance \cite[Chapter~IX]{Vilen}, \cite[Chapter~11]{ErdelyiHTFII}), we will now focus on the radial solutions on $\Hi_R^d$ and $\Si_R^d$
coming from $\big({-}\Delta\pm\beta^2\big)$ in geodesic polar coordinate systems. These radial solutions respectively satisfy the following ordinary differential equations (cf.~(\ref{genhyplap})) for all $R\in(0,\infty)$, namely,
\begin{gather}
 -\frac{{\rm d}^2h}{{\rm d}r^2}-(d-1)\coth r\frac{{\rm d}h}{{\rm d}r}+\left(\frac{l(l+d-2)}{\sinh^2r}+\beta^2R^2\right)h=0,\label{hypodep}\\
-\frac{{\rm d}^2h}{{\rm d}r^2}-(d-1)\coth r\frac{{\rm d}h}{{\rm d}r}+\left(\frac{l(l+d-2)}{\sinh^2r}-\beta^2R^2\right)h=0,\label{hypodem}
\end{gather}
on $\Hi_R^d$ and
\begin{gather}
-\frac{{\rm d}^2s}{{\rm d}\theta^2}-(d-1)\cot\theta\frac{{\rm d}s}{{\rm d}\theta} +\left(\frac{l(l+d-2)}{\sin^2\theta}+\beta^2R^2\right)s=0,\label{sphodep}\\
-\frac{{\rm d}^2s}{{\rm d}\theta^2}-(d-1)\cot\theta\frac{{\rm d}s}{{\rm d}\theta} +\left(\frac{l(l+d-2)}{\sin^2\theta}-\beta^2R^2\right)s=0,\label{sphodem}
\end{gather}
on $\Si_R^d$.

Eight solutions to the ordinary differential equation (\ref{hypodep}) $h_{1,R,\beta,\pm}^{d,l,\pm},h_{2,R,\beta,\pm}^{d,l,\pm}\colon (0,\infty)\to\C$ are given by
\begin{gather*}
h_{1,R,\beta,\pm}^{d,l,\pm}(r)=\frac{1}{\sinh^{\frac{d}{2}-1}r} P_{-{\frac{1}{2}}\pm{\frac{1}{2}}\sqrt{(d-1)^2+4\beta^2 R^2}}^{\pm(\frac{d}{2}-1+l)}(\cosh r),\\
h_{2,R,\beta,\pm}^{d,l,\pm}(r)=\frac{1}{\sinh^{\frac{d}{2}-1}r}Q_{-{\frac{1}{2}}\pm{\frac{1}{2}}\sqrt{(d-1)^2+4\beta^2 R^2}}^{\pm(\frac{d}{2}-1+l)}(\cosh r).
\end{gather*}
Similarly, if $(d-1)^2-4\beta^2R^2 \ge 0$, the solutions to (\ref{hypodem}) $f_{1,\beta,R,\pm}^{d,l,\pm}, f_{2,R,\beta,\pm}^{d,l,\pm}\colon (0,\infty)\to\C$ are
\begin{gather*}
f_{1,R,\beta,\pm}^{d,l,\pm}(r)=\frac{1}{\sinh^{\frac{d}{2}-1}r} P_{-{\frac{1}{2}}\pm{\frac{1}{2}}\sqrt{(d-1)^2-4\beta^2R^2}}^{\pm(\frac{d}{2}-1+l)}(\cosh r),\\
f_{2,R,\beta,\pm}^{d,l,\pm}(r)=\frac{1}{\sinh^{\frac{d}{2}-1}r} Q_{-{\frac{1}{2}}\pm{\frac{1}{2}}\sqrt{(d-1)^2-4\beta^2R^2}}^{\pm(\frac{d}{2}-1+l)}(\cosh r).
\end{gather*}
If $(d-1)^2-4\beta^2R^2 \le 0$, the solutions to (\ref{hypodem}) are given in terms of the associated Legendre conical functions \cite[Section~14.20]{NIST:DLMF}
\begin{gather*}
y_{1,R,\beta,\pm}^{d,l,\pm}(r) = \frac{1}{\sinh^{\frac{d}{2}-1}r}P_{-{\frac{1}{2}}\pm\frac{{\rm i}}{2}\sqrt{4\beta^2R^2-(d-1)^2}}^{\pm(\frac{d}{2}-1+l)}(\cosh r),\\
y_{2,R,\beta,\pm}^{d,l,\pm}(r) = \frac{1}{\sinh^{\frac{d}{2}-1}r}Q_{-{\frac{1}{2}}\pm\frac{{\rm i}}{2}\sqrt{4\beta^2R^2-(d-1)^2}}^{\pm(\frac{d}{2}-1+l)}(\cosh r).
\end{gather*}

On $\Si_R^d$, solutions to (\ref{sphodem}) $s_{1,R,\beta,\pm}^{d,l,\pm}, s_{2,R,\beta,\pm}^{d,l,\pm}\colon (0,\pi)\to\C$ are given by
\begin{gather*}
s_{1,R,\beta,\pm}^{d,l,\pm}(\theta) = \frac{1}{\sin^{\frac{d}{2}-1}\theta} \sP_{-{\frac{1}{2}}\pm{\frac{1}{2}}\sqrt{(d-1)^2+4\beta^2R^2}}^{\pm(\frac{d}{2}-1+l)}(\pm\cos\theta),\\
s_{2,R,\beta,\pm}^{d,l,\pm}(\theta) = \frac{1}{\sin^{\frac{d}{2}-1}\theta}\sQ_{-{\frac{1}{2}}\pm{\frac{1}{2}}\sqrt{(d-1)^2+4\beta^2R^2}}^{\pm(\frac{d}{2}-1+l)}(\pm\cos\theta),
\end{gather*}
where $\sP_\nu^\mu, \sQ_\nu^\mu\colon (-1,1)\to\C$ are Ferrers functions of the first and second kind \cite[Section 14.3(i)]{NIST:DLMF}. Similarly, if $(d-1)^2-4\beta^2R^2 \ge 0$, solutions to (\ref{sphodep})
$w_{1,\beta,\pm}^{d,l,\pm}, w_{2,\beta,\pm}^{d,l,\pm}\colon (0,\pi)\to\C$ are
\begin{gather*}
w_{1,R,\beta,\pm}^{d,l,\pm}(\theta) = \frac{1}{\sin^{\frac{d}{2}-1}\theta} \sP_{-{\frac{1}{2}}\pm{\frac{1}{2}}\sqrt{(d-1)^2-4\beta^2R^2}}^{\pm(\frac{d}{2}-1+l)}(\pm\cos\theta),\\
w_{2,R,\beta,\pm}^{d,l,\pm}(\theta) = \frac{1}{\sin^{\frac{d}{2}-1}\theta}\sQ_{-{\frac{1}{2}}\pm{\frac{1}{2}}\sqrt{(d-1)^2-4\beta^2R^2}}^{\pm(\frac{d}{2}-1+l)}(\pm\cos\theta).
\end{gather*}
Otherwise, if $(d-1)^2-4\beta^2R^2 \le 0$, the solutions to (\ref{sphodep}) are given in terms of Ferrers conical functions \cite[Section~14.20]{NIST:DLMF}
\begin{gather*}
u_{1,R,\beta,\pm}^{d,l,\pm}(\theta)= \frac{1}{\sin^{\frac{d}{2}-1}\theta}\sP_{-{\frac{1}{2}}\pm\frac{{\rm i}}{2}\sqrt{4\beta^2R^2-(d-1)^2}}^{\pm(\frac{d}{2}-1+l)}(\pm\cos\theta),\\
u_{2,R,\beta,\pm}^{d,l,\pm}(\theta)=\frac{1}{\sin^{\frac{d}{2}-1}\theta}\sQ_{-{\frac{1}{2}}\pm\frac{{\rm i}}{2}\sqrt{4\beta^2R^2-(d-1)^2}}^{\pm(\frac{d}{2}-1+l)}(\pm\cos\theta).
\end{gather*}

\begin{rem}\label{ellpintrem}\sloppy If $\frac{d}{2}-1+l$ is a half odd integer ($d$ odd), the solutions can be expressed as~ele\-men\-tary functions. From \cite[equations~(14.5.15)--(14.5.17)]{NIST:DLMF}, then $P_\nu^{\pm(\frac{d}{2}-1+l)}(\cosh r)$ and $Q_\nu^{\pm(\frac{d}{2}-1+l)}(\cosh r)$ with odd $d$ can be found by using the order-recurrence relation \cite[equation~(14.10.6)]{NIST:DLMF}. From \cite[equations~(14.5.11)--(14.5.14)]{NIST:DLMF}, then ${\sf P}_\nu^{\pm(\frac{d}{2}-1+l)}(\cos\theta)$ and ${\sf Q}_\nu^{\pm(\frac{d}{2}-1+l)}(\cos\theta)$ with odd $d$ can be found by using the order-recurrence relation \cite[equation~(14.10.1)]{NIST:DLMF}. On the other hand, if $d$ is even and $(d-1)^2-4\beta^2R^2=0$, the solutions to (\ref{hypodep}), (\ref{sphodep}) can be written in terms of complete elliptic integrals of the first and second kind $K\colon [0,1)\to[\tfrac{\pi}{2},\infty)$ and $E\colon [0,1]\to[1,\tfrac{\pi}{2}]$ respectively. These are given in terms of the Gauss hypergeometric function, namely \cite[equations~(19.5.1) and (19.5.2)]{NIST:DLMF} or through definite integrals \cite[Section~19.2]{NIST:DLMF}.
Using \cite[equations~(14.5.24)--(14.5.27)]{NIST:DLMF}, then, $P_{-{\frac{1}{2}}}^1(\cosh r)$ is obtained from $P_{{\frac{1}{2}}}(\cosh r)$ with $P_{-{\frac{1}{2}}}(\cosh r)$, and $Q_{-{\frac{1}{2}}}^1(\cosh r)$ from $Q_{{\frac{1}{2}}}(\cosh r)$ with $Q_{-{\frac{1}{2}}}(\cosh r)$, by the recurrence relations \cite[equation~(14.10.7)]{NIST:DLMF}. Using \cite[equations~(14.5.20)--(14.5.23)]{NIST:DLMF}, then, $\sP_{-{\frac{1}{2}}}^1(\cos\theta)$ is obtained from $\sP_{{\frac{1}{2}}}(\cos\theta)$ with $\sP_{-{\frac{1}{2}}}(\cos\theta)$, and $\sQ_{-{\frac{1}{2}}}^1(\cos\theta)$ from $\sQ_{{\frac{1}{2}}}(\cos\theta)$ with $\sQ_{-{\frac{1}{2}}}(\cos\theta)$, by the recurrence relations \cite[equation~(14.10.2)]{NIST:DLMF}. Subsequently, the recurrence relation \cite[equation~(14.10.6)]{NIST:DLMF} can be used to find $P_{-{\frac{1}{2}}}^{\pm(\frac{d}{2}-1+l)}(\cosh r)$ and $Q_{-{\frac{1}{2}}}^{\pm(\frac{d}{2}-1+l)}(\cosh r)$ in terms of complete elliptic integrals for all half odd integer values of $\pm(\frac{d}{2}-1+l)$ by starting with $P_{-{\frac{1}{2}}}(\cosh r)$, $P_{-{\frac{1}{2}}}^1(\cosh r)$, $Q_{-{\frac{1}{2}}}(\cosh r)$, and $Q_{-{\frac{1}{2}}}^1(\cosh r)$. Similarly, \cite[equation~(14.10.1)]{NIST:DLMF} can be used to find $\sP_{-{\frac{1}{2}}}^{\pm(\frac{d}{2}-1+l)}(\cos\theta)$ and $\sQ_{-{\frac{1}{2}}}^{\pm(\frac{d}{2}-1+l)}(\cos\theta)$ in terms of complete elliptic integrals for all half odd integer values
of $\pm(\frac{d}{2}-1+l)$ by starting with $\sP_{-{\frac{1}{2}}}(\cos\theta)$, $\sP_{-{\frac{1}{2}}}^1(\cos\theta)$, $\sQ_{-{\frac{1}{2}}}(\cos\theta)$, and $\sQ_{-{\frac{1}{2}}}^1(\cos\theta)$.
\end{rem}

\section[A fundamental solution of $\big({-}\Delta\pm\beta^2\big)$ in Riemannian spaces of constant curvature]{A fundamental solution of $\boldsymbol{\big({-}\Delta\pm\beta^2\big)}$ in Riemannian spaces\\ of constant curvature}\label{AGreensfunctioninthehyperboloidmodel}

Due to the fact that the spaces $\Hi_R^d$ and $\Si_R^d$ are homogeneous with respect to their isometry groups, the pseudo-orthogonal group ${\rm SO}(d,1)$ of $\Hi_R^d$ and the orthogonal group ${\rm O}(d)$ of $\Si_R^d$, and therefore isotropic manifolds, we expect that there exists a fundamental solution of the Helmholtz equation on each space with spherically symmetric dependence. We specifically expect these solutions to be given on $\Hi_R^d$ in terms of the associated Legendre function of the second kind with argument given by~$\cosh r$. This associated Legendre function naturally fits our requirements because it is singular at $r=0$ and vanishes at infinity, whereas the associated Legendre function of the first kind, with the same argument, is regular at $r=0$ and singular at infinity. One also might expect a fundamental solution of the Helmholtz equation on hyperspheres to be a~Ferrers function of the second kind with argument given by $\cos\theta$, since this is the result we have found previously for Laplace's equation (see~(\ref{thms1deq}) below). However on hyperspheres for the Helmholtz equation, there is a twist, as we will see in the next subsection.

\subsection[Properties of fundamental solutions of inhomogeneous Laplace/Helmholtz equations on $\Hi_R^d$ and $\Si_R^d$]{Properties of fundamental solutions \\ of inhomogeneous Laplace/Helmholtz equations on $\boldsymbol{\Hi_R^d}$ and $\boldsymbol{\Si_R^d}$} \label{PropertiesLaplaceHelmholtz}

In computing a fundamental solution of the Helmholtz equation, we know that on~$\Hi_R^d$, $\Si_R^d$ respectively, a fundamental solution of Helmholtz operators satisfy
\begin{gather}
\big({-}\Delta\pm\beta^2\big)\mch_{R,\beta}^{d,\pm}(\bfx,\bfxp)= \delta_h(\bfx,\bfxp),\label{hypeq3}\\
\big({-}\Delta\pm\beta^2\big)\mss_{R,\beta}^{d,\pm}(\bfx,\bfxp)= \delta_s(\bfx,\bfxp),\label{spheq3}
\end{gather}
where $h$ and $s$ are Riemannian metrics on $\Hi_R^d$ and $\Si_R^d$ respectively, and $\delta_h(\bfx,\bfxp)$ and $\delta_s(\bfx,\bfxp)$ are the corresponding Dirac delta distributions on those Riemannian manifolds.
\begin{rem}\label{badvaluessphere} Note that solutions to the inhomogeneous Helmholtz equation corresponding to the operator $\big({-}\Delta-\beta^2\big)$ on~$\Si_R^d$ exist except at a countably infinite number of `bad' positive values of $\beta^2$. These correspond to the eigenvalues of the Laplace--Beltrami operator on $\Si_R^d$ and are given by $\nu\in\N_0$. Also note that Szmytkowski~\cite{Szmytkowskiunitsphere07} has studied a {\it modified} or {\it reduced} fundamental solution to deal with this case, but we have not yet pursued this.
\end{rem}

Consider Poisson's equation on $\Si_R^d$, $-\Delta u=\rho$ on a compact Riemannian manifold $M$ with boundary $\partial M$. The divergence theorem on this manifold is given by (cf.~\cite[p.~43]{Lee})
\begin{gather}
\int_M \operatorname{div} \bfX {\rm d}V=\int_{\partial M} \langle \bfX,\bfN\rangle \mathrm{d}{\tilde V},\label{divergencethm}
\end{gather}
where ${\rm d}V$ is the Riemannian volume measure on $M$, $\bfN$ is the outward unit normal to $\partial M$, and ${\rm d}{\tilde V}$ is the Riemannian volume measure of the induced metric on~$\partial M$. If one invokes the divergence theorem on~$\Si_R^d$, with regard to Poisson's equation on this manifold using $\bfX=\nabla u$, then since $\partial \Si_R^d=\varnothing$, one ascertains
\begin{gather*}
\int_{\Si_R^d}\rho \mathrm{d}V=0.
\end{gather*}
Hence on $\Si_R^d$ (and on all compact manifolds without boundary), there does not exist a source density distribution $\rho$, satisfying Poisson's equation, with non-vanishing integral. In fact, a~fundamental solution of Laplace's equation on~$\Si_R^d$ (\ref{thms1deq}), as pointed out in~\cite[Section~5.4]{Chapling16} is actually the solution to the Poisson equation whose inhomogeneous source distribution is given by a~point source at the
origin and another with opposite sign, on the opposite pole of the hypersphere (both modeled by Dirac delta distributions).

Applying the divergence theorem to solutions of the inhomogeneous Helmholtz equations $\big({-}\Delta\pm\beta^2\big)u=\rho$ on $\Si_R^d$ produces
\begin{gather}
\int_{\Si_R^d} \rho \mathrm{d}V = \pm \beta^2\int_{\Si_R^d} u \mathrm{d}V.\label{hyperspheresinglenorm}
\end{gather}
Hence for the Helmholtz equation, the integral of the density distribution over the entire manifold is not required to vanish. Instead, it must be equal to the total integral of the solution multiplied by the wavenumber $\pm \beta^2$. Note that in the case where the total integral of the density distribution vanishes, this implies that the total integral of the solution must also vanish. This is the case for a fundamental solution of Laplace's equation on~$\Si_R^d$, in which case a fundamental solution must be an odd function of the geodesic distance as measured from the equator $\theta=\tfrac12 \pi$ of the hypersphere $\Si_R^d$. We will also study odd solutions for Helmholtz operators on these manifolds and compare the $\beta\to 0^{+}$ limit for Laplace operators for the opposite antipodal fundamental solutions
\begin{gather}
\big({-}\Delta\pm\beta^2\big)\mas_{R,\beta}^{d,\pm}(\bfx,\bfxp)= \delta_s(\bfx,\bfxp)-\delta_s(-\bfx,\bfxp),\label{oaspheq3}
\end{gather}
whose total integral over the manifold must vanish.

Another important requirement for a fundamental solution of the $\big({-}\Delta- \beta^2\big)$ Helmholtz operator in Euclidean space is the {\it Sommerfeld radiation condition}. Sommerfeld defined this radiation condition for a scalar field satisfying the Helmholtz equation as \cite[p.~189]{Sommerfeld49} (see also \cite[p.~396]{Schot92}) ``\dots\ the sources must be {\it sources}, not {\it sinks} of energy. The energy which is radiated from the sources must scatter to infinity; {\it no energy may be radiated from infinity into the prescribed singularities of the field}.'' On $\R^d$, this radiation condition is expressed mathematically as
\begin{gather}
\lim_{r\to\infty}r^{\frac{d-1}{2}}\left(\frac{\partial}{\partial r}-{\rm i}\beta\right)\mcg_{\beta}^{d,-}(r)=0,\label{sommerfeldcond}
\end{gather}
where $r=\|\bfx-\bfxp\|$, $\bfx,\bfxp\in\R^d$. The fundamental solution $\mcg_{\beta}^{d,-}$ satisfying this condition will be given in Theorem~\ref{thmg1n} below.

\subsection[A fundamental solution of the Helmholtz equation in spaces of constant curvature]{A fundamental solution of the Helmholtz equation\\ in spaces of constant curvature}\label{FunSolLapHd}

The Dirac delta distribution with metric $g$ is defined for an open set $U\subset\Hi_R^d$ with $\bfx,\bfxp\in\Hi_R^d$, or $U\subset\Si_R^d$ with $\bfx,\bfxp \in \Si_R^d$, such that
\begin{gather*}
\int_U\delta_g(\bfx,\bfxp) \mathrm{d}\mathrm{vol}_g =
 \begin{cases}
 1, & \mathrm{if}\ \bfxp\in U, \\
 0, & \mathrm{if}\ \bfxp\notin U,
 \end{cases}
\end{gather*}
where ${\rm d}\mathrm{vol}_g$ is the Riemannian (volume) measure, invariant under the isometry group ${\rm SO}(d,1)$ of the Riemannian manifold $\Hi_R^d$ or ${\rm O}(d)$ of $\Si_R^d$, given in standard geodesic polar coordinates on~$\Hi_R^d$ by
\begin{gather}
{\rm d}\mathrm{vol}_h=R^d \sinh^{d-1}r {\rm d}r {\rm d}\omega:= R^d\sinh^{d-1}r {\rm d}r \sin^{d-2}\theta_{d-1}\cdots\sin\theta_2 {\rm d}\theta_{1}\cdots {\rm d}\theta_{d-1},\label{eucsphmeasureinv}
\end{gather}
and on $\Si_R^d$ by
\begin{gather*}
{\rm d}\mathrm{vol}_s=R^d \sin^{d-1}\theta {\rm d}\theta {\rm d}\omega:= R^d\sin^{d-1}\theta {\rm d}\theta \sin^{d-2}\theta_{d-1}\cdots\sin\theta_{2} {\rm d}{\theta_1}\cdots {\rm d}\theta_{d-1},\label{sphdvol}
\end{gather*}
where ${\rm d}\omega$ is the Riemannian volume measure on $\Si^{d-1}$. Notice that as $r\to 0^+$ and $\theta\to 0^+$, ${\rm d}\mathrm{vol}_h$ and ${\rm d}\mathrm{vol}_s$ go to the Euclidean volume measure, invariant under the Euclidean motion group~${\rm E}(d)$, in standard geodesic polar coordinates. Therefore in standard geodesic polar coordinates, we have the following:
\begin{gather*}
\delta_h(\bfx,\bfxp)=\frac{\delta(r-r^\prime)}{R^d\sinh^{d-1}r^\prime}\frac{\delta(\theta_1-\theta_1^\prime)\delta(\theta_2-\theta_2^\prime)\cdots\delta(\theta_{d-1}-\theta_{d-1}^\prime)} {\sin\theta_2^\prime\sin^2\theta_3'\cdots\sin^{d-2}\theta_{d-1}^\prime},\label{diracdeltasubg}\\
\delta_s(\bfx,\bfxp)=\frac{\delta(\theta-\theta^\prime)}{R^d\sin^{d-1}\theta^\prime}\frac{\delta(\theta_1-\theta_1^\prime)\delta(\theta_2-\theta_2^\prime)
\cdots\delta(\theta_{d-1}-\theta_{d-1}^\prime)}{\sin\theta_2^\prime\sin^2\theta_3'\cdots\sin^{d-2}\theta_{d-1}^\prime}.\label{sphdiracdeltasubg}
\end{gather*}

In general since we can add any function that satisfies the homogeneous Helmholtz equation to a fundamental solution of the same equation and still have a fundamental solution, we will use this freedom to make our fundamental solution as simple as possible. It is reasonable to expect that there exists a particular spherically symmetric fundamental solution on each manifold ($\mch_{R,\beta}^{d,\pm}(\bfx,\bfxp)$ on $\Hi_R^d$ and $\mss_{R,\beta}^{d,\pm}(\bfx,\bfxp)$ on $\Si_R^d$) with pure radial ($\rho_h(\wbfx,\wbfxp)$~(\ref{rhodefn}) on $\Hi_R^d$ or $\rho_s(\wbfx,\wbfxp)$~(\ref{sphrhodefn}) on~$\Si_R^d$) and constant angular dependence (invariant under rotations centered about the origin), due to the influence of the point-like nature of the Dirac delta distribution. For a spherically symmetric solution of the Helmholtz equation, the corresponding $\Delta_{\Si^{d-1}}$ term vanishes since only the $l=0$ term survives. In other words, we expect there to exist a fundamental solution of the Helmholtz equation such that, aside from a multiplicative constant which depends on $R$, $\beta$, and $d$, $\mch_{R,\beta}^{d,\pm}(\bfx,\bfxp)=f_1(\rho_h)$ on $\Hi_R^d$ and $\mss_{R,\beta}^{d,\pm}(\bfx,\bfxp)=f_2(\rho_s)$ on $\Si_R^d$.

In the remainder of this section, we derive fundamental solutions for $\big({-}\Delta \pm \beta^2\big)$ on $\Hi_R^d$ (Section \ref{fshoHiRd}), $\big({-}\Delta+\beta^2\big)$ on $\Si_R^d$ (Section \ref{fsonsrdmdmb2}), and study two candidate fundamental solutions for $\big({-}\Delta-\beta^2\big)$ on $\Si_R^d$ (Section \ref{cfsSiRdmDmb2}). We will require a fundamental solution of Helmholtz operators in Euclidean space $\R^d$. These are well-known and can be found in, for instance, \cite[p.~139]{Cialdea}.

\begin{thm}\label{thmg1n}Let $d\in\N$. Define
\begin{gather}
\mcg_\beta^{d,+}({\bf x},{\bf x}^\prime)=\frac{1}{(2\pi)^{\frac{d}{2}}}\left( \frac{\beta}{\|\bfx-\bfxp\|}\right)^{\frac{d}{2}-1} K_{\frac{d}{2}-1}(\beta\|\bfx-\bfxp\|),\label{thmg1np}\\
\mcg_\beta^{d,-}(\bfx, \bfxp) =\frac{{\rm i}}{4}\left(\frac{\beta}{2\pi\|\bfx-\bfxp\|}\right)^{\frac{d}{2}-1} H_{\frac{d}{2}-1}^{(1)}(\beta\|\bfx-\bfxp\|),\label{thmg1nm}
\end{gather}
then $\mcg_\beta^{d,\pm}$ respectively are fundamental solutions for $\big({-}\Delta\pm\beta^2\big)$ in Euclidean space $\R^d$, where~$\Delta$ is the Laplace operator on $\R^d$.
\end{thm}

Note most authors only present the above theorem for the case $d\ge 2$ but it is easily verified to also be valid for the case $d=1$ as well.
Fundamental solutions for $-\Delta$ on $\Hi_R^d$, $\Si_R^d$ (opposite antipodal) are respectively (see \cite[Theorem~3.1]{CohlKalII}, \cite[equation~(2.4)]{CohlPalmer})
\begin{gather}
\mch_R^d({\bf x},{\bf x}^\prime):= \frac{{\rm e}^{-{\rm i}\pi(\frac{d}{2}-1)}}{(2\pi)^{\frac{d}{2}}R^{d-2}\sinh^{\frac{d}{2}-1}\rho_h}Q_{\frac{d}{2}-1}^{\frac{d}{2}-1}(\cosh\rho_h),\label{thmh1deq}\\
\mathcal A_R^d ({\bf x},{\bf x}^\prime):=\frac{(d-2)!}{(2\pi)^{\frac{d}{2}}R^{d-2}\sin^{\frac{d}{2}-1}\rho_s}{\sf Q}_{\frac{d}{2}-1}^{1-\frac{d}{2}}(\cos\rho_s),\label{thms1deq}
\end{gather}
where the geodesic distances are given by (\ref{diststandard}), (\ref{diststandards}). These fundamental solutions will be useful to constrain the possible homogeneous solutions which are to be used to obtain our fundamental solutions in the $\beta\to0^{+}$ limit (when they exist, see Section~\ref{PropertiesLaplaceHelmholtz}).

Due to the fact that the spaces $\Hi_R^d$ and $\Si_R^d$ are homogeneous with respect to their iso\-met\-ry groups ${\rm SO}(d,1)$ and ${\rm O}(d)$ respectively, and therefore isotropic manifolds, without loss of generality, we are free to map the point $\bfxp\in\Hi_R^d$ or $\bfxp\in\Si_R^d$ to the origin. In this case, the global geodesic distance function $\rho_h$ coincides with the radial parameter~$r$ in geodesic polar coordinates, and~$\rho_s$ with~$\theta$; and therefore we may interchange $r$ with $\rho_h$ and $\theta$ with~$\rho_s$ accordingly (cf.~(\ref{diststandard}) with $r^\prime=0$) in our representation of a fundamental solution for the Helmholtz equation on each manifold. Notice that we can add any homogeneous solution to a~fundamental solution of the Helmholtz equations (\ref{genhyplap}), (\ref{gensphlap}), respectively, and still have a fundamental solution of Helmholtz operators. This is because a~fundamental solution of the Helmholtz equation must satisfy
\begin{gather*}
\int_{\Hi_R^d} \left(-\Delta\pm\beta^2\right)(\varphi_h(\bfxp))\mch_R^d(\bfx,\bfxp) \mathrm{d}\mathrm{vol}_h^\prime = \varphi_h(\bfx),\\
\int_{\Si_R^d} \left(-\Delta\pm\beta^2\right)(\varphi_s(\bfxp))\mathcal A_R^d(\bfx,\bfxp) \mathrm{d}\mathrm{vol}_s^\prime = \varphi_s(\bfx),
\end{gather*}
for all $\varphi_h\in {\mathcal D}(\Hi_R^d)$, $\varphi_s\in {\mathcal D}(\Si_R^d)$, where ${\mathcal D}$ is the space of test functions, and ${\rm d}\mathrm{vol}_{h,s}^\prime$ are the Riemannian (volume) measures on $\Hi_R^d$, $\Si_R^d$ respectively in the primed coordinates.

\subsection[Fundamental solutions of Helmholtz operators on $\Hi_R^d$]{Fundamental solutions of Helmholtz operators on $\boldsymbol{\Hi_R^d}$}\label{fshoHiRd}

\begin{thm}\label{thmh1d} Let $d\in\{2,3,\ldots\}$. Define $\mch_{R, \beta}^{d,+}\colon (\Hi_R^d\times\Hi_R^d)\setminus\{(\bfx,\bfx)\colon \bfx\in\Hi_R^d\}\to\R$ as
\begin{gather*}
\mch_{R, \beta}^{d,+}(\bfx,\bfx^\prime) :=\frac{{\rm e}^{-{\rm i}\pi(\frac{d}{2}-1)}}{(2\pi)^{\frac{d}{2}}R^{d-2}(\sinh\rho_h)^{\frac{d}{2}-1}}Q_{-{\frac{1}{2}}+{\frac{1}{2}}\sqrt{(d-1)^2+4\beta^2R^2}}^{\frac{d}{2}-1}(\cosh\rho_h),
\end{gather*}
where $\rho_h := \cosh^{-1} ([\wbfx, \wbfxp] )$ is the geodesic distance between $\wbfx$ and $\wbfxp$ on the pseudo-sphere of unit radius $\Hi^d$, with $\wbfx=\bfx/R$, $\wbfxp=\bfxp/R$. Then $\mch_{R, \beta}^{d,+}$ is a fundamental solution for the Helmholtz operator $\big({-}\Delta+\beta^2\big)$ on~$\Hi_R^d$.
\end{thm}
\begin{rem}\label{refereerem} It has been brought to our attention by one of the referees that a fundamental solution for the Helmholtz equation $\mch_{1, \beta}^{d,+}(\bfx,\bfx^\prime)$ on the unit-radius hyperboloid $\Hi^d$ is given in \cite[Theorem~3.3]{Matsumoto2001}. Note that the half-space model of hyperbolic geometry is used there, rather than the hyperboloid model which is used in this paper.
\end{rem}
\begin{proof}Our derivation for a fundamental solution of the Helmholtz equation $\mch_{R, \beta}^{d,+}(\bfx,\bfx^\prime)$ on the $R$-radius hyperboloid $\Hi_R^d$ is as follows. By starting with the spherically symmetric solution
\begin{gather*}
\frac{1}{\sinh^{\frac{d}{2}-1}r} Q_{-{\frac{1}{2}}+{\frac{1}{2}}\sqrt{(d-1)^2+4\beta^2R^2}}^{\frac{d}{2}-1}(\cosh r),
\end{gather*}
which is singular at $r=0$, to the homogeneous the Helmholtz equation, we have
\begin{gather*}
\mch_{R,\beta}^{d,+} (\bfx, \bfxp) = \frac{c}{\sinh^{\frac{d}{2}-1}r}Q_{-{\frac{1}{2}}+{\frac{1}{2}}\sqrt{(d-1)^2+4\beta^2R^2}}^{\frac{d}{2}-1}(\cosh \rho_h).
\end{gather*}
The constant $c$ is obtained by matching to a Euclidean fundamental solution of the Helmholtz equation in the flat-space limit $R\to\infty$ (see \cite[Section~2.4]{CohlPalmer}) with geodesic distance $R\rho_h$ and thus $\rho_h\to 0^+$. Using (\ref{iQuaa}), $\nu =-{\tfrac{1}{2}}+{\tfrac{1}{2}}\sqrt{(d-1)^2+4\beta^2R^2} \sim\beta R \to \infty$, as $R\to\infty$ and $\mu=\frac{d}{2}-1$, we have
\begin{gather}
\mch_{R,\beta}^{d,+}(\bfx,\bfxp) \sim \frac{c}{\sinh^{\frac{d}{2}-1}(\rho)}{\rm e}^{{\rm i}\pi\mu}\nu^\mu \left(\frac{\rho_h}{\sinh\rho_h}\right)^{{\frac{1}{2}}} K_\mu \left((\nu+{\tfrac{1}{2}})\rho_h\right)\nonumber\\
\hphantom{\mch_{R,\beta}^{d,+}(\bfx,\bfxp)}{} \sim c {\rm e}^{{\rm i}(\frac{d}{2}-1)\pi}\beta^{\frac{d}{2}-1}\frac{R^{d-2}}{(R\rho)^{\frac{d}{2}-1}} K_{\frac{d}{2}-1}(\beta R \rho)\nonumber\\
\hphantom{\mch_{R,\beta}^{d,+}(\bfx,\bfxp)}{} \sim c {\rm e}^{{\rm i}(\frac{d}{2}-1)\pi}\beta^{\frac{d}{2}-1}\frac{R^{d-2}}{\|\bfx-\bfxp\|^{\frac{d}{2}-1}} K_{\frac{d}{2}-1}(\beta\|\bfx-\bfxp\|),\label{huaa}
\end{gather}
\noindent where we have used $\sinh\rho_h\sim\rho_h$ as $\rho_h\to 0$. By equating (\ref{huaa}) with $\mcg_\beta^{d,+}$ and solving for $c$, we obtain $c = {\rm e}^{-{\rm i}\pi(\frac{d}{2}-1)} (2\pi)^{-\frac{d}{2}}R^{2-d}$, which completes the proof.
\end{proof}

\begin{rem}For $d=2$, the particular choice of $\beta=\sqrt{3}/(2R)$, a fundamental solution for this Helmholtz operator $\mch_{R,\beta}^{d,+}$ reduces to an expression involving complete elliptic integrals of the first and second kind through \cite[equations~(14.5.24)--(14.5.27)]{NIST:DLMF}. This is also true for even $d\ge 2$ with different choices of $\beta$, namely
\begin{gather*}
\beta=\frac{\sqrt{(n+1)^2-(d-1)^2}}{2R},
\end{gather*}
where $n$ is an odd-integer, so that the associated Legendre function of the second kind becomes a toroidal harmonic. All toroidal harmonics can be obtained, for instance, by using \cite[equations~(14.5.24)--(14.5.27)]{NIST:DLMF} as starting points for the recurrence relations for associated Legendre functions of the second kind such as \cite[equations~(14.10.3), (14.10.6) and (14.10.7)]{NIST:DLMF}.
\end{rem}

\begin{thm}\label{thmh1dminus}Let $d\in\{2,3,\ldots\}.$ Define $\mch_{R, \beta}^{d,-}\colon (\Hi_R^d\times\Hi_R^d)\setminus\{(\bfx,\bfx)\colon \bfx\in\Hi_R^d\}\to\R$ as
\begin{gather*}
\mch_{R, \beta}^{d,-}(\bfx,\bfx^\prime) :=\frac{{\rm e}^{-{\rm i}\pi(\frac{d}{2}-1)}}{(2\pi)^{\frac{d}{2}}R^{d-2}(\sinh\rho_h)^{\frac{d}{2}-1}}Q_\nu^{\frac{d}{2}-1}(\cosh\rho_h),
\end{gather*}
where $\rho_h := \cosh^{-1}\left([\wbfx, \wbfxp]\right)$ is the geodesic distance between $\wbfx$ and $\wbfxp$ on the pseudo-sphere of unit radius~$\Hi^d$,
\begin{gather}
\nu=
\begin{cases}
-{\tfrac{1}{2}}+{\tfrac{1}{2}}\sqrt{(d-1)^2-4\beta^2R^2}, & \mathrm{if} \ (d-1)^2-4\beta^2R^2 \geq 0,\vspace{1mm}\\
-{\tfrac{1}{2}}-\frac{{\rm i}}{2}\sqrt{4\beta^2R^2-(d-1)^2}, & \mathrm{if} \ (d-1)^2-4\beta^2R^2 \le 0,
\end{cases}\label{nuhyperboloidminus}
\end{gather}
$\wbfx=\bfx/R$, $\wbfxp=\bfxp/R$. Then $\mch_{R, \beta}^{d,-}$ is a fundamental solution for the Helmholtz operator $\big({-}\Delta- \beta^2\big)$ on $\Hi_R^d$.
\end{thm}

\begin{proof}Our derivation for a fundamental solution of the Helmholtz equation $\mch_{R, \beta}^{d,-}(\bfx,\bfx^\prime)$ on the $R$-radius hyperboloid $\Hi_R^d$ is as follows. In the flat-space limit $R\to\infty$, we must start with a solution such that $4\beta^2R^2>(d-1)^2$. Therefore we look for spherically symmetric conical solutions of the form
\begin{gather*}
\frac{1}{\sinh^{\frac{d}{2}-1}r} Q_{-{\frac{1}{2}}\pm {\tfrac{{\rm i}}{2}}\sqrt{4\beta^2R^2-(d-1)^2}}^{\frac{d}{2}-1}(\cosh r),
\end{gather*}
(singular at $r=0$), of the homogeneous Helmholtz equation. Therefore we must decide which conical form
\begin{gather*}
\mch_{R,\beta}^{d,-} (\bfx, \bfxp) = \frac{c}{\sinh^{\frac{d}{2}-1}r}Q_{-{\frac{1}{2}}\pm{\tfrac{{\rm i}}{2}}\sqrt{4\beta^2R^2-(d-1)^2}}^{\frac{d}{2}-1}(\cosh \rho_h),
\end{gather*}
where $4\beta^2R^2>(d-1)^2$, is appropriate. Using Theorem~\ref{thmg1n}, we know the flat-space limit should produce a Hankel function of the first kind $H_{\frac{d}{2}-1}^{(1)}$, and therefore from Theorem \ref{thm:1}, we see that the solution with the minus sign in front of the imaginary unit is necessary. The constant~$c$ is obtained by matching to a Euclidean fundamental solution of the Helmholtz equation in the flat-space limit $R\to\infty$ with geodesic distance $R\rho_h$ and thus $\rho_h\to 0^+$. In the flat-space limit $\tau\sim\beta R$ and $\rho_h\sim r/R$, where~$r$ is the geodesic distance between two points in Euclidean space~$\R^d$. Using~(\ref{Dun.eq22}) (with the plus sign in the order), as $R\to\infty$ and $\mu=\frac{d}{2}-1$, we have
\begin{gather}
\mch_{R,\beta}^{d,-}(\bfx,\bfxp) \sim \frac{{\rm i}\pi c {\rm e}^{{\rm i}\pi\mu}\tau^\mu}{2(\sinh\rho_h)^{\frac{d}{2}-1}}\sqrt{\frac{\rho_h}{\sinh\rho_h}}H_\mu^{(1)}(\tau\rho_h)\nonumber\\
\hphantom{\mch_{R,\beta}^{d,-}(\bfx,\bfxp)}{} \sim \frac{c{\rm i}\pi {\rm e}^{{\rm i}\pi(\frac{d}{2}-1)}\beta^{\frac{d}{2}-1} R^{d-2}}{2r^{\frac{d}{2}-1}}H^{(1)}_{\frac{d}{2}-1}(\beta r)\nonumber\\
\hphantom{\mch_{R,\beta}^{d,-}(\bfx,\bfxp)}{} = \frac{c{\rm i}\pi {\rm e}^{{\rm i}\pi(\frac{d}{2}-1)}\beta^{\frac{d}{2}-1} R^{d-2}}{2\|\bfx-\bfxp\|^{\frac{d}{2}-1}}H^{(1)}_{\frac{d}{2}-1}(\beta \|\bfx-\bfxp\|).\label{huaam}
\end{gather}
\noindent By equating (\ref{huaam}) with $\mcg_\beta^{d,-}$ in Theorem \ref{thmg1n} and solving for $c$, we obtain
\begin{gather*}
c = \frac{{\rm e}^{-{\rm i}(\frac{d}{2}-1)\pi}}{(2\pi)^{\frac{d}{2}}R^{d-2}},
\end{gather*}
which proves the full behavior of $\mch_{R,\beta}^{d,-}$ for $4\beta^2R^2>(d-1)^2$. For the case $4\beta^2R^2\le (d-1)^2$, we match up to a fundamental solution of $-\Delta$ on $\Hi_R^d$ (\ref{thmh1deq}) as $\beta\to 0$ which requires a fundamental solution with functional dependence as follows
\begin{gather*}
\mch_{R,\beta}^{d,-}(\bfx,\bfxp)=\frac{g}{\sinh^{\frac{d}{2}-1}\rho_h} Q_{-{\frac{1}{2}}+{\frac{1}{2}}\sqrt{(d-1)^2-4\beta^2R^2}}^{\frac{d}{2}-1}(\cosh \rho_h).
\end{gather*}
Matching the constants $c$ and $g$ at $\beta=(d-1)/(2R)$, completes the proof.
\end{proof}

\subsubsection[Uniqueness of fundamental solutions for $\big({-}\Delta\pm\beta^2\big)$ on $\Hi_R^d$]{Uniqueness of fundamental solutions for $\boldsymbol{\big({-}\Delta\pm\beta^2\big)}$ on $\boldsymbol{\Hi_R^d}$}\label{Uniquenessoffundamentalsolutionintermsofdecayatinfinity}

It is clear that, in general, a fundamental solution of the Helmholtz equation in the hyperboloid model of hyperbolic geometry ${\mathcal H}_{R,\beta}^{d,\pm}$ is not unique since one can add any homogeneous solution of the Helmholtz equations $h_\pm\colon {\bf H}_R^d\to{\mathbf R}$ to ${\mathcal H}_{R,\beta}^{d,\pm}$ and still obtain solutions to~(\ref{hypeq3}) since $h_\pm$ is in the kernel of $\big({-}\Delta\pm\beta^2\big)$.

\begin{prop}\label{propuniquelaplace} There exists precisely one $C^\infty$-function $H_\pm\colon (\Hi_R^d\times\Hi_R^d)\setminus\{(\bfx,\bfx)\colon \bfx\in\Hi_R^d\}\to\R$ such that for all $\bfxp\in\Hi_R^d$ the function
$(H_\pm)_\bfxp\colon \Hi_R^d\setminus\{\bfxp\}\to\R$ defined by $(H_\pm)_\bfxp(\bfx):=(H_\pm)(\bfx,\bfxp)$ is a function on $\Hi_R^d$ with $\big({-}\Delta\pm\beta^2\big) (H_\pm)_\bfxp=\delta_g(\cdot,\bfxp)$, and
\begin{gather}
\lim_{d_h(\bfx,\bfxp)\to\infty}(H_\pm)_\bfxp(\bfx)=0,\label{limitequationunique}
\end{gather}
 where $d_h(\bfx,\bfxp)$ is the geodesic distance between two points $\bfx,\bfxp\in\Hi_R^d$.
\end{prop}
\begin{proof} Existence is clear. For uniqueness, suppose $H_\pm$ and $\tilde{H}_\pm$ are two such functions. Let $\bfxp\in\Hi_R^d$. Define the $C^\infty$-function $h_\pm\colon \Hi_R^d\setminus\{\bfxp\}\to\R$ by $h_\pm=(H_\pm)_\bfxp-({\tilde{H}_\pm)}_\bfxp$. Then $h_\pm$ is a~function on $\Hi_R^d$ with $\big({-}\Delta\pm\beta^2\big) h_\pm=0$. Since $\Hi_R^d$ is locally Euclidean, one has by local elliptic regularity that~$h_\pm$ can be extended to a $C^\infty$-function $\hat{h}_\pm\colon \Hi_R^d\to\R$. It follows from~(\ref{limitequationunique}) for~$H_\pm$ and~$\tilde{H}_\pm$ that
\begin{gather}
\lim_{d_{h,s}(\bfx,\bfxp)\to\infty}\hat{h}_\pm(\bfx)=0.\label{limitequationunique2}
\end{gather}
The strong elliptic maximum/minimum principle on a Riemannian manifold for a bounded domain $\Omega$ states that if $u$ is a homogeneous solution to an elliptic operator, then the supremum/infimum of $u$ in $\Omega$ coincides with the supremum/infimum of~$u$ on the boundary~$\partial\Omega$. By using a compact exhaustion sequence~$\Omega_k$ in a non-compact connected Riemannian manifold and passing to a subsequence $\bfx_k\in\partial\Omega_k$ such that $\bfx_k\to\infty$, the strong elliptic maximum/minimum principle can be extended to non-compact connected Riemannian manifolds with boundary conditions at infinity (see for instance \cite[Section~8.3.2]{Grigor}). Taking $\Omega_k\subset\Hi_R^d$, the strong elliptic maximum/minimum principle for non-compact connected Riemannian manifolds implies using~(\ref{limitequationunique2}) that $\hat{h}_\pm=0$. Therefore $h_\pm=0$ and $H_\pm(\bfx,\bfxp)=\tilde{H}_\pm(\bfx,\bfxp)$ for all $\bfx\in\Hi_R^d\setminus\{\bfxp\}$.
\end{proof}

By Proposition \ref{propuniquelaplace}, for $d\ge 2$, the functions $\mch_{R,\beta}^{d,\pm}$ are the unique fundamental solutions for these Helmholtz operators which satisfy the vanishing decay~(\ref{limitequationunique}).

\subsection[Fundamental solutions of Helmholtz operators $\big({-}\Delta\pm\beta^2\big)$ on $\Si_R^d$]{Fundamental solutions of Helmholtz operators $\boldsymbol{\big({-}\Delta\pm\beta^2\big)}$ on $\boldsymbol{\Si_R^d}$}\label{fssrd}

It has been discussed in Section~\ref{PropertiesLaplaceHelmholtz} that on $\Si_R^d$ there exists fundamental solutions of Helmholtz operators with a~single Dirac delta distribution at the origin. Such fundamental solutions should have the properties that they are singular at the origin and bounded at the opposite pole of the hypersphere. Since homogeneous solutions of the Helmholtz equation are of the form $(\sin\rho_s)^{-\mu}\big[c_1^{(\pm,\pm)}{\sf P}_\nu^{\pm \mu}(\pm\cos\rho_s) +c_2^{(\pm,\pm)}{\sf Q}_\nu^{\pm\mu}(\pm \cos\rho_s)\big]$, where $\mu=\frac{d}{2}-1$ (see Section~\ref{SepVarStaHyp}), one should study
the behavior of these homogeneous solutions as $x\to -1^{+}$. In fact, ${\sf P}_\nu^{-\mu}(-x)$ is the unique solution to the associated Legendre differential equation with the required behavior as $x\to -1^{+}$ (see Section~\ref{SomepropertiesFerrers}). In particular, $(\sin\rho_s)^{-\mu}{\sf P}_\nu^{-\mu}(-x)$ has the required spherically symmetric regular behavior on the opposite pole.

Furthermore, as was seen in Section \ref{PropertiesLaplaceHelmholtz}, a fundamental solution of the Helmholtz equation on~$\Si_R^d$ must also satisfy a total integral normalization property (\ref{hyperspheresinglenorm}), due to applying the divergence theorem. An alternative integration constraint due to the divergence theorem on a~hypersphere can be considered as follows. Consider the compact manifold with boundary by a~ball of radius~$R\epsilon$, $0<\epsilon\ll1$, embedded within $\Si_R^d$ with center at the origin $\theta=0$. Let $\bfX=\nabla {\mathcal S}_{R,\beta}^{d,\pm}(\bfx,\bfxp)$ and through~(\ref{spheq3}), one has
\begin{gather*}
\operatorname{div} \bfX=\Delta{\mathcal S}_{R,\beta}^{d,\pm}(\bfx,\bfxp)=-\delta_s(\bfx,\bfxp)\pm \beta^2{\mathcal S}_{R,\beta}^{d,\pm}(\bfx,\bfxp).
\end{gather*}
Now use the divergence theorem (\ref{divergencethm}), and one obtains what we refer to as the $\epsilon$-ball integral constraint, namely
\begin{gather}
-1\pm\beta^2\int_{\Si_{R\epsilon}^d}{\mathcal S}_{R,\beta}^{d,\pm}(\bfx,\bfxp)\mathrm{d}V =\int_{\Si_{R\epsilon}^{d-1}}\left[\frac{1}{R}\frac{\partial}{\partial\theta} {\mathcal S}_{R,\beta}^{d,\pm}(\bfx,\bfxp)\right]_{\theta=\epsilon}
\mathrm{d}{\tilde V},\label{epsballintcon}
\end{gather}
which can be verified directly. This is because for some differentiable function $u(\theta)$ on $\Si_R^d$, $\theta\in(0,\pi)$, the normal derivative in standard geodesic polar coordinates (\ref{standardsph}) is
\begin{gather*}
\langle\nabla u(\theta),N\rangle=\frac{1}{R}\frac{\partial u(\theta)}{\partial \theta}.
\end{gather*}

Apart from these local behaviors, fundamental solutions of Helmholtz operators on $\Si_R^d$ may or may not have the desired limiting behaviors as $\beta\to0$ or in the flat-space limit $R\to\infty$. For instance, there does not exist a~fundamental solution of Laplace's equation on $\Si_R^d$ with non-vanishing total integral. Hence for fundamental solutions of a single Dirac delta distribution~$\mss_{R,\beta}^{d,\pm}$~(\ref{spheq3}), the \mbox{$\beta\to0$} limit should not exist. However, for an opposite antipodal fundamental solution~$\mas_{R,\beta}^{d,\pm}$~(\ref{oaspheq3}), the $\beta\to0$ limit should exist, namely~(\ref{thms1deq})
\begin{gather*}
\lim_{\beta\to0}\mas_{R,\beta}^{d,\pm}=\mathcal A_R^d.
\end{gather*}

One would also expect that since all manifolds are locally Euclidean, that a fundamental solution of Helmholtz operators should look locally like a Euclidean fundamental solutions of Helmholtz operators (\ref{thmg1np}), (\ref{thmg1nm}), namely
\begin{gather*}
\mss_{R,\beta}^{d,\pm}\stackrel{R\to\infty}{\sim}\mcg_\beta^{d,\pm},\qquad \mas_{R,\beta}^{d,\pm}\stackrel{R\to\infty}{\sim}\mcg_\beta^{d,\pm}.
\end{gather*}
Since the Sommerfeld radiation condition requirement does not apply to the Helmholtz operator $\big({-}\Delta+\beta^2\big)$, for this operator, problems associated with this requirement should not arise. In fact, we will see the above requirements for this operator is easily satisfied.

\subsubsection[Fundamental solutions on $\Si_R^d$ for $\big({-}\Delta+\beta^2\big)$]{Fundamental solutions on $\boldsymbol{\Si_R^d}$ for $\boldsymbol{\big({-}\Delta+\beta^2\big)}$}\label{fsonsrdmdmb2}

In analogy to solutions of a damped pendulum, for the operator under study in this section, the following has been brought to our attention by one of the referees. If $\beta^2$ is not constrained to be either positive or negative, physicists would call the cases:~(1) $4\beta^2R^2-(d-1)^2<0$, ``underdamped'' or ``oscillatory''; (2) $4\beta^2R^2-(d-1)^2>0$, ``overdamped''; and (3) $4\beta^2R^2-(d-1)^2=0$, ``critically damped''. A fundamental solution for $\big({-}\Delta+\beta^2\big)$ will oscillate increasingly on $\Si_R^d$ when $\beta^2$ is made increasingly negative, so that $4\beta^2R^2-(d-1)^2<0$ is increasingly negative. However if $\beta^2$ is made sufficiently positive, the oscillations in a fundamental solution of $\big({-}\Delta+\beta^2\big)$ will disappear.

A fundamental solution with a single Dirac delta distribution at the origin for $\big({-}\Delta+\beta^2\big)$ on~$\Si_R^d$ is given in the following theorem.

\begin{thm}\label{singDiracplus}Let $d\in\{2,3,\ldots\}$. Define $\mss_{R,\beta}^{d,+}\colon \big(\Si_R^d\times\Si_R^d\big)\setminus\big\{(\bfx,\bfx)\colon \bfx\in\Si_R^d\big\}\to\R$ as
\begin{gather}
\mss_{R,\beta}^{d,+}(\bfx,\bfx^\prime) :=\frac{\Gamma(\nu+\mu+1) \Gamma(\mu-\nu)}{2^{\frac{d}{2}+1}\pi^{\frac{d}{2}}R^{d-2}(\sin\rho_s)^{\frac{d}{2}-1}}\sP_{\nu}^{-\mu}(- \cos\rho_s),\label{singDiracpluseqn}
\end{gather}
 where
\begin{gather}
\mu=\tfrac{d}{2}-1, \qquad \nu=
\begin{cases}
-{\tfrac{1}{2}}+{\tfrac{1}{2}}\sqrt{(d-1)^2-4\beta^2R^2}, & \mathrm{if} \ (d-1)^2-4\beta^2R^2 \geq 0,\vspace{1mm}\\
-{\tfrac{1}{2}}+\frac{{\rm i}}{2}\sqrt{4\beta^2R^2-(d-1)^2}, & \mathrm{if} \ (d-1)^2-4\beta^2R^2 \le 0,
\end{cases}\label{munuDiracplus}
\end{gather}
$\rho_s := \cos^{-1}((\wbfx, \wbfxp))$ is the geodesic distance between $\wbfx$ and $\wbfxp$ on the hypersphere of unit radius~$\Si^d$, $\wbfx=\bfx/R$, $\wbfxp=\bfxp/R$. Then $\mss_{R, \beta}^{d,+}$ is a fundamental solution
for the Helmholtz operator $\big({-}\Delta+ \beta^2\big)$ on $\Si_R^d$.
\end{thm}

\begin{proof}Apart from a constant in $\nu$, $\mu$, we guess a form of a fundamental solution for this Helmholtz operator for a single Dirac delta distribution at the origin on~$\Si_R^d$. Hence our ansatz will be
\begin{gather*}
\mss_{R,\beta}^{d,+}(\bfx,\bfx^\prime) :=\frac{c(\nu,\mu) {\sf P}_{\nu}^{-\mu}(-x)}{\big(1-x^2\big)^{\frac{\mu}{2}}},
\end{gather*}
where $x=\cos\rho_s$, and $c(\nu,\mu)$ is constant in $x$. This constant, for instance, can be determined through the integral normalization requirement~(\ref{hyperspheresinglenorm}). We must verify
\begin{gather}
\int_{\Si_R^d}\mss_{R,\beta}^{d,+}(\bfx,\bfx^\prime)\mathrm{d}\mathrm{vol}_s^\prime= \frac{1}{\beta^2}.\label{normalizationintegralplus}
\end{gather}
Starting with the left-hand side of (\ref{normalizationintegralplus}), and using (\ref{Mellininquestion}), noting the standard volume integral
\begin{gather}
\int_{\Si_R^{d-1}} \mathrm{d}{\rm vol}_s'=\frac{2\pi^{\frac{d}{2}}}{\Gamma\big(\frac{d}{2}\big)}R^{d-1},\label{integraldm1hypers}
\end{gather}
after simplifying, one requires that
\begin{gather*}
c(\nu ,\mu )=\frac{\Gamma (\nu +\mu +1)\Gamma (\mu -\nu )}{2^{\frac{d}{2}+1}\pi ^{\frac{d}{2}}R^{d-2}}.
\end{gather*}
All that remains is to demonstrate that (\ref{singDiracpluseqn}) produces~(\ref{thmg1np}) in the flat-space limit, $R\to\infty$. In the flat-space limit $\rho_s\sim \frac{\vartheta}{R}$ and $\nu\sim -\tfrac{1}{2}+{\rm i}\beta R$, with $\vartheta \to \|\bfx-\bfxp\|$. Using the uniform asymptotics given by~(\ref{FcPmmumx}), the demonstration is validated. This completes the proof of Theorem~\ref{singDiracplus}.
\end{proof}

\begin{rem}Note that (\ref{derivopppole}) gives the strength of the cusp on the opposite pole (at $\theta=\pi$) for fundamental solutions on hyperspheres, such as that given in Theorem~\ref{singDiracplus}.
\end{rem}

\begin{rem}If one considers the $\beta\sim0$ approximation for (\ref{singDiracpluseqn}), we see that as $\beta\to0$,
\begin{gather*}
\mss_{R,\beta}^{d,+}(\bfx,\bfxp)\sim\mathcal A_R^d (\bfx,\bfxp)+\frac{\Gamma\big(\frac{d+1}{2}\big)} {2\pi^{\frac{d+1}{2}}R^d\beta^2}.
\end{gather*}
In order to obtain this, we have used \cite[equation~(14.5.18)]{NIST:DLMF}, the binomial and small angle approximations for trigonometric functions. Hence $\lim\limits_{\beta\to0}\mss_{R,\beta}^{d,+}(\bfx,\bfxp)=\infty$, and as one expects (see Section~\ref{PropertiesLaplaceHelmholtz}) the $\beta\to 0$ (Laplace) limit of $\mss_{R,\beta}^{d,+}(\bfx,\bfx^\prime)$ does not exist.
\end{rem}

\begin{rem}\label{epsilonballcondp}The $\epsilon$-ball integral constraint (\ref{epsballintcon}) for $\mss_{R,\beta}^{d,+}$ is satisfied. This can be verified by noting that as $\rho_s\to 0^{+}$,
\begin{gather*}
\mss_{R,\beta}^{d,+}(\bfx,\bfxp)\sim \frac{\Gamma\big(\frac{d}{2}-1\big)} {4\pi^{\frac{d}{2}}(R\rho_s)^{d-2}},\qquad \frac{1}{R}\frac{\partial}{\partial\rho_s} \mss_{R,\beta}^{d,+}(\bfx,\bfxp)
\sim -\frac{\Gamma(\frac{d}{2})} {2\pi^{\frac{d}{2}}(R\rho_s)^{d-1}},
\end{gather*}
and using (\ref{integraldm1hypers}).
\end{rem}

A fundamental solution with two opposite antipodal Dirac delta distributions, one at the origin and another on the opposite pole of~$\Si_R^d$, for $\big({-}\Delta+\beta^2\big)$ is given in the following theorem.

\begin{thm}\label{antiDiracplus} Let $d\in\{2,3,\ldots\}$. Define $\mas_{R,\beta}^{d,+}\colon \big(\Si_R^d\times\Si_R^d\big)\setminus\big\{(\bfx,\bfx)\colon \bfx\in\Si_R^d\big\}\to\R$ by
\begin{gather}
\mas_{R,\beta}^{d,+}(\mathbf{x},\mathbf{x}^{\prime}):=\frac{\Gamma (\nu +\mu +1)\Gamma (\mu -\nu )}{2^{\frac{d}{2}+1}\pi^{\frac{d}{2}}R^{d-2}(\sin \rho _{s})^{\frac{d}{2}-1}}{\mathsf f}_{\nu}^{-\mu}(\cos\rho _{s}),
\label{twoantiDiracpluseqn}
\end{gather}
where $\mu$ and $\nu$ are given by \eqref{munuDiracplus}, $\rho_s = \cos^{-1} ((\wbfx, \wbfxp) )$ is the geodesic distance between~$\wbfx$ and~$\wbfxp$ on the hypersphere of unit radius~$\Si^d$, $\wbfx=\bfx/R$,
$\wbfxp=\bfxp/R$, and ${\mathsf f}_{\nu}^{\mu}$ is the odd Ferrers function defined by~\eqref{FPodddefn}. Then $\mas_{R, \beta}^{d,+}$ is an opposite antipodal fundamental solution for the Helmholtz operator $\big({-}\Delta+ \beta^2\big)$ on~$\Si_R^d$.
\end{thm}

\begin{proof}Since the opposite antipodal fundamental solution of the Helmholtz operator $\big({-}\Delta+\beta^2\big)$ must be an odd function, it should be given through~(\ref{FPodddefn}). Starting with~(\ref{twoantiDiracpluseqn}) we then perform the flat-space limit $R\to\infty$ by using the asymptotics given by~(\ref{CPodd1}), along with $\sin(\pi(\nu-\mu))\sim-\tfrac12{\rm e}^{\pi\tau}{\rm e}^{\pm{\rm i} \pi\mu}$; as in the proof of Theorem~\ref{singDiracplus}, we must consider that $\rho_s\sim \vartheta/R$, $\vartheta\sim\|\bfx-\bfxp\|$, $\bfx,\bfxp\in\R^d$. Note that the asymptotic contribution due to $-{\sf P}_\nu^{-\mu}(\cos\rho_s)$ is negligible as compared to ${\sf P}_\nu^{-\mu}(- \cos\rho_s)$. One can then easily see that~(\ref{twoantiDiracpluseqn}) correctly approaches~(\ref{thmg1np}). Furthermore, since~(\ref{twoantiDiracpluseqn}) is an odd function about $\theta=\tfrac{\pi}{2}$
on the hypersphere~$\Si_R^d$, the integral normalization requirement~(\ref{hyperspheresinglenorm}) is clearly satisfied. This completes the proof.
\end{proof}

\begin{rem} If one considers the $\beta\sim0$ approximation for (\ref{twoantiDiracpluseqn}), we see that as $\beta\to0$,
\[
\mas_{R,\beta}^{d,+}(\bfx,\bfxp)\sim \mathcal A_R^d (\bfx,\bfxp)=\frac{(d-2)!}{(2\pi)^{\frac{d}{2}}R^{d-2}\sin^{\frac{d}{2}-1}\rho_s}{\sf Q}_{\frac{d}{2}-1}^{1-\frac{d}{2}}(\cos\rho_s),
\]
the opposite antipodal fundamental solution of Laplace's equation on the hypersphere~(\ref{thms1deq}). In order to demonstrate this limit, we have used the binomial, small angle approximations for trigonometric functions, as well as~(\ref{Pnummumxconnection}).
\end{rem}

\begin{rem} The $\epsilon$-ball integral constraint (\ref{epsballintcon}) for $\mas_{R,\beta}^{d,+}$ is also satisfied. This can be verified in an identical fashion to Remark~\ref{epsilonballcondp} while also noting that ${\sf P}_\nu^{-\mu}(x)$ and its derivative vanish as $x\to1^{-}$, $\Re\mu>0$.
\end{rem}

\begin{rem}Note that the integral normalization requirement (\ref{hyperspheresinglenorm}) is satisfied trivially since~$\mas_{R,\beta}^{d,+}$, is an odd function of $\rho_s$ about $\tfrac{\pi}{2}$.
\end{rem}

\subsubsection[Candidate fundamental solutions on $\Si_R^d$ for $\big({-}\Delta- \beta^2\big)$]{Candidate fundamental solutions on $\boldsymbol{\Si_R^d}$ for $\boldsymbol{\big({-}\Delta- \beta^2\big)}$}\label{cfsSiRdmDmb2}

\begin{defn}\label{defnminussph} Let $d\in\{2,3,\ldots\}$, $R,\beta>0$. Define the {\it candidate functions}
\begin{gather*}
\msss_{R,\beta}^{d,-}\colon \ \big(\Si_R^d\times\Si_R^d\big)\setminus \big\{(\bfx,\bfx)\colon \bfx\in\Si_R^d\big\}\to\R,\\
\mssf_{R,\beta}^{d,-}\colon \ \big(\Si_R^d\times\Si_R^d\big)\setminus \big\{(\bfx,\bfx)\colon \bfx\in\Si_R^d\big\}\to\C,
\end{gather*}
as
\begin{gather}
\msss_{R,\beta}^{d,-}(\mathbf{x},\mathbf{x}^{\prime}):=\frac{\Gamma (\nu +\mu +1)\Gamma (\mu -\nu )}{2^{\frac{d}{2}+1}\pi ^{\frac{d}{2}}R^{d-2}(\sin \rho_{s})^{\frac{d}{2}-1}}\,\mathsf{P}_{\nu}^{-\mu
}(-\cos \rho _{s}), \label{msssRbetadm} \\
\mssf_{R,\beta}^{d,-}(\mathbf{x},\mathbf{x}^{\prime}):=\frac{\Gamma (\nu +\mu +1)}{\Gamma (\nu-\mu +1)R^{d-2}(2\pi )^{\frac{d}{2}}(\sin \rho _{s})^{\frac{d}{2}-1}} \left( \mathsf{Q}_{\nu}^{-\mu}(\cos \rho _{s})+{\rm i}
\frac{\pi}{2}\mathsf{P}_{\nu}^{-\mu}(\cos \rho _{s})\right) \notag\\
\hphantom{\mssf_{R,\beta}^{d,-}(\mathbf{x},\mathbf{x}^{\prime})}{} =\frac{\Gamma (\nu +\mu +1)\Gamma (\mu -\nu )}{2^{\frac{d}{2}+1}\pi ^{\frac{d}{2}}R^{d-2}(\sin \rho _{s})^{\frac{d}{2}-1}}\left(
\mathsf{P}_{\nu}^{-\mu}(-\cos \rho _{s})-{\rm e}^{{\rm i}\pi (\nu -\mu )}\mathsf{P}_{\nu}^{-\mu}(\cos \rho _{s})\right), \label{mssfRbetadm}
\end{gather}
where
\begin{gather}
\mu=\tfrac{d}{2}-1,\qquad \nu=-{\tfrac{1}{2}}+{\tfrac{1}{2}}\sqrt{4\beta^2R^2+(d-1)^2}, \label{munuforminussinglesphereDirac}
\end{gather}
and $\rho_s = \cos^{-1}([\wbfx, \wbfxp])$ is the geodesic distance between $\wbfx$ and $\wbfxp$ on the hypersphere of unit radius $\Si^d$, with $\wbfx=\bfx/R$, $\wbfxp=\bfxp/R$.
\end{defn}

\begin{rem}\label{ssrem1} Note that apart from a sign difference (due to our choice of a positive Lapla\-ce--Beltrami operator $-\Delta$) and the inclusion of the radius of curvature $R$ of the hypersphere $\Si_R^d$, the candidate function $\msss_{R,\beta}^{d,-}(\bfx,\bfxp)$ is identical to that which was found by Liu \& Ryan~(2002) \cite[equation~(5.1)]{LiuRyan2002} and Szmytkowski~(2007) \cite[equation~(3.23)]{Szmytkowskiunitsphere07} in terms of Gegenbauer functions (cf.~(\ref{Dun2eq12})) for unit radius hyperspheres. Note that in order to convert the wavenumber given in~\cite{Szmytkowskiunitsphere07}, so that it is the same as given here, one must take $\beta^2R^2=\lambda(\lambda+d-1)$, and therefore take
\begin{gather*}
\lambda=-\frac{d-1}{2}\pm\frac12 \sqrt{(d-1)^2+4\beta^2R^2},
\end{gather*}
with the plus sign chosen to match the formulae given here.
\end{rem}

\begin{rem}\label{ssrem2}The candidate function $\msss_{R,\beta}^{d,-}$ satisfies the integral normalization require\-ment (\ref{hyperspheresinglenorm}), namely
\begin{gather*}
\int_{\Si_R^d} \msss_{R,\beta}^{d,-}(\bfx,\bfx^\prime)\mathrm{d}\mathrm{vol}_s^\prime= -\frac{1}{\beta^2}.
\end{gather*}
This follows straightforwardly using (\ref{Mellininquestion}), (\ref{msssRbetadm}), with~(\ref{munuforminussinglesphereDirac}).
\end{rem}

\begin{rem}\label{ssrem3}The $\epsilon$-ball integral constraint~(\ref{epsballintcon}) for $\msss_{R,\beta}^{d,-}$ is also satisfied. This can be verified by noting that as $\rho_s\to 0^{+}$,
\begin{gather*}
\msss_{R,\beta}^{d,-}(\bfx,\bfxp)\sim \frac{\Gamma\big(\frac{d}{2}-1\big)} {4\pi^{\frac{d}{2}}(R\rho_s)^{d-2}},\qquad \frac{1}{R}\frac{\partial}{\partial\rho_s} \msss_{R,\beta}^{d,-}(\bfx,\bfxp) \sim -\frac{\Gamma(\frac{d}{2})}{2\pi^{\frac{d}{2}}(R\rho_s)^{d-1}},
\end{gather*}
and using (\ref{integraldm1hypers}).
\end{rem}

\begin{rem}\label{ssrem4}If one considers the small $\beta$ approximation for $\msss_{R,\beta}^{d,-}(\bfx,\bfxp)$ we see that as $\beta\to0$,
\begin{gather*}
\msss_{R,\beta}^{d,-}(\bfx,\bfxp)\sim \mss_{R}^{d}(\bfx,\bfxp)- \frac{\Gamma\big(\frac{d+1}{2}\big)}{2\pi^{\frac{d+1}{2}}R^d\beta^2}.
\end{gather*}
In order to obtain this we have used \cite[equation~(14.5.18)]{NIST:DLMF}, the binomial and small angle approximations for trigonometric functions. Hence $\lim\limits_{\beta\to0}\msss_{R,\beta}^{d,-}(\bfx,\bfxp)=\infty$, and as one would expect (see Section~\ref{PropertiesLaplaceHelmholtz}) the $\beta\to 0$ (Laplace) limit of $\msss_{R,\beta}^{d,-}(\bfx,\bfx^\prime)$ does not exist.
\end{rem}

\begin{rem}\label{ssrem5}The flat-space limit $R\to\infty$ of the function $\msss_{R,\beta}^{d,-}$ should not exist. In fact, it oscillates wildly in this limit, as it approaches the following function
\begin{gather*}
\msss_{R,\beta}^{d,-}\sim \frac{{\rm i}}{4} \left(\frac{\beta}{2\pi\|\bfx-\bfxp\|}\right)^{\frac{d}{2}-1}\left( {\rm i}\cot\big(\pi\beta R+\tfrac{\pi(1-d)}{2}\big)J_{\frac{d}{2}-1}(\beta\|\bfx-\bfxp\|)
+{\rm i}Y_{\frac{d}{2}-1}(\beta\|\bfx-\bfxp\|)\right),
\end{gather*}
where $\bfx,\bfxp\in\R^d$.
\end{rem}

On the other hand, the candidate function $\mssf_{R,\beta}^{d,-}(\bfx,\bfxp)$ does approach the correct function in the flat-space $R\to\infty$ limit~(\ref{thmg1nm}), but does not obey the integral normalization
requirement~(\ref{hyperspheresinglenorm}).

\begin{thm}In the flat-space limit $R\to\infty$, the function $\mssf_{R,\beta}^{d,-}$ approaches the Euclidean fundamental solution of $\big({-}\Delta- \beta^2\big)$, namely \eqref{thmg1nm}
\begin{gather*}
\mssf_{R,\beta}^{d,-}(\bfx,\bfxp)\sim \mcg_\beta^{d,-}(\bfx, \bfxp) = \frac{{\rm i}}{4}\left(\frac{\beta}{2\pi\|\bfx-\bfxp\|}\right)^{\frac{d}{2}-1} H_{\frac{d}{2}-1}^{(1)}(\beta\|\bfx-\bfxp\|).
\end{gather*}
\end{thm}

\begin{proof}In the flat-space limit $R\to\infty$, with geodesic distance $R\rho_s$ bounded, and thus $\rho_s\to 0$. Using (\ref{Fuaa}), (\ref{thmg1nm}),
one has $\mu=\frac{d}{2}-1$, $\nu=-\frac12+\frac12\sqrt{(d-1)^2+4\beta^2R^2}\sim\beta R$, and as $R\to\infty$, $\rho_s\to 0^{+}$, that $\sin\rho_s\sim\rho_s$, $R\rho_s\sim\|\bf x-\bf x'\|$, and
\begin{gather*}
\mssf_{R,\beta}^{d,-}(\bfx,\bfxp)\sim\mcg_\beta^{d,-}(\bfx,\bfxp),
\end{gather*}
since $(\rho_s/\sin\rho_s)\sim 1$ as $\rho_s\to 0$. This completes the proof.
\end{proof}

\begin{rem}The candidate function $\mssf_{R,\beta}^{d,-}$ does not satisfy the integral normalization requirement (\ref{hyperspheresinglenorm}), namely
\begin{gather*}
\int_{\Si_R^d}\mssf_{R,\beta}^{d,-}(\bfx,\bfx^\prime)\mathrm{d}\mathrm{vol}_s^\prime= -\frac{1}{\beta^2}\big(1-{\rm e}^{{\rm i}\pi(\nu-\mu)}\big).
\end{gather*}
This follows straightforwardly using~(\ref{mssfRbetadm}), (\ref{Mellininquestion}) with (\ref{munuforminussinglesphereDirac}).
\end{rem}

\begin{rem}In the limit $\beta\to 0$, the candidate function $\mssf_{R,\beta}^{d,-}$ approaches a Laplace fundamental solution, namely~(\ref{thmg1nm}), except that it differs from~(\ref{thms1deq}) by a purely imaginary constant (this satisfies the homogeneous Laplace equation), namely
\begin{gather*}
\mssf_{R,0}^{d,-}(\bfx,\bfxp)=\mss_R^d(\bfx,\bfxp)+ \frac{{\rm i} \Gamma\big(\frac{d-1}{2}\big)}{4\pi^{\frac{d-1}{2}}R^{d-2}}.
\end{gather*}
\end{rem}

\begin{defn}In contrast to the situation for the operator $\big({-}\Delta +\beta ^{2}\big)$ on ${\mathbf{S}}_{R}^{d}$, we may also define the opposite antipodal candidate functions,
\begin{gather*}
\mass_{R,\beta}^{d,-}(\mathbf{x},\mathbf{x}^{\prime}):=\frac{\Gamma (\nu +\mu +1)\Gamma (\mu -\nu )}{2^{\frac{d}{2}+1}\pi ^{\frac{d}{2}}R^{d-2}(\sin \rho _{s})^{\frac{d}{2}-1}}{\mathsf f}_{\nu}^{-\mu}(\cos\rho _{s}),\\
\masf_{R,\beta}^{d,-}(\mathbf{x},\mathbf{x}^{\prime}):=\frac{\Gamma (\nu +\mu +1)\Gamma (\mu -\nu )\big(1+{\rm e}^{{\rm i}\pi (\nu -\mu)}\big) }{2^{\frac{d}{2}+1}\pi ^{\frac{d}{2}}R^{d-2}(\sin \rho _{s})^{\frac{d}{2}-1}}
{\mathsf f}_{\nu}^{-\mu}(\cos\rho _{s}),
\end{gather*}
where $\masf_{R,\beta}^{d,-}$ is constructed by converting (\ref{mssfRbetadm}) into an odd function using (\ref{FPodddefn}), and the variable definitions are as in Definition~\ref{defnminussph}.
\end{defn}

\begin{rem}The opposite antipodal candidate function $\mass_{R,\beta}^{d,-}(\bfx,\bfx^\prime)$ approaches the Laplace fundamental solution as $\beta\to 0^{+}$, namely (\ref{thms1deq})
\begin{gather*}
\lim_{\beta\to0^{+}} \mass_{R,\beta}^{d,-}(\bfx,\bfx^\prime) =\mathcal A_R^d ({\bf x},{\bf x}^\prime)=\frac{(d-2)!}{(2\pi)^{\frac{d}{2}}R^{d-2}\sin^{\frac{d}{2}-1}\rho_s}{\sf Q}_{\frac{d}{2}-1}^{1-\frac{d}{2}}(\cos\rho_s),
\end{gather*}
and in the flat-space limit it approaches
\begin{gather*}
\mass_{R,\beta}^{d,-}\sim \frac{{\rm i}}{4} \left(\frac{\beta}{2\pi\|\bfx-\bfxp\|}\right)^{\frac{d}{2}-1}\left(H_\mu^{(1)}(\beta\|\bfx-\bfxp\|)- \big(1+{\rm i}\tan(\tfrac{\pi}{2}(\nu-\mu)\big) J_\mu(\beta\|\bfx-\bfxp\|)\right).
\end{gather*}
\end{rem}

\begin{rem}The opposite antipodal candidate function $\masf_{R,\beta}^{d,-}(\bfx,\bfx^\prime)$ approaches the Laplace fundamental solution as $\beta\to 0^{+}$, namely~(\ref{thms1deq})
\begin{gather*}
\lim_{\beta\to0^{+}} \masf_{R,\beta}^{d,-}(\bfx,\bfx^\prime) =\mathcal A_R^d ({\bf x},{\bf x}^\prime)=\frac{(d-2)!}{(2\pi)^{\frac{d}{2}}R^{d-2}\sin^{\frac{d}{2}-1}\rho_s}{\sf Q}_{\frac{d}{2}-1}^{1-\frac{d}{2}}(\cos\rho_s),
\end{gather*}
and in the flat-space limit it approaches
\begin{gather*}
\masf_{R,\beta}^{d,-}\sim \frac{{\rm i}}{4} \left(\frac{\beta}{2\pi\|\bfx-\bfxp\|}\right)^{\frac{d}{2}-1}\left(H_\mu^{(1)}(\beta\|\bfx-\bfxp\|)-{\rm e}^{{\rm i}\pi(\nu-\mu)}J_\mu(\beta\|\bfx-\bfxp\|)\right).
\end{gather*}
\end{rem}

\begin{rem} Note that for both functions, the integral normalization requirement (\ref{hyperspheresinglenorm}) is satisfied trivially since $\mass_{R,\beta}^{d,-}$, $\masf_{R,\beta}^{d,-}$ are odd functions of $\rho_s$ about $\tfrac{\pi}{2}$.
\end{rem}

\begin{Conjecture}\label{singDiracminus} Let $d\in\{2,3,\ldots\}$. Then the candidate function $\msss_{R,\beta}^{d,-}(\bfx,\bfxp)$ is a fundamental solution for the Helmholtz operator $\big({-}\Delta-\beta^2\big)$ on~$\Si_R^d$. Similarly, the candidate function $\mass_{R,\beta}^{d,-}(\bfx,\bfx^\prime)$ is an opposite antipodal fundamental solution for the Helmholtz operator $\big({-}\Delta-\beta^2\big)$ on~$\Si_R^d$.
\end{Conjecture}

Conjecture~\ref{singDiracminus} seems quite plausible since, unlike $\mssf_{R,\beta}^{d,-}$, $\msss_{R,\beta}^{d,-}$ satisfies all the critical Remarks~\ref{ssrem1}--\ref{ssrem4}. Apart from having the correct flat-space limit, $\mssf_{R,\beta}^{d,-}$ does not satisfy the integral normalization requirement~(\ref{hyperspheresinglenorm}) which is not allowed, as shown in Section~\ref{PropertiesLaplaceHelmholtz}. The only critical property that
$\mssf_{R,\beta}^{d,-}$ satisfies is the correct Euclidean flat-space limit. The flat-space behavior of~$\msss_{R,\beta}^{d,-}$, displayed in Remark~\ref{ssrem5}, may well be due to imposing incorrect boundary conditions since~$\Si_R^d$ is a compact manifold without boundary, whereas Euclidean space is actually a~noncompact manifold with boundary (at infinity). Furthermore, $\mssf_{R,\beta}^{d,-}$ is partially `antipodal' to the source point (diverges at the source point as well as at the antipole), which is not a~desirable property. It was also pointed out by one of the referees that due to the $\Gamma(\mu-\nu)$ factor, $\msss_{R,\beta}^{d,-}$ will diverge when $\beta=0$, as is necessary (see Remark~\ref{ssrem3}), and at an infinite number of points in the underdamped (oscillatory) regime which correspond to the wavenumber being one of the eigenvalues of the Laplace--Beltrami operator on~$\Si_R^d$. Similarly, $\mass_{R,\beta}^{d,-}$ has all the right properties that an opposite antipodal fundamental solution of the Helmholtz operator $\big({-}\Delta-\beta^2\big)$ on~$\Si_R^d$ should satisfy, with the exception of the correct flat-space limit. Given a resolution of the flat-space limit for $\msss_{R,\beta}^{d,-}$, the same problem will be solved for~$\mass_{R,\beta}^{d,-}$. The ultimate resolution of this Conjecture~\ref{singDiracminus} must also be associated with a correct and consistent formulation of the Sommerfeld radiation condition~(\ref{sommerfeldcond}) on~$\Si_R^d$.

\section{Gegenbauer expansions in geodesic polar coordinates}\label{Gegenbauerexpansioninhyperbolichypersphericalcoordinates}

For convenience, define the following degrees
\begin{gather*}\nu_{ +}:=
\begin{cases}
-{\tfrac{1}{2}}+{\tfrac{1}{2}}\sqrt{(d-1)^2-4\beta^2R^2}, & \mathrm{if} \ (d-1)^2-4\beta^2R^2 \geq 0,\vspace{1mm}\\
-{\tfrac{1}{2}}+\frac{{\rm i}}{2}\sqrt{4\beta^2R^2-(d-1)^2}& \mathrm{if} \ (d-1)^2-4\beta^2R^2 \le 0,
\end{cases}\\
\nu_{ -}:=-{\tfrac{1}{2}}+{\tfrac{1}{2}}\sqrt{(d-1)^2+4\beta^2R^2}.
\end{gather*}
\noindent In the boundary case, $(d-1)^2-4\beta^2R^2=0$, or equivalently,
\begin{gather*}
\beta\mapsto\frac{d-1}{2R}, \qquad \nu_{ +}\mapsto-\frac12, \qquad \nu_{ -}\mapsto\frac{\sqrt{2}(d-1)-1}{2},
\end{gather*}
the wavenumber only has dependence on the dimension of the space $d$ and the radius of curvature of the manifold $R$. In this case, the Ferrers and associated Legendre functions reduce to elementary or elementary transcendental functions. For instance, if $\mu=\frac{d}{2}-1 \in \Z$ ($d$~even), then the fundamental solutions are given in terms of complete elliptic integrals (the Legendre functions are toroidal harmonics, and the equivalent for the Ferrers functions). Alternatively, if $\mu$ is a half odd integer ($d$~odd), then the Ferrers functions and associated Legendre functions reduce to trigonometric and hyperbolic functions respectively. Given starting points for various degrees and orders, a~lattice of degrees and orders which satisfy these properties can be obtained using recurrence relations for these functions, namely
\cite[equations~(14.10.3), (14.10.6) and (14.10.7)]{NIST:DLMF}, \cite[equations~(14.10.1) and (14.10.2)]{NIST:DLMF}.

In fact, recent results due to Maier \cite[Theorem~6.1]{Maier2016} have shown that the space of Ferrers and associated Legendre functions is much more rich than was previously expected. In fact if the orders are integers and the degrees differ by $\pm 1/r$ $(r=2,3,4,6)$ from an integer, then the resulting Legendre functions are also given in terms of complete elliptic integrals of the first and second kind!

Define the constants
\begin{gather*}
\mathsf{a}_{R}^{d}:=\frac{\Gamma \big(\frac{d}{2}\big)}{2(d-2)\pi ^{\frac{d}{2} }R^{d-2}},\qquad \mathsf{b}_{R,\beta }^{d,\pm }:=\frac{2^{\mu }\Gamma (\mu)\Gamma (\nu _{ \pm } +\mu +1)\Gamma (\mu -\nu _{ \pm } )}{2^{\frac{d}{
2}+2}\pi ^{\frac{d}{2}}R^{d-2}}.
\end{gather*}
\begin{thm}\label{Gegfundsolexpansionsthm} Let $d\ge 3$, $R>0$, $\beta\in\R$, $r,r'\in (0,\infty)$, $r\ne r'$, $\theta,\theta'\in(0,\pi)$, $\gamma\in[0,\pi]$. Then
\begin{gather}
\mch_{R,\beta}^{d,+}(\bfx,\bfxp)= \frac{{\rm e}^{-{\rm i}\pi(\frac{d}{2}-1)}{\sf a}_{R}^{d}}{(\sinh r\sinh r')^{\frac{d}{2}-1}}\label{hgegenexpansionp} \\
\hphantom{\mch_{R,\beta}^{d,+}(\bfx,\bfxp)=}{} \times\sum_{l=0}^\infty (-1)^l(2l + d - 2)P_{\nu_{+}}^{-(\frac{d}{2} - 1 +l)}(\cosh r_<)Q_{\nu_{+}}^{\frac{d}{2}-1+l}(\cosh r_>)
C_l^{\frac{d}{2}-1}(\cos\gamma),\nonumber\\
\mch_{R,\beta}^{d,-}(\bfx,\bfxp)=\frac{{\rm e}^{-{\rm i}\pi(\frac{d}{2}-1)}{\sf a}_{R}^{d}}{(\sinh r\sinh r')^{\frac{d}{2}-1}} \label{hgegenexpansionm}\\
\hphantom{\mch_{R,\beta}^{d,-}(\bfx,\bfxp)=}{} \times\sum_{l=0}^\infty (-1)^l (2l + d - 2) P_{\nu_{-}}^{-(\frac{d}{2} - 1 +l)}(\cosh r_<)Q_{\nu_{-}}^{\frac{d}{2}-1+l}(\cosh r_>)C_l^{\frac{d}{2}-1}(\cos\gamma),\nonumber\\
\mss_{R,\beta}^{d,+}(\bfx,\bfxp)=\frac{{\sf b}_{R,\beta}^{d,+}}{(\sin\theta\sin\theta')^{\frac{d}{2}-1}}\sum_{l=0}^{\infty}(2l+d-2)\big(\nu_{ +} +\tfrac{d}{2}\big)_l\big(\tfrac{d}{2}-1-\nu_{ +}\big)_l\nonumber \\
\hphantom{\mss_{R,\beta}^{d,+}(\bfx,\bfxp)=}{} \times \sP_{\nu_{+}}^{-(\frac{d}{2}-1+l)}(\cos\theta_<)\sP_{\nu_{+}}^{-(\frac{d}{2}-1+l)}(- \cos\theta_>)C_l^{\frac{d}{2}-1}(\cos\gamma),\label{sgegenexpansionp}\\
\mas_{R,\beta}^{d,+}(\bfx,\bfxp)=\frac{{\sf b}_{R,\beta}^{d,+}}{(\sin\theta\sin\theta')^{\frac{d}{2}-1}}\sum_{l=0}^{\infty}(2l+d-2)\left(\nu_{ +} +\tfrac{d}{2}\right)_l\big(\tfrac{d}{2}-1-\nu_{ +}\big)_l\nonumber \\
\hphantom{\mas_{R,\beta}^{d,+}(\bfx,\bfxp)=}{} \times \sP_{\nu_{+}}^{-(\frac{d}{2}-1+l)}(\cos\theta_<){\mathsf f}_{\nu_{+}}^{-(\frac{d}{2}-1+l)}(\cos\theta_>)
C_l^{\frac{d}{2}-1}(\cos\gamma),\label{sgegenexpansionap}\\
\msss_{R,\beta}^{d,-}(\bfx,\bfxp)=\frac{{\sf b}_{R,\beta}^{d,-}}{(\sin\theta\sin\theta')^{\frac{d}{2}-1}}\sum_{l=0}^{\infty}(2l+d-2)\big(\nu_{ -} +\tfrac{d}{2}\big)_l \big(\tfrac{d}{2}-1-\nu_{ -}\big)_l\nonumber \\
\hphantom{\msss_{R,\beta}^{d,-}(\bfx,\bfxp)=}{} \times \sP_{\nu_{-}}^{-(\frac{d}{2}-1+l)}(\cos\theta_<)\sP_{\nu_{-}}^{-(\frac{d}{2}-1+l)}(- \cos\theta_>)C_l^{\frac{d}{2}-1}(\cos\gamma),\label{sgegenexpansionm1}\\
\mssf_{R,\beta}^{d,-}(\bfx,\bfxp)=\frac{{\sf a}_R^d \Gamma\big(\nu_{-} +\frac{d}{2}\big)}{\Gamma\big(\nu_{-} -\frac{d}{2}+2\big)(\sin\theta\sin\theta')^{\frac{d}{2}-1}}\nonumber\\
\hphantom{\mssf_{R,\beta}^{d,-}(\bfx,\bfxp)=}{} \times
\sum_{l=0}^{\infty}(-1)^l(2l+d-2)\big(\nu_{-} +\tfrac{d}{2}\big)_l \big({-}\nu_{-} +\tfrac{d}{2}-1\big)_l \sP_{\nu_{-}}^{-(\tfrac{d}{2}-1+l)}(\cos\theta_<)\\
\hphantom{\mssf_{R,\beta}^{d,-}(\bfx,\bfxp)=}{} \times
\left[\sQ_{\nu_{-}}^{-(\frac{d}{2}-1+l)}(\cos\theta_>)+{\rm i}\frac{\pi}{2}\sP_{\nu_{-}}^{-(\frac{d}{2}-1+l)}(\cos\theta_>)\right]
C_l^{\frac{d}{2}-1}(\cos\gamma),\label{sgegenexpansionm2}
\end{gather}
where for uniform convergence, $\theta\ne\theta'$, $r\ne r'$, and for \eqref{sgegenexpansionp}--\eqref{sgegenexpansionm2}, require \eqref{Dun2eq21}. In~\eqref{hgegenexpansionm}, ${\rm i}$~should be replaced with $-{\rm i}$ for $4\beta^2R^2> (d-1)^2$, in $\nu_{-}$ $($see~\eqref{nuhyperboloidminus}$)$.
\end{thm}

\begin{proof}The Gegenbauer expansions for $\mch_{R,\beta}^{d,\pm}$, $\mss_{R,\beta}^{d,+}$, $\mas_{R,\beta}^{d,+}$, $\msss_{R,\beta}^{d,-}$, $\mssf_{R,\beta}^{d,-}$, follow immediately from (\ref{conj}), (\ref{conjPmPm})--(\ref{conjPmPmmx}).
\end{proof}

\begin{rem}Similar Gegenbauer expansions for $\mass_{R,\beta}^{d,-}$, $\masf_{R,\beta}^{d,-}$, can be easily obtained using (\ref{conjPmPm}), (\ref{conjPmPmmx}). We leave these to the reader.
\end{rem}

\begin{thm} In the flat-space limit, the fundamental solution Gegenbauer expansions for $\mch_{R,\beta}^{d,\pm}$, $\mss_{R,\beta}^{d,+}$, $\mas_{R,\beta}^{d,+}$, approach the corresponding Euclidean fundamental solution expansions, namely {\rm \cite[equations~(11.41.4) and (11.41.8)]{Watson}}
\begin{gather}
\frac{H_{\mu}^{(1)}(\beta\|\bfx-\bfxp\|)}{\|\bfx-\bfxp\|^{\mu}}=\frac{2^{\mu}\Gamma(\mu)}{(\beta\|\bfx\|\|\bfxp\|)^{\mu}}\sum_{l=0}^{\infty}(l+\mu)J_{\mu+l}(\beta\|\bfx\|_<)H_{\mu+l}^{(1)}(\beta\|\bfx\|_>)
C_l^\mu(\cos\gamma),\label{eusolexpm}\\
\frac{K_{\mu}(\beta\|\bfx-\bfxp\|)}{\|\bfx-\bfxp\|^{\mu}}=\frac{2^{\mu}\Gamma(\mu)}{(\beta\|\bfx\|\|\bfxp\|)^{\mu}}\sum_{l=0}^{\infty}(l+\mu)I_{\mu+l}(\beta\|\bfx\|_<)K_{\mu+l}(\beta\|\bfx\|_>)C_l^\mu(\cos\gamma).
\label{eusolexpp}
\end{gather}
\end{thm}
\begin{proof} We have already shown in the flat-space limit that $\mch_{R,\beta}^{d,\pm}$, $\mss_{R,\beta}^{d,+}$, $\mas_{R,\beta}^{d,+}$ $\mapsto \mcg_{\beta}^{d,\pm}$ (see proofs of Theorems~\ref{thmh1d}, \ref{thmh1dminus}, \ref{singDiracplus}, \ref{antiDiracplus}). Now we only need to check that in the flat-space limits, the expansions on the right-hand sides of Theorem~\ref{Gegfundsolexpansionsthm} approach the correct series coefficients.
Using standard geodesic polar coordinates (\ref{standardhyp}) or (\ref{standardsph}), in the flat-space limit $(r,r')\sim (\bar{r},\bar{r}')/R$, and $(\theta,\theta')\sim(\bar{r},\bar{r}')/R$, where $\bar{r}:=\|\bfx\|$, $\bar{r}':=\|\bfxp\|$, are the radial coordinates for the primed and unprimed position vectors in $d$-dimensional Euclidean space $\bfx,\bfxp\in\R^d$. Let $\mu=\frac{d}{2}-1$. In the flat-space limit
\begin{gather*}
{\sf a}_R^d\sim\frac{\Gamma(\mu)}{4\pi^{\frac{d}{2}}R^{d-2}}, \qquad \nu_{ -} \sim-\tfrac{1}{2}+\beta R\sim\beta R,\\ \nu_{ +} \sim-\tfrac{1}{2}+{\rm i}\beta R, \qquad {\sf b}_{R,\beta}^{d,+}\sim\frac{\beta^{2\mu}
\Gamma(\mu){\rm e}^{-\pi\beta R}}{4\pi^{\frac{d}{2}-1}}.
\end{gather*}
By applying the appropriate flat-space limits, the Gegenbauer expansions (\ref{hgegenexpansionp}), (\ref{sgegenexpansionp}), (\ref{sgegenexpansionap}), all approach (\ref{eusolexpp}) using (\ref{Gammaratio}), (\ref{Gammaratioz}), (\ref{negratiogamma}), (\ref{iPuaa}), (\ref{iQuaa}), (\ref{Fcuab}), (\ref{Fcuaa}), (\ref{CPodd1}). Similarly, the Gegenbauer expansions (\ref{hgegenexpansionm}), (\ref{sgegenexpansionm2}) both approach (\ref{eusolexpm}) using (\ref{Dun.eq16}), (\ref{Dun.eq21}), (\ref{Fuaa}), (\ref{Fuab}).
\end{proof}
\begin{rem}
For (\ref{sgegenexpansionm1}), the flat-space limit of this expansion does not exist.
\end{rem}
\begin{rem} Notice that (\ref{hgegenexpansionp})--(\ref{sgegenexpansionm2}) can be further expanded over the remaining $(d-2)$-quantum numbers in $K$ in terms of a simply separable product of normalized hyperspherical harmonics $Y_l^K(\wbfx)\overline{Y_l^K(\wbfxp)}$, where $\wbfx,\wbfxp\in\Si^{d-1}$, using the addition theorem for hyperspherical harmonics \cite[Chapter~9]{AAR}
\begin{gather*}
C_l^{\frac{d}{2}-1}(\cos\gamma)=\frac{2\pi^{\frac{d}{2}}(d-2)}{(2l+d-2)\Gamma\big(\frac{d}{2}\big)} \sum_{K}Y_l^K(\wbfx)\overline{Y_l^K(\wbfxp)}.\end{gather*}
\end{rem}

\subsection{Azimuthal Fourier series in two dimensions}\label{AzimuthalFourierseries}

It is interesting to consider how we can obtain azimuthal Fourier expansions for our fundamental solutions. For the particular case where a fundamental solution is given in terms of an associated Legendre function of the second kind, we can use the addition theorem~(\ref{conj}) for the two-dimensional case. One has the following azimuthal Fourier expansions in $d=2$.

\begin{cor}\label{Fourfundsolexpansionsthm}Let $\beta\in\R$, $r,r'\in (0,\infty)$, $r\ne r'$, $\theta,\theta',\phi-\phi'\in[0,\pi]$. Then
\begin{gather}
\mch_{R,\beta}^{2,+}(\bfx,\bfxp)=\sum_{l=0}^\infty(-1)^l\epsilon_l P_{\nu_{+}}^{-l}(\cosh r_<)Q_{\nu_{+}}^{l}(\cosh r_>)T_l(\cos(\phi-\phi')),\label{hfourenexpansionp}\\
\mch_{R,\beta}^{2,-}(\bfx,\bfxp)=\sum_{l=0}^\infty (-1)^l \epsilon_lP_{\nu_{-}}^{-l}(\cosh r_<) Q_{\nu_{-}}^{l}(\cosh r_>)T_l(\cos(\phi-\phi')),\label{hfourenexpansionm}\\
\mss_{R,\beta}^{2,+}(\bfx,\bfxp)=\frac{-1}{4\sin(\pi\nu_{+})}\nonumber\\
\hphantom{\mss_{R,\beta}^{2,+}(\bfx,\bfxp)=}{}\times \sum_{l=0}^{\infty}\epsilon_l(-\nu_{ +})_l(\nu_{ +} +1)_l\sP_{\nu_{+}}^{-l}(\cos\theta_<)\sP_{\nu_{+}}^{-l}(- \cos\theta_>)T_l(\cos(\phi-\phi')),\label{sfourenexpansionp}\\
\mas_{R,\beta}^{2,+}(\bfx,\bfxp)=\frac{-1}{4\sin(\pi\nu_{+})} \!\sum_{l=0}^{\infty}\epsilon_l(-\nu_{ +})_l(\nu_{ +} +1)_l\sP_{\nu_{+}}^{-l}(\cos\theta_<){\mathsf f}_{\nu_{+}}^{-l}(\cos\theta_>)T_l(\cos(\phi-\phi')),\!\!\!\!\!
\label{sfourenexpansionp+}\\
\mss_{R,\beta}^{2,-}(\bfx,\bfxp)=\sum_{l=0}^{\infty}(-1)^l\epsilon_l\sP_{\nu_{-}}^{-l}(\cos\theta_<)\left[\sQ_{\nu_{-}}^{l}(\cos\theta_>)+{\rm i}\frac{\pi}{2}\sP_{\nu_{-}}^{l}(\cos\theta_>)\right]
T_l(\cos(\phi-\phi')). \label{sfourenexpansionm}
\end{gather}
For \eqref{hfourenexpansionm}, in $\nu_{ -}$, ${\rm i}$ should be replaced with $-{\rm i}$ for $4\beta^2R^2> (d-1)^2$ $($see~\eqref{nuhyperboloidminus}$)$.
\end{cor}

\begin{proof}This follows from (\ref{deq2addnthm}), (\ref{ferdep2addnthm})--(\ref{ferdepm2addnthm}).
\end{proof}

If $\beta=1/(2R)$ then the fundamental solution $\mch_{R,1/(2R)}^{2,-}$ reduces to a complete elliptic integral of the first kind through \cite[equations~(14.5.24)--(14.5.27)]{NIST:DLMF}, namely
\begin{gather*}
\mch_{R,1/(2R)}^{2,-}(\bfx,\bfxp)=\frac{1}{2\pi}\sech\big(\tfrac12\rho\big)K \big( \sech\big(\tfrac12\rho\big)\big)\\
\hphantom{\mch_{R,1/(2R)}^{2,-}(\bfx,\bfxp)}{} =\frac{1}{2\pi}\sum_{m=0}^\infty \epsilon_m (-1)^m P_{-\frac12}^{-m}(\cosh r_<)Q_{-\frac12}^{m}(\cosh r_>)\cos(m(\phi-\phi')).
\end{gather*}
Similar expansions are possible for the other fundamental solutions.

Two-dimensions is the only case where we are able to compute azimuthal Fourier expansions of our fundamental solutions whose coefficients are given without a sum. For dimensions greater than two, the approach used previously for Laplace's equation \cite[Section~4]{CohlKalII} does not easily generalize for the Helmholtz operator.

\appendix

\section{Proof of the addition Theorem \ref{GegexpansionsFerrers}}\label{ProofoftheAdditionTheoremGegexpansionsFerrers}

Since the proof of the addition Theorem \ref{GegexpansionsFerrers} is fairly long, we have delegated it to this appendix. We begin by studying the convergence of (\ref{conjPmPp})--(\ref{conjPmQm}) by determining the asymptotic behavior of the $n$th terms as $n\rightarrow \infty$. To this end, for (\ref{conjPmPm}) use \cite[equation~(14.15.1)]{NIST:DLMF}, and from this we deduce that as ${\Re}\mu \rightarrow \infty$,
\begin{gather}
\mathsf{P}_{\nu}^{-\mu} ( \cos \theta ) \sim \frac{\tan ^{\mu }\bigl( \frac{1}{2}\theta \bigr)}{\Gamma ( 1+\mu )},\qquad
\mathsf{P}_{\nu}^{-\mu}(- \cos \theta ) \sim \frac{\cot ^{\mu}\big( {{\frac{1}{2}}\theta}\big)}{\Gamma ( 1+\mu)}.\label{Dun2eq1}
\end{gather}
Then, from the large argument asymptotics for the gamma function (Stirling's formula \cite[equation~(5.11.3)]{NIST:DLMF}) we obtain
\begin{gather}
( n+\mu) ( \nu +\mu +1)_{n}( \mu -\nu) _{n}\mathsf{P}_{\nu}^{-( \mu +n)}\bigl( \cos \theta _{<}\bigr) \mathsf{P}_{\nu}^{-(\mu +n)
}\bigl( {\cos \theta _{>}}\bigr) \notag \\
\qquad{} \sim \frac{\big\{ {\tan \bigl( {{\frac{1}{2}}\theta _{<}}\bigr) \tan \bigl( {{\frac{1}{2}}\theta _{>}}\bigr)}\big\} ^{\mu +n}}{\Gamma (\nu +\mu +1) \Gamma (\mu -\nu)}.\label{Dun2eq3}
\end{gather}
Using the connection formula \cite[equation~(14.9.7)]{NIST:DLMF} and (\ref{Dun2eq1}) we have as ${\Re}\mu \rightarrow \infty$,
\begin{gather}
\mathsf{P}_{\nu}^{\mu} ( \cos \theta) \sim \frac{\Gamma( \nu +\mu +1) \Gamma (\mu -\nu)}{\pi \Gamma( \mu +1)}\big\{ {-}\sin (\nu \pi) \tan ^{\mu}\big( {{\tfrac{1}{2}}\theta}\big) +\sin ( \mu \pi) \cot
^{\mu}\bigl( {{\tfrac{1}{2}}\theta}\bigr) \big\} . \label{Dun2eq5}
\end{gather}
Hence, for (\ref{conjPmPp}), from (\ref{Dun2eq1}), (\ref{Dun2eq5}), and Stirling's formula, one obtains as $n\rightarrow \infty$,
\begin{gather}
( n+\mu) \mathsf{P}_{\nu}^{-(\mu +n)}\bigl( \cos \theta _{<}\bigr) \mathsf{P}_{\nu}^{\mu +n}( \cos \theta _{>}) \notag \\
\qquad{} \sim \frac{1}{\pi}\tan ^{\mu +n}\bigl( {{\tfrac{1}{2}}\theta
_{<}}\bigr) \big\{ {-}\sin ( {\nu \pi} ) \tan ^{\mu +n}\bigl( {{%
\tfrac{1}{2}}\theta _{>}}\bigr) + ( {-1} ) ^{n}\sin \left( {\mu
\pi}\right) \cot ^{\mu +n}\bigl( {{\tfrac{1}{2}}\theta _{>}}\bigr)\big\} . \label{Dun2eq6}
\end{gather}
For (\ref{conjPmQp}), using the connection formula \cite[equation~(14.9.9)]{NIST:DLMF}, one has as $n\rightarrow \infty$,
\begin{gather}
( n+\mu) \mathsf{P}_{\nu}^{-(\mu +n)}( \cos \theta _{<}) \mathsf{Q}_{\nu}^{\mu +n}(\cos \theta _{>}) \notag \\
\qquad{} \sim \frac{1}{2}\tan ^{\mu +n}\bigl( {{\tfrac{1}{2}}\theta _{<}}\bigr) \big\{ {-}\cos ( \nu \pi) \tan ^{\mu +n}\bigl( {{\tfrac{1}{2}}\theta _{>}}\bigr) +( {-1}) ^{n}\cos ( \mu\pi) \cot ^{\mu +n}\bigl( {{\tfrac{1}{2}}\theta _{>}}\bigr)\big\} . \label{Dun2eq8}
\end{gather}
Similarly for (\ref{conjPmQm}), using (\ref{Dun2eq1}), the connection formula \cite[equation~(14.9.10)]{NIST:DLMF}, and Stirling's formula, yields for $\mu-\nu \notin \N$
\begin{gather}
 (n+\mu )(\nu +\mu +1)_{n}(\mu -\nu )_{n}\mathsf{P}_{\nu}^{-( {\mu +n})}\bigl( {\cos \theta _{<}}\bigr) \mathsf{Q}_{\nu}^{-( {\mu +n})}( \cos \theta _{>}) \sim \frac{\pi \tan ^{\mu
+n} \bigl( {{\frac{1}{2}}\theta _{<}}\bigr)}{2 \Gamma ( \nu +\mu +1) \Gamma ( \mu -\nu)} \notag \\
\qquad{} \times \big\{ {\cot ( {(\nu -\mu )\pi} ) \tan ^{\mu +n}\bigl( {{\tfrac{1}{2}}\theta _{>}}\bigr) +( -1)^{n+1}\csc ( (\nu -\mu )\pi) \cot ^{\mu +n}\bigl( {{\tfrac{1}{2}}\theta _{>}}\bigr)}\big\} . \label{Dun2eq10}
\end{gather}
It remains to estimate the Gegenbauer polynomials of the first kind which appear in (\ref{conjPmPp})--(\ref{conjPmPm}). We were not able to find asymptotic results in the literature for large $n$ and complex parameter $\mu$. However the following estimate will suffice.

\begin{prop}\label{Dun2prop}For $\gamma \in (0,\pi )$ and $\mu \in \C$ bounded, one has as $n\rightarrow \infty $,
\begin{gather}
C_n^\mu(\cos\gamma) =\mathcal{O}\big( {n^{{\Re}\mu -1}}\big) . \label{Dun2propeq}
\end{gather}%
\end{prop}
\begin{proof}Assume first that ${\Re}\mu >0.$ Then for $\gamma \in (0,\pi )$ we use the following integral representation \cite[equation~(14.12.1)]{NIST:DLMF}
\begin{gather}
\mathsf{P}_{\mu +n-\frac{1}{2}}^{\frac{1}{2}-\mu}(\cos\gamma) =\frac{\sqrt{2} ( {\sin \gamma} ) ^{\frac{1}{2}-\mu}}{\sqrt{\pi}\Gamma ( \mu )}\int_{0}^{\gamma}\frac{\cos ( {(\mu +n) t})}{( {\cos t-\cos \gamma})
^{1-\mu}}\mathrm{d}t. \label{Dun2eq13}
\end{gather}
Then with $\cos \theta ={\frac{1}{2}}\big( {{\rm e}^{{\rm i}\theta}+{\rm e}^{-{\rm i}\theta}}\big) $, and substituting $t\mapsto \gamma -t$, one arrives at
\begin{gather}
 \int_{0}^{\gamma}\frac{\cos ( {(\mu +n) t} )}{ ( {\cos t-\cos \gamma} ) ^{1-\mu}} \mathrm{d}t \label{Dun2eq15}\\
 \qquad {} =\frac{{\rm e}^{{\rm i}(\mu +n) \gamma}}{2}\int_{0}^{\gamma}\frac{{\rm e}^{-{\rm i}\mu t}{\rm e}^{-{\rm i}nt}}{( {\cos ( {\gamma -t}) -\cos \gamma})^{1-\mu}}\mathrm{d}t
+\frac{{\rm e}^{-{\rm i}(\mu +n) \gamma}}{2}\int_{0}^{\gamma}\frac{{\rm e}^{{\rm i}\mu t}{\rm e}^{{\rm i}nt}}{( {\cos ( {\gamma -t}) -\cos \gamma}) ^{1-\mu}}\mathrm{d}t.\notag
\end{gather}
Thus from (\ref{Dun2eq13}), (\ref{Dun2eq15}), one has for $\gamma \in (0,\pi)$, ${\Re}\mu >0$,
\begin{gather}
\mathsf{P}_{\mu +n-\frac{1}{2}}^{\frac{1}{2}-\mu}(\cos\gamma) =\frac{\sqrt{2} ( {\sin \gamma} ) ^{\frac{1}{2}-\mu}}{\sqrt{\pi} \Gamma ( \mu )}\left[ {\int_{0}^{\infty}\phi
_{+} (t) {\rm e}^{{\rm i}nt}\,\mathrm{d}t+\int_{0}^{\infty}\phi
_{-} (t) {\rm e}^{-{\rm i}nt} \mathrm{d}t}\right] , \label{Dun2eq16}
\end{gather}
where $\phi _{\pm}\colon (0,\infty )\rightarrow {\mathbf{C}}$ are defined by
\begin{gather*}
\phi _{\pm} (t) :=\frac{{\rm e}^{\mp {\rm i}(\mu +n) \gamma }{\rm e}^{\pm {\rm i}\mu t}H ( {\gamma -t} )}{2 ( {\cos ( {\gamma -t} ) -\cos \gamma} ) ^{1-\mu}},
\end{gather*}
where $H(x) $ is the Heaviside function \cite[equation~(1.16.13)]{NIST:DLMF}. Next $\phi _{\pm}(t) \!\sim\! {\frac{1}{2}}{\rm e}^{\mp {\rm i}(\mu +n) \gamma} ( {t\sin \gamma} ) ^{\mu -1}$ as $t\rightarrow 0$ and therefore from the asymptotics of Fourier integrals \cite[Section~4.2, Theorem~1, p.~199]{Wong2001} applied to both integrals on the right-hand side of~(\ref{Dun2eq16}) we deduce for $\gamma \in (0,\pi )$, ${\Re}\mu >0$, $n\rightarrow \infty $,
\begin{gather*}
\mathsf{P}_{\mu +n-\frac{1}{2}}^{\frac{1}{2}-\mu}(\cos\gamma) =\mathcal{O}\big( {n^{-{\Re}\mu}}\big).
\end{gather*}%
The estimate (\ref{Dun2propeq}) then follows from (\ref{Dun2eq12}) with $\lambda =n$ and Stirling's formula. Finally, to extend this to bounded $\mu$ lying in the left half plane, we have from \cite[equation~(18.9.7)]{NIST:DLMF}
\begin{gather}
C_n^{\mu-1}(\cos\gamma) =\frac{\mu -1}{n+\mu -1}\big( C_n^\mu(\cos\gamma)-C_{n-2}^{\mu}(\cos\gamma) \big) =\mathcal{O}\big( {n^{{\Re}\mu -2}}\big), \label{Dun2eq16a}
\end{gather} for ${\Re}\mu >0$. Hence (\ref{Dun2propeq}) holds for ${\Re}\mu >-1$, and on repeated applications of~(\ref{Dun2eq16a}), for ${\Re}\mu >-m$ for arbitrary bounded positive integer~$m$.
\end{proof}

\begin{thm}\label{tantanthm}The series in \eqref{conjPmPm} is absolutely convergent under the condition \eqref{Dun2eq21}. The series in \eqref{conjPmPp}, \eqref{conjPmQp}, \eqref{conjPmQm}--\eqref{conjQmPmmx} are absolutely convergent if \eqref{Dun2eq21} holds, and in addition $\theta \neq \theta ^{\prime}$.
\end{thm}

\begin{proof}Using Proposition \ref{Dun2prop}, we see that absolute convergence of (\ref{conjPmPp}), (\ref{conjPmQp}), (\ref{conjPmQm}) is assured if the right-hand sides of (\ref{Dun2eq3}), (\ref{Dun2eq6})--(\ref{Dun2eq10}) are exponentially small in~$n$. For (\ref{conjPmPm}), we see from (\ref{Dun2eq3}) that this is so if (\ref{Dun2eq21}) holds. Similarly we see that (\ref{Dun2eq6})--(\ref{Dun2eq10}) are exponentially small in $n$ provided \textit{both} of the following hold (unless one of the coefficient trigonometric functions vanish, in which case one of these can be relaxed, as described below)
\begin{gather}
\tan \bigl( {{\tfrac{1}{2}}\theta _{<}}\bigr) \tan \bigl( {{\tfrac{1}{2}}\theta _{>}}\bigr) <1\qquad \text{and} \qquad \tan \bigl( {{\tfrac{1}{2}}\theta _{<}}\bigr) \cot \bigl( {{\tfrac{1}{2}}\theta _{>}}\bigr) <1. \label{Dun2eq22}
\end{gather}
But the latter of (\ref{Dun2eq22}) is equivalent to $\tan \bigl( {{\frac{1}{2}}\theta _{<}}\bigr) <\tan \bigl( {{\frac{1}{2}}\theta _{>}}\bigr)$, i.e., $\theta _{<}<\theta _{>}$, or equivalently $\theta\neq \theta'$. Thus
convergence of (\ref{conjPmPp}), (\ref{conjPmQp}), (\ref{conjPmQm}) is also assured by condition (\ref{Dun2eq21}) provided $\theta\neq \theta'$. Absolute convergence of the expansions (\ref{conjPmPmmx}), (\ref{conjQmPmmx}) is proved similarly.
\end{proof}

\begin{rem} The condition $\theta _{<}\neq \theta _{>}$ or equivalently $\theta\ne\theta'$, can be relaxed for (\ref{conjPmPp}) if $\mu$ is an integer, and for (\ref{conjPmQp}) if $\mu$ is half an odd integer. Interestingly, if $\nu$ is an integer, then from (\ref{Dun2eq6}), (\ref{Dun2eq8}) we see that (\ref{conjPmPp}), (\ref{conjPmQp}) converge without requiring (\ref{Dun2eq21}), with the only requirement in this case being $\theta\ne\theta'$.
\end{rem}

Assume (\ref{Dun2eq21}) and define
\begin{gather}
\mathsf{A}( {z;\lambda ,\theta ,{\theta}^{\prime}}):=\sum\limits_{n=0}^{\infty}{(-1) ^{n}( {n+\lambda}) ( {2\lambda +z+1}) _{n}( {-z}) _{n}}\nonumber\\
\hphantom{\mathsf{A}( {z;\lambda ,\theta ,{\theta}^{\prime}}):=}{}\times \mathsf{P}_{\lambda +z}^{-( {\lambda +n})}( {\cos \theta})
\mathsf{P}_{\lambda +z}^{-( {\lambda +n})}( {\cos {\theta}^{\prime}}) C_n^\lambda(\cos\gamma) .\label{Dun3eq1}
\end{gather}%
Note that for all $n\in\N_0$,
\begin{gather*}
( {2\lambda +z+1}) _{n}( {-z}) _{n}=\frac{\Gamma( {n+2\lambda +z+1}) \Gamma ( {n-z})}{\Gamma ( {2\lambda +z+1}) \Gamma ( {-z})}=(-1) ^{n}\frac{\Gamma ( {n+2\lambda +z+1}) \Gamma ( {z+1})}{\Gamma ( {2\lambda +z+1}) \Gamma ( {z+1-n})},
\end{gather*}%
and is analytic for all $z\in\C$.

\begin{lem}\label{lintlemma}Let $l\in {\mathbf{N}}_{0}$. Then
\begin{gather}
\mathsf{A}( {l;\lambda ,\theta ,{\theta}^{\prime}})=\sum\limits_{n=0}^{l}{(-1) ^{n}( {n+\lambda})( {2\lambda +l+1}) _{n}( {-l}) _{n}}\mathsf{P}_{\lambda +l}^{-( {\lambda +n})}( {\cos \theta})
\mathsf{P}_{\lambda +l}^{-( {\lambda +n})}( {\cos {\theta}^{\prime}}) C_n^\lambda(\cos\gamma) \notag \\
\hphantom{\mathsf{A}( {l;\lambda ,\theta ,{\theta}^{\prime}})}{} =\frac{1}{2^\lambda\Gamma ( \lambda )}\left(\frac{{\sin \theta \sin {\theta}^{\prime}}}{\sin \Theta}\right) ^{\lambda
}\mathsf{P}_{\lambda +l}^{-\lambda}( {\cos \Theta}) .\label{Dun3eq3}
\end{gather}
\end{lem}

\begin{proof} The first equality of (\ref{Dun3eq3}) comes from (\ref{Dun3eq1}) and noting that $(-l)_{n}=0$ for $n\geq l+1$. Next, use the addition theorem for Gegenbauer polynomials \cite[equation~(18.18.8)]{NIST:DLMF} with degree~$l$, which is a finite sum over $n=0,1,2,\ldots ,l$. The Gegenbauer polynomial on the left-hand side and the first two Gegenbauer polynomials appearing in the sum are expressed in terms of Ferrers functions of the first kind using~(\ref{Dun2eq12}). Performing these replacements and simplifying produces the second equality of~(\ref{Dun3eq3}).
\end{proof}

In what follows we shall make use the following uniqueness theorem of complex analysis, due to Carlson (cf.~\cite[Section 5.81]{Titchmarsh1939}).

\begin{thm} \label{Carlsonthm} Let $\kappa <\pi$, $\mathsf{F}\colon {\mathbf{C}}\rightarrow {\mathbf{C}}$ be analytic and $\mathcal{O}\big({\rm e}^{\kappa |z|}\big)$ for $\Re z\geq 0$, and in addition $\mathsf{F}(z)=0$ for $z\in {\mathbf{N}}_{0}$. Then $\mathsf{F}(z)=0$.
\end{thm}

\begin{rem}The order restriction as stated is only required for $\Re z = 0$, and can be relaxed to $\mathsf{F}(z)$ being of (arbitrary) exponential type for $\Re z > 0$; however the above weaker version of Carlson's
theorem suffices for our purposes.
\end{rem}

In order to apply the above theorem, we need estimates on $\mathsf{A}( {z;\lambda ,\theta ,\theta^{\prime}})$ for large complex $z$. To do so, we need to temporarily assume that ${\theta}$, ${{\theta}^{\prime}}$ are
complex. To this end, we again use \cite[equation~(14.12.1)]{NIST:DLMF} to obtain the integral representation for ${\theta}\in ( 0,\pi )$, ${\Re}\mu >-\frac{1}{2}$,
\begin{gather}
\mathsf{P}_{\lambda +z}^{-\mu}( {\cos \theta}) =\left( {\frac{2}{\pi}}\right) ^{\frac12}\frac{1}{\Gamma \big( {\mu + {\frac{1}{2}}}\big)}%
\int_{0}^{\theta}{\frac{\cos \big( {\big( {\lambda +z+{\frac{1}{2}}}\big) t}\big)}{ ( {\cos t-\cos \theta} ) ^{\frac12}}\left( {\frac{\cos t-\cos \theta}{\sin \theta}}\right) ^{\mu}\mathrm{d}t}. \label{PFcomplex}
\end{gather}
Taking principal branches, we then use this to (uniquely) define $\mathsf{P}_{\lambda +z}^{-\mu} ({\cos \theta} )$ as an analytic function of $\theta \in \mathbf{C}$ with ${\Re}\theta \in ( 0,\pi )$. Then if we choose the path of integration in~(\ref{PFcomplex}) to be a~straight line from $t=0$ to $t=\theta$, parameterize $t=s\theta $ ($0\leq s\leq 1$), and use a~Maclaurin series for trigonometric functions, we have%
\begin{gather}
\frac{\cos t-\cos \theta}{\sin \theta}=\frac{1}{2}\big( 1-s^{2}\big) \theta +{\mathcal{O}}\big( \theta ^{3}\big) . \label{est1}
\end{gather}
In addition, for $t$ lying on this path we have
\begin{gather}
\cos \big( {\big( {\lambda +z+{\tfrac{1}{2}}}\big) t}\big) =\cos \big( {\big( {\lambda +z+{\tfrac{1}{2}}}\big)}s\theta \big) ={\mathcal{O}}\big( {\rm e}^{ \vert {\theta z} \vert}\big) . \label{est2}
\end{gather}
From (\ref{PFcomplex}), (\ref{est1}), (\ref{est2}) it follows that for complex values of $\theta$ such that ${\Re}\theta >0$, ${\vert \theta \vert <}\frac{1}{2}\pi -1$ (a bound clarified below), there is an assignable positive number $a_{0}$ such that{\samepage
\begin{gather}
\mathsf{P}_{\lambda +z}^{-\mu} ( {\cos \theta} ) =\frac{( {a}_{0}{ \vert \theta \vert} ) ^{{\Re}\mu}{\rm e}^{\vert {\theta z}\vert}}{\Gamma \big( {\mu +{\frac{1}{2}}}\big)}{\mathcal{O}}(1) , \label{Dun3eq5}
\end{gather}
uniformly for one or both $ \vert z \vert$ and ${\Re}\mu$ large.}

Next from \cite[equations~(4.5.13) and (5.6.1)]{NIST:DLMF} and the definition of the Pochhammer symbol, we have for unbounded positive $n$, and $w\in \mathbf{C}$,
\begin{gather*}
\frac{\bigl\vert {(w) _{n}}\bigr\vert}{\Gamma ( n)}\leq \frac{n^{\frac12}{\rm e}^{n}\vert w\vert ( {\vert w \vert +1} ) ( { \vert w \vert +2} ) \cdots
( { \vert w \vert +n} )}{( {2\pi})^{\frac12}n^{n}}\leq \frac{n^{\frac12}{\rm e}^{n}}{( {2\pi}) ^{\frac12}}\left( {1+\frac{ \vert w \vert}{n}}\right) ^{n}\leq \frac{n^{\frac12}{\rm e}^{n}{\rm e}^{\vert w\vert}}{( {2\pi}) ^{\frac12}},
\end{gather*}
and therefore from (\ref{Dun3eq5}) and Stirling's formula,
\begin{gather*}
( {n+\lambda} ) ( {2\lambda +z+1} ) _{n} ( {-z}) _{n}\mathsf{P}_{\lambda +z}^{-( {\lambda +n})}( {\cos \theta}) \mathsf{P}_{\lambda +z}^{-( {\lambda +n})}( {\cos {\theta}^{\prime}}) \\
=\frac{\Gamma^2 ( n)}{\Gamma^2 \big( {n+\lambda +{\frac{1}{2}}}\big)} \big( {a}_{0}^{2}{{\rm e}^{2}\vert {\theta {\theta}^{\prime}}\vert}\big) ^{n}{\rm e}^{( {2+\vert \theta\vert +\vert {{\theta}^{\prime}}\vert}) \vert
z\vert}{\mathcal{O}}\big( {n^{2}}\big) =\big( {a}_{0}^{2}{{\rm e}^{2} \vert {\theta {\theta}^{\prime}} \vert}\big) ^{n}{\rm e}^{( {2+\vert \theta \vert +\vert {{\theta}^{\prime}}\vert}) \vert z\vert}{\mathcal{O}}\big({n^{1-2{\Re}\lambda}}\big) .
\end{gather*}
Using (\ref{Dun2propeq}), along with the temporary assumption that $\vert \theta \vert ,\vert {{\theta}^{\prime}}\vert \leq \theta _{0}<\frac{1}{2}\pi -1$ (for some positive~$\theta _{0}$, to be specified shortly) we arrive at
\begin{gather}
( {n+\lambda}) ( {2\lambda +z+1}) _{n}( {-z}) _{n}\mathsf{P}_{\lambda +z}^{-( {\lambda +n})} ( {\cos \theta}) \mathsf{P}_{\lambda +z}^{- ( {\lambda +n} )
} ( {\cos {\theta}^{\prime}} )C_n^\lambda ( {\cos\gamma} )\nonumber\\
\qquad{} =\delta ^{n}{\rm e}^{\kappa \vert z \vert}{\mathcal{O}}\big( {n^{-{\Re}\lambda}}\big) , \label{Dun3eq8}
\end{gather}
where $\kappa :=2+\vert \theta \vert +\vert {{\theta} ^{\prime}}\vert \leq 2+2\theta _{0}<\pi$, and $\delta :={a}_{0}^{2}{\rm e}^{2}\vert {\theta {\theta}^{\prime}}\vert \leq {a} _{0}^{2}{\rm e}^{2}\theta _{0}^{2}$. We wish that $\delta <1$, and this is achieved by choosing $\theta_{0}$ so that $\theta_0 \in \big(0,\min\big\{ \frac{1}{2}\pi -1,{a}_{0}^{-1}{\rm e}^{-1}\big\}\big)$.

Next, from (\ref{cosTheta}), we observe that $\Theta \rightarrow 0$ as $ \vert \theta \vert + \vert {{\theta}^{\prime}} \vert \rightarrow 0$, and therefore there exists a~posi\-ti\-ve~$\theta_{1}$ such that $ \vert \Theta \vert \leq \kappa < \pi$ when $\max \{ |\theta|, |\theta'| \} \leq \theta_1$. Now define $\theta _{2}:=\min \{ \theta _{0},\theta _{1}\}$, and also define
\begin{gather}
\mathsf{F}( {z;\lambda ,\theta ,{\theta}^{\prime}}) :=\frac{1}{2^\lambda\Gamma ( \lambda )}\left( \frac{{\sin \theta \sin {
\theta}^{\prime}}}{\sin \Theta}\right) ^\lambda\mathsf{P}_{\lambda+z}^{-\lambda}( {\cos \Theta}) -\mathsf{A}( {z;\lambda,\theta ,{\theta}^{\prime}}) . \label{Dun3eq9}
\end{gather}%
Then from (\ref{Dun3eq1}), (\ref{Dun3eq5}), (\ref{Dun3eq8}), we see that for $| \theta | ,| {{\theta}^{\prime}}|\in(0,\theta_2)$, ${\Re}z\geq 0$, the function $ \mathsf{F} ( {z;\lambda ,\theta ,{\theta}^{\prime}} )$ is well-defined and analytic as a function of~$z$, and is ${\mathcal{O}}\big( {{\rm e}^{\kappa \vert z \vert}} \big)$ as $z\rightarrow \infty$, where $\kappa <\pi$. Moreover from (\ref{Dun3eq3}), (\ref{Dun3eq9}), we see that $\mathsf{F} ( {z;\lambda ,\theta ,{\theta}^{\prime}} ) =0$ for $z\in {\mathbf{N}}_{0}$, and therefore by Carlson's Theorem \ref{Carlsonthm} we deduce that $ \mathsf{F} ( {z;\lambda,\theta ,{\theta}^{\prime}} )$ is identically zero.

The restriction $0< \vert \theta \vert , \vert {{\theta}^{\prime}} \vert <\theta _{2}$ can be relaxed as follows. For fixed~$z$, $\lambda$, the function given by~(\ref{Dun3eq9}) is also analytic in $\theta$ in a neighborhood of the origin with ${\Re}\theta >0$, and likewise for ${\theta}^{\prime}$. Hence by analytic continuation, we deduce that it must be identically zero in larger domains (open and connected) containing this neighborhood, and for which $\mathsf{P}_{\lambda +z}^{-\lambda} ( {\cos \Theta} )$ is analytic in both of these variables, and for which the series in (\ref{Dun3eq1}) is absolutely convergent. In particular, it follows that $ \mathsf{F} ( {z;\lambda ,\theta ,{\theta}^{\prime}} ) =0$ for positive values of these variables such that \mbox{$\tan \bigl( {{\frac{1}{2}}\theta}\bigr) \tan \bigl( {{\frac{1}{2}}{\theta}^{\prime}}\bigr) <1$}. The identity (\ref{conjPmPm}) follows by replacing $z\mapsto \nu -\lambda$, $\lambda \mapsto \mu$, in (\ref{Dun3eq9}) and equating this function to zero.

It remains to prove (\ref{conjPmPp}), (\ref{conjPmQp}), (\ref{conjPmQm})--(\ref{conjQmPmmx}). These come from (\ref{conjPmPm}) and connection formulas. In (\ref{conjPmPm}), apply the connection relation (\ref{Pnummuconnection}) to the Ferrers function of the first kind on the left-hand side and the one furthest to the right on the right-hand side. This produces two separate sums which are given by (\ref{conjPmPp}), (\ref{conjPmQp}). In (\ref{conjPmQp}), applying the connection relation \cite[equation~(14.9.1)]{NIST:DLMF} to the left-hand side, and to the Ferrers function of the second kind furthest on the right, produces yet two more sums, one being (\ref{conjPmQm}). The addition theorems (\ref{conjPmPmmx}), (\ref{conjQmPmmx}), then follow using (\ref{Pnummumxconnection}), (\ref{Qnummumxconnection}) in (\ref{conjPmPm}), (\ref{conjPmQm}). This completes the proof of Theorem~\ref{GegexpansionsFerrers}. \begin{rem}In terms of the validity of the series expansions, it should be noted that the Ferrers function of the second kind $\mathsf{Q}_{\nu}^{\mu}(x)$ is not defined for $\nu+\mu =-1,-2,\ldots$, except for the anomalous cases $\nu =-\frac{3}{2},-\frac{5}{2},\ldots$, see~(\ref{defFerrersQhypergeo}).
\end{rem}

\subsection*{Acknowledgements}
We thank George Pogosyan and Loyal Durand for valuable discussions, and the referees for many helpful suggestions. T.M.D.~acknowledges support from Ministerio de Econom\'{i}a y~Competitividad, Spain, project MTM2015-67142-P (MINECO/FEDER, UE).

\pdfbookmark[1]{References}{ref}
\LastPageEnding

\end{document}